\documentclass[11pt,english]{article}
\usepackage{lmodern}
\usepackage{lmodern}
\usepackage[T1]{fontenc}
\usepackage[latin9]{inputenc}
\usepackage{babel}
\usepackage{prettyref}
\usepackage{float}
\usepackage{mathrsfs}
\usepackage{mathtools}
\usepackage{enumitem}
\usepackage{amsmath}
\usepackage{amsthm}
\usepackage{amssymb}
\usepackage{stmaryrd}
\usepackage{graphicx}
\usepackage{setspace}
\usepackage[unicode=true,pdfusetitle,
 bookmarks=true,bookmarksnumbered=true,bookmarksopen=false,
 breaklinks=false,pdfborder={0 0 1},backref=false,colorlinks=false]
 {hyperref}
\hypersetup{
 colorlinks,linkcolor={blue},citecolor={crimson},urlcolor={orange}}

\makeatletter
  \theoremstyle{plain}
  \newtheorem{prop}{\protect\propositionname}[section]
  \theoremstyle{remark}
  \newtheorem*{rem*}{\protect\remarkname}
  \theoremstyle{plain}
  \newtheorem{lem}{\protect\lemmaname}[section]
  \theoremstyle{plain}
  \newtheorem{thm}{\protect\theoremname}[section]
  \theoremstyle{plain}
  \newtheorem{cor}{\protect\corollaryname}[section]

\usepackage{babel}

\usepackage{dsfont}
\usepackage{bbm}

\newcommand{\ones}{\mathbbm{1}}

\usepackage{xcolor}

\definecolor{crimson}{HTML}{DC143C}

\definecolor{huntergreen}{HTML}{355E3B}

\definecolor{emerald}{HTML}{50C878}

\newrefformat{prop}{Proposition \ref{#1}}
\newrefformat{figure}{Figure \ref{#1}}

\newrefformat{subsec}{Section \ref{#1}}

\newrefformat{proof}{Proof in \ref{#1}}

\usepackage{babel}

\usepackage[totalwidth=460pt, totalheight=660pt]{geometry}

\usepackage{tikz}
\usetikzlibrary{arrows,calc}
\tikzset{
>=stealth',
help lines/.style={dashed, thick},
axis/.style={<->},
important line/.style={thick},
connection/.style={thick, dotted},
}

\usepackage{pgfplots}

\pgfmathdeclarefunction{gauss}{2}{%
  \pgfmathparse{1/(#2*sqrt(2*pi))*exp(-((x-#1)^2)/(2*#2^2))}%
}

\usepackage{caption}
\usepackage{subcaption}

\makeatother

  \providecommand{\lemmaname}{Lemma}
  \providecommand{\propositionname}{Proposition}
  \providecommand{\remarkname}{Remark}
\providecommand{\corollaryname}{Corollary}
\providecommand{\theoremname}{Theorem}

\usepackage[blocks]{authblk}

\date{}

\begin{document}

\title{\textbf{On the Convergence of the EM Algorithm:}\\
\textbf{A Data-Adaptive Analysis}}

\author[1]{\href{mailto:cwu@live.com}{\textcolor{black}{Chong Wu}}}
\author[1]{\href{mailto:eeyang@hkbu.edu.hk}{\textcolor{black}{Can Yang}}}
\author[2]{\href{mailto:hongyu.zhao@yale.edu}{\textcolor{black}{Hongyu Zhao}}}
\author[3]{\href{mailto:jizhu@umich.edu}{\textcolor{black}{Ji Zhu}}}

\affil[1]{Department of Mathematics\\Hong Kong Baptist University}
\affil[2]{Department of Biostatistics\\Yale School of Public Health\\Yale University}
\affil[3]{Department of Statistics\\University of Michigan}

\affil[ ]{ }

\affil[ ]{\href{mailto:chongwu@hkbu.edu.hk}{\textcolor{black}{\texttt{\{chongwu,eeyang\}@hkbu.edu.hk}}}}
\affil[ ]{\href{mailto:hongyu.zhao@yale}{\textcolor{black}{\texttt{hongyu.zhao@yale.edu}}}}
\affil[ ]{\href{mailto:jizhu@umich.edu}{\textcolor{black}{\texttt{jizhu@umich.edu}}}}

\renewcommand\Authands{ and }

\maketitle
\vspace{-35pt}

\begin{abstract}
The Expectation-Maximization (EM) algorithm is an iterative method
to maximize the log-likelihood function for parameter estimation.
Previous works on the convergence analysis of the EM algorithm have
established results on the asymptotic (population level) convergence
rate of the algorithm. In this paper, we give a data-adaptive analysis
of the sample level local convergence rate of the EM algorithm. In
particular, we show that \emph{the local convergence rate of the EM
algorithm is a random variable} $\overline{K}_{n}$ derived from the
data generating distribution, which adaptively yields the convergence
rate of the EM algorithm on each finite sample data set from the same
population distribution. We then give a non-asymptotic concentration
bound of $\overline{K}_{n}$ on the population level optimal convergence
rate $\overline{\kappa}$ of the EM algorithm, which implies that
$\overline{K}_{n}\to\overline{\kappa}$ in probability as the sample
size $n\to\infty$. Our theory identifies the effect of sample size
on the convergence behavior of sample EM sequence, and explains a
surprising phenomenon in applications of the EM algorithm, i.e. the
finite sample version of the algorithm sometimes converges faster
even than the population version. We apply our theory to the EM algorithm
on three canonical models and obtain specific forms of the adaptive
convergence theorem for each model.
\end{abstract}

\section{Introduction}

The iterative algorithm of expectation-maximization (EM) has been
proposed in various special forms by a number of authors as early
as in the 1970s, notably \cite{Baum,Orchard,Sundberg-1,Rubin,Sundberg-2,Sundberg-3}.
Since the advent of its modern formulation by Dempster, Laird and
Rubin \cite{Dempster}, the EM algorithm has received much attention
in the statistical community. A vast literature on theoretical properties
and real applications of the EM algorithm has been accumulated thereafter
(see e.g. \cite{Dempster,Boyles,Wu,Redner,Meng-1,Meng-2}). Classical
work of Wu \cite{Wu} established general convergence results for
EM sequences to the MLE or some stationary points of the log-likelihood
function; Redner and Walker \cite{Redner} proved asymptotic results
on the convergence of the EM algorithm for mixture of densities from
the exponential family; Meng and Rubin \cite{Meng-1} analyzed both
asymptotic componentwise and global convergence rates of the EM algorithm;
some variants or generalizations of the EM algorithm were also proposed:
Meng and Rubin \cite{Meng-0} developed ECM algorithm to replace a
complicated $M$-step by several simpler $CM$-steps (conditional
maximization); Liu et al. \cite{Liu} proposed PX-EM to use the expanded
complete-data model to accelerate the convergence of the EM algorithm.
The book of McLachlan and Krishnan \cite{McLachlan} gave a comprehensive
account on both theoretical and practical aspects of the EM algorithm.

Recent work of Balakrishnan et al. \cite{Balakrishnan} presented
statistical guarantees for the local linear convergence of the EM
algorithm and first-order EM algorithm to the true population parameter
$\theta^{*}$ within statistical precision. Along this line, Wang
et al. \cite{Wang} considered extensions to high-dimensional settings
by introducing a truncation step; Yi and Caramanis \cite{Yi} proved
statistical guarantees for generalizations to regularized EM algorithms
in high-dimensional latent variable models. In this paper, we give
a data-adaptive analysis of the finite sample level convergence behavior
of the EM algorithm, especially the dynamics of the convergence rate
when the EM algorithm is performed on multiple finite random data
sets (with possibly different sample sizes) sampled from the same
population distribution.

\subsection{Problem Setup}

Suppose $\{\mathbb{P}_{\theta}\mid\theta\in\Omega\subseteq\mathbb{R}^{p}\}$
is a family of parametric distributions, and $\mathbb{P}_{\theta}$
has density function $p_{\theta}(y)$ with respect to the Lebesgue
measure on $\mathbb{R}^{d}$. A set of i.i.d. samples $\{y_{k}\}_{k=1}^{n}$
of $Y\sim\mathbb{P}_{\theta^{*}}$ is observed, where $\theta^{*}\in\Omega$
is an unknown population true parameter.

In latent variable models, $Y$ is the observed part of a pair $(Y,Z)$
of random variables and $Z$ is a latent variable. Suppose $f_{\theta}(y,z)$
is the joint density of $(Y,Z)$ and for $\theta\in\Omega$, the density
$p_{\theta}(y)=\int_{\mathcal{Z}}f_{\theta}(y,z)dz$ is the marginalization
of $f_{\theta}(y,z)$ over $z$, then the EM algorithm can be applied
to estimate $\theta^{*}$ from the samples $\{y_{k}\}_{k=1}^{n}$
of $Y\sim\mathbb{P}_{\theta^{*}}$.

Specifically, one first calculates the sample $Q$-function (see Definition
1 in \cite{Balakrishnan}) by a conditional \emph{Expectation} ($E$-step):
\[
Q_{n}(\theta'|\theta;\{y_{k}\})=\frac{1}{n}\sum_{k=1}^{n}\int_{\mathcal{Z}(y_{k})}\log\left(f_{\theta'}(y_{k},z)\right)k_{\theta}(z|y_{k})dz,
\]
where $k_{\theta}(z|y)\coloneqq\frac{f_{\theta}(y,z)}{p_{\theta}(y)}$
is the conditional density of $Z$ given $Y$. Then for an initial
point $\theta_{n}^{0}\in B_{r}(\theta^{*})$, the sample EM sequence $\{\theta_{n}^{t}\}_{t\ge0}$
is constructed by \emph{Maximization} ($M$-step):
\[
\theta_{n}^{t+1}\in\arg\max\{Q_{n}(\theta'|\theta_{n}^{t};\{y_{k}\})\mid\theta'\in\Omega\},
\]
and we refer to this procedure as \emph{the EM algorithm is performed
on the samples $\{y_{k}\}_{k=1}^{n}$.}

We notice that the sample $Q$-function $Q_{n}(\theta'|\theta;\{y_{k}\})$
depends on a specific set of samples $\{y_{k}\}_{k=1}^{n}$, hence
so does the sample EM sequence $\{\theta_{n}^{t}\}_{t\ge0}$ defined above.
Since the samples are i.i.d. realizations of $Y\sim\mathbb{P}_{\theta^{*}}$,
it is sensible to conjecture that the convergence rate of the sample
EM sequence depends on the data generating distribution $\mathbb{P}_{\theta^{*}}$.
When the EM algorithm is performed on \emph{different} sets of samples
from the \emph{same} population, the corresponding sample EM sequences
constructed as in the above procedure ought to converge at different rates.
In the subsequent numerical experiments, we have also confirmed this
phenomenon, e.g. see \prettyref{figure:simu}. This observation motivates
us to characterize the convergence rate of the EM algorithm as a data-adaptive
quantity.

\begin{figure}[ht]
\centering \vspace{-60pt}
 \begin{subfigure}{.5\textwidth} \centering \includegraphics[width=1\linewidth]{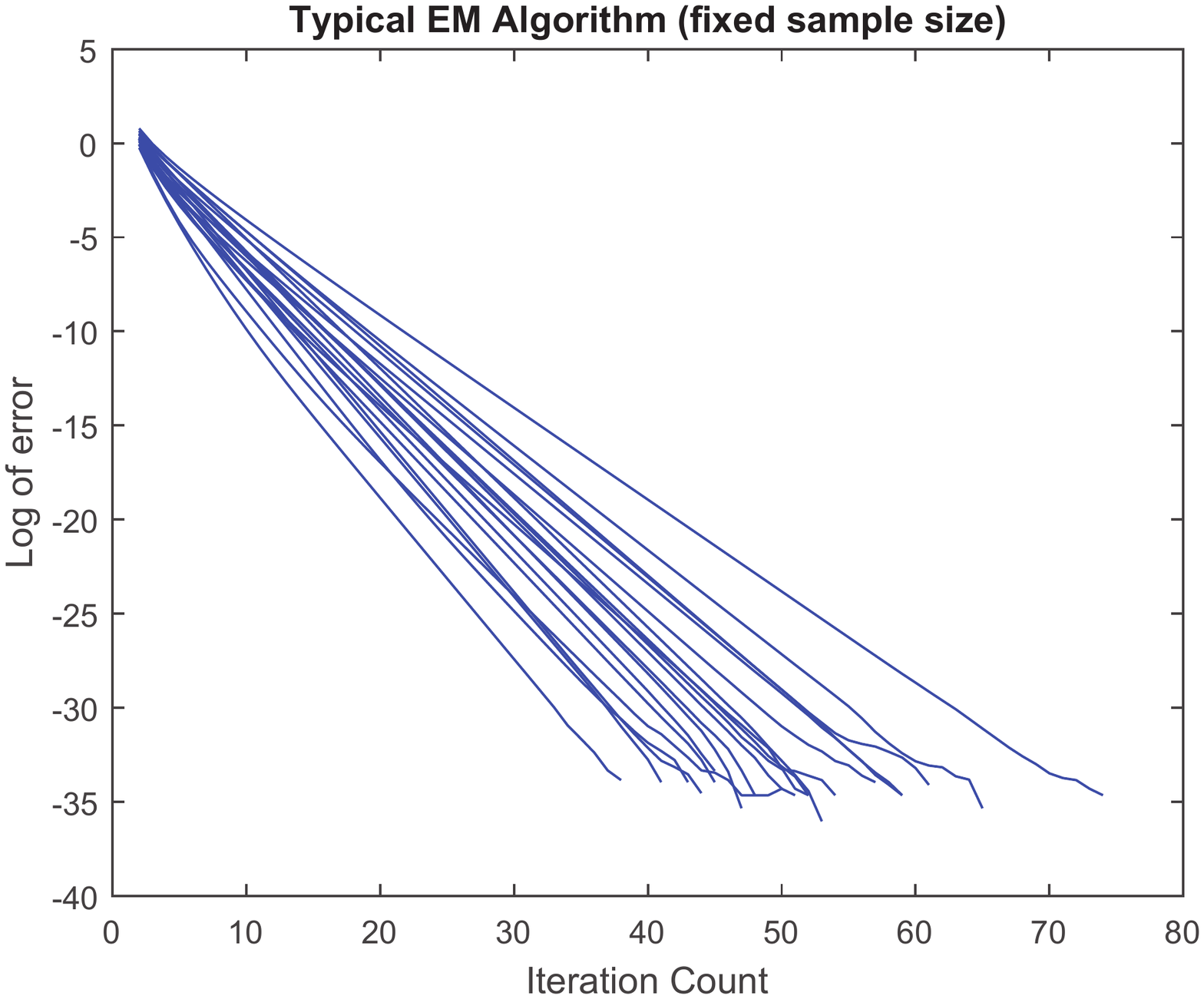}
\vspace{-70pt}
 \caption{}
\label{figure:em-fix-sam-size} \end{subfigure}\begin{subfigure}{.5\textwidth}
\centering \includegraphics[width=1\linewidth]{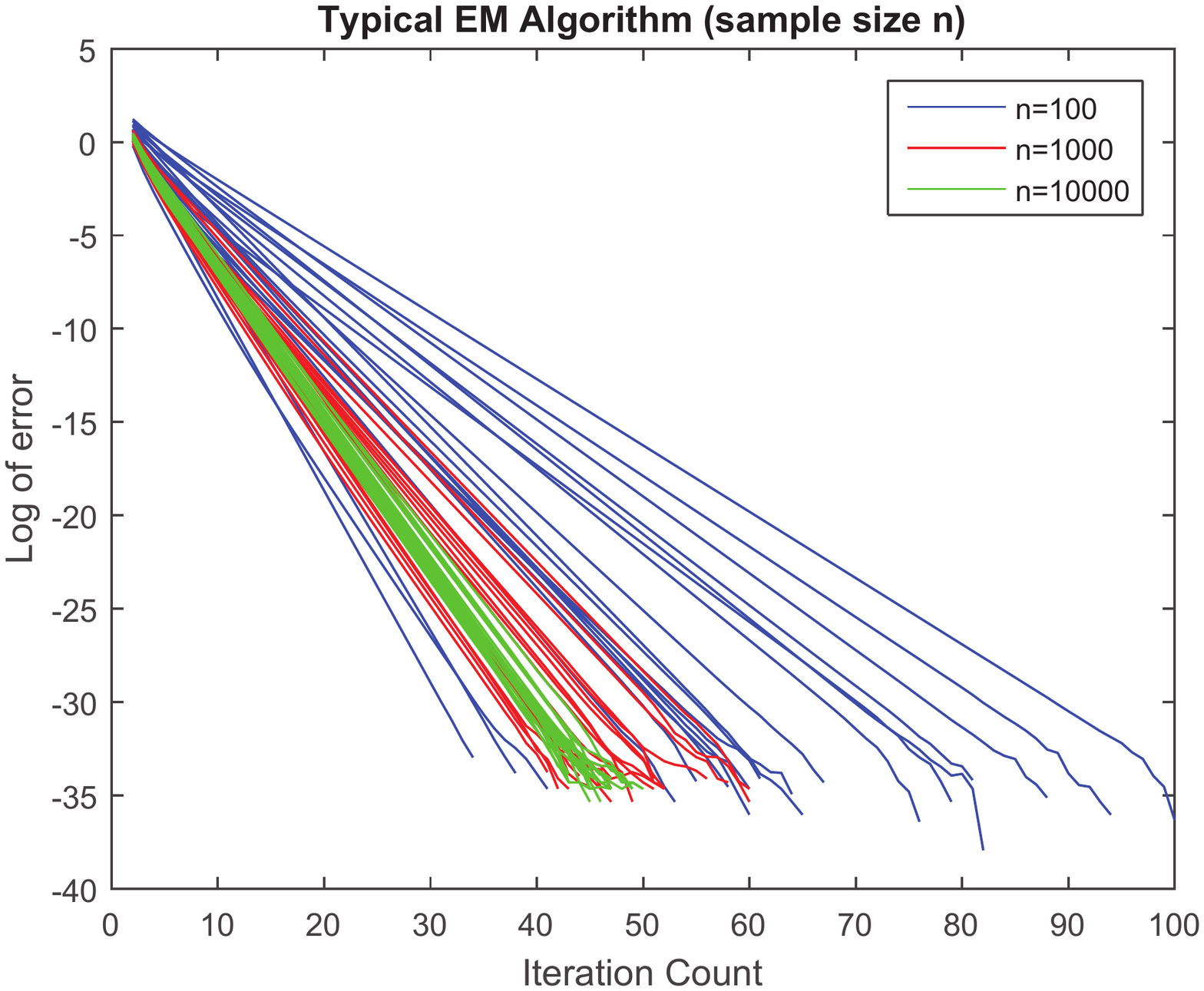}
\vspace{-70pt}
 \caption{}
\label{figure:em-var-sam-size} \end{subfigure} \caption{{\small{}These plots are generated by applying the EM algorithm to
simulated data from a }\emph{\small{}fixed}{\small{} Gaussian Mixture
Model (\prettyref{subsec:gau-mix}) with the dimension of $\theta^{*}$
set to $p=5$, and the SNR (Signal-to-Noise Ratio) $\frac{\left\Vert \theta^{*}\right\Vert }{\sigma}=1$.
Each line in the plots represents an instance of the EM algorithm
performed on a different data set sampled from the Gaussian mixture
distribution; the convergence rates of the EM algorithm are the slopes
of these lines. }\textbf{\small{}(a)}{\small{} 20 instances of the
EM algorithm each performed on a different set of $n=300$ random
samples; the slopes of the lines vary from one instance to another.
}\textbf{\small{}(b)}{\small{} 20 instances of the EM algorithm performed
on different sets of $n$ random samples for each $n\in\{100,1000,10000\}$;
when the sample size is small, the lines are more spread-out (blue
lines), hence the fluctuations in the convergence rate are large;
when the sample size is large, the lines are more clustered (red lines and green
lines), hence the fluctuations in the convergence rate are small.}}
\label{figure:simu}
\end{figure}

\subsection{Main Results and Contributions}

The main results of this paper are as follows: we characterize the
convergence rate of the \emph{empirical EM sequence} as a derived
random variable $\overline{K}_{n}$ of the data generating distribution
$\mathbb{P}_{\theta^{*}}$ in \prettyref{thm:oect}, then we give
the concentration bound of $\overline{K}_{n}$ in \prettyref{thm:orct}.

\paragraph{\emph{Optimal Empirical Convergence Theorem} }

The primary goal of \prettyref{thm:oect} is to show that \emph{the
convergence rate of the EM algorithm is a random variable}. To this
end, we adopt a novel data-adaptive viewpoint in the finite sample
level analysis by considering the samples as i.i.d. copies $\{Y_{k}\}_{k=1}^{n}$
of the random variable $Y\sim\mathbb{P}_{\theta^{*}}$ and exploiting
the concentration of measure phenomenon to obtain non-asymptotic sample
level convergence results.

The theorem states that if the EM algorithm is initialized as $\Theta_{n}^{0}\in B_{r}(\theta^{*})$
in the ball of population contraction (to be defined precisely), then
with high probability, we have a convergence inequality in the form
\begin{equation}
\left\Vert \Theta_{n}^{t}-\theta^{*}\right\Vert \le\left(\overline{K}_{n}\right)^{t}\left\Vert \Theta_{n}^{0}-\theta^{*}\right\Vert +\frac{\overline{E}_{n}}{\overline{V}_{n}-\overline{\Gamma}_{n}},\label{eq:oect-intro}
\end{equation}
where $\{\Theta_{n}^{t}\}_{t\ge0}$ is the \emph{empirical EM sequence},
defined as
\[
\Theta_{n}^{t+1}\in\arg\max\{Q_{n}(\Theta'|\Theta_{n}^{t};\{Y_{k}\})\mid\Theta'\in B_{R}(\theta^{*})\},
\]
for a set of i.i.d. copies $\{Y_{k}\}_{k=1}^{n}$ of $Y\sim\mathbb{P}_{\theta^{*}}$.
The quantities $\overline{\varGamma}_{n}$, $\overline{V}_{n}$, $\overline{E}_{n}$
and $\overline{K}_{n}$ are measurable functions of $(Y_{1},\cdots,Y_{n})$,
hence are random variables derived from $\mathbb{P}_{\theta^{*}}$.
$\overline{K}_{n}$ is called the \emph{optimal empirical convergence
rate} (See \prettyref{subsec:def-oecr} for the definitions), which
holds the information of how the data generating distribution $\mathbb{P}_{\theta^{*}}$
``propagates'' the randomness in sample data to the convergence
rate of the empirical EM sequence.

This theorem characterizes the convergence behavior of sample EM sequence
\emph{adaptively}: Given a set of i.i.d. realizations (or samples)
$\{y_{k}\}_{k=1}^{n}$ of $Y\sim\mathbb{P}_{\theta^{*}}$, we have
corresponding realizations $g_{n}$, $v_{n}$, $e_{n}$ and $k_{n}$
of $\overline{\varGamma}_{n}$, $\overline{V}_{n}$, $\overline{E}_{n}$
and $\overline{K}_{n}$ respectively, and a realization of the convergence
inequality \prettyref{eq:oect-intro} as
\begin{equation}
\left\Vert \theta_{n}^{t}-\theta^{*}\right\Vert \le\left(k_{n}\right)^{t}\left\Vert \theta_{n}^{0}-\theta^{*}\right\Vert +\frac{e_{n}}{v_{n}-g_{n}},\label{eq:oect-rel-intro}
\end{equation}
where the sample EM sequence $\{\theta_{n}^{t}\}_{t\ge0}$, as a realization
of $\{\Theta_{n}^{t}\}_{t\ge0}$, is constructed as
\[
\theta_{n}^{t+1}\in\arg\max\{Q_{n}(\theta'|\theta_{n}^{t};\{y_{k}\})\mid\theta'\in B_{R}(\theta^{*})\}.
\]
Hence this particular realization $k_{n}$ of $\overline{K}_{n}$
gives the convergence rate of the corresponding sample EM sequence
$\{\theta_{n}^{t}\}_{t\ge0}$ constructed when the EM algorithm is
performed on the samples $\{y_{k}\}_{k=1}^{n}$. A different set of
i.i.d. samples $\{y_{k}'\}_{k=1}^{n'}$ gives rise to a different
sample EM sequence $\{\theta_{n'}^{'t}\}_{t\ge0}$, a different realization
$k_{n'}'$ of $\overline{K}_{n'}$ and a different realization of
\prettyref{eq:oect-intro} in a form similar to \prettyref{eq:oect-rel-intro}.
Thus given each sample data set, the random variable $\overline{K}_{n}$
adaptively yields the convergence rate of the corresponding sample
EM sequence, and \prettyref{thm:oect} is precisely the mathematical
substantiation of our claim that the convergence rate of the EM algorithm
is a random variable.

\paragraph{\emph{Optimal Rate Convergence Theorem}}

Given the data generating distribution $\mathbb{P}_{\theta^{*}}$,
it is in general difficult to calculate the distribution or density
function of the derived random variables $\overline{\varGamma}_{n}$,
$\overline{V}_{n}$, $\overline{E}_{n}$ or $\overline{K}_{n}$. Nonetheless,
in \prettyref{thm:orct} we give a non-asymptotic concentration bound
of $\overline{K}_{n}$ on the \emph{optimal oracle convergence rate
}$\overline{\kappa}$, which sheds some light on the stochastic behavior
of the derived random variable $\overline{K}_{n}$.

The theorem states that if the EM algorithm is initialized within
the ball of population contraction, the optimal empirical convergence
rate $\overline{K}_{n}$ satisfies
\begin{equation}
\left|\overline{K}_{n}-\overline{\kappa}\right|\le\frac{2}{\overline{\nu}}\left(\varepsilon_{1}(\delta,r,n,p)+\overline{\kappa}\varepsilon_{2}(\delta,r,R,n,p)\right)\label{eq:orct-intro}
\end{equation}
with probability at least $1-\delta$, where $\varepsilon_{1}(\delta,r,n,p)$
and $\varepsilon_{2}(\delta,r,R,n,p)$ are infinitesimals as $n\to\infty$.
It then follows that $\overline{K}_{n}\to\overline{\kappa}$ in probability
as $n\to\infty$. One of our contributions on the three canonical
models is the calculation of the infinitesimals $\varepsilon_{1}(\delta,r,n,p)$
and $\varepsilon_{2}(\delta,r,R,n,p)$ in closed forms and the concentration
bound of the random variable $\overline{K}_{n}$ for each model (see
\prettyref{sec:app-cla-mod}).

\paragraph{\emph{On the Convergence of the EM Algorithm}}

The data-adaptive analysis in our paper offers some new insights and
theoretical explanations to the convergence behavior of the EM algorithm.
\begin{enumerate}
\item The sample size does not \emph{directly} affect the convergence rate
of the EM algorithm. Indeed, as we observed in numerical experiments
and real applications, the EM algorithm performed on smaller sample
sets can converge faster than performed on larger sample sets, even
faster than the population (with infinite many samples) EM algorithm.
\prettyref{thm:oect} suggests a theoretical explanation to this phenomenon:
the sample EM sequence constructed from a finite sample data set $\{y_{k}\}_{k=1}^{n}$
converges at the rate $k_{n}$, which is a realization of the random
variable $\overline{K}_{n}$ given the sample data set. Since $\overline{K}_{n}$
randomly fluctuates around $\overline{\kappa}$, and in view of the
concentration bound \prettyref{eq:orct-intro}, it is possible that
the realization $k_{n}<\overline{\kappa}$. When this is the case,
the sample EM sequence exhibits a faster convergence rate than the
population EM sequence.

\emph{The convergence rate of the sample EM sequence randomly fluctuates
around the optimal population convergence rate, and it is not simply
proportional to the sample size.}
\item The convergence behavior displayed in \prettyref{figure:simu} is
ubiquitous in numerical experiments and real applications of the EM
algorithm. Our theory provides a cogent explanation to such phenomena.
For \prettyref{figure:em-fix-sam-size}, the EM algorithm is performed
on $20$ data sets with the same sample size $n=300$. By \prettyref{thm:oect},
the convergence rates of the sample EM sequences are $20$ realizations
of the random variable $\overline{K}_{n}$, one for each sample data
set. The randomness of the sampling process causes random fluctuations
among the $20$ realizations of $\overline{K}_{n}$, which accounts
for the variations of the slopes of these blue lines. For \prettyref{figure:em-var-sam-size},
the EM algorithm is performed on $20$ data sets for each sample size
$n\in\{100,1000,10000\}$. In view of the concentration bound \prettyref{eq:orct-intro}
in \prettyref{thm:orct}, when the sample size $n$ is large, the
right-hand side of \prettyref{eq:orct-intro} is small and the realizations
of $\overline{K}_{n}$ are more concentrated around $\overline{\kappa}$,
hence the convergence rate is \emph{stable} (i.e. the green lines
cluster together). Conversely, when the sample size $n$ is smaller,
the right-hand side of \prettyref{eq:orct-intro} is larger and the
realizations of $\overline{K}_{n}$ are more scattered, hence the
convergence rate is unstable (i.e. the red lines, and especially the
blue lines fan out).

\emph{The sample size regulates the stability of the convergence rate
of the sample EM sequence. The convergence rate of the EM algorithm
performed on larger sample sets is stabler than on smaller sample
sets.}
\item In low-dimensional regime where $p\ll n$, the convergence behavior
of the sample EM sequence is ``concentrated'' on the convergence
behavior of the \emph{corresponding} (initialized at the same point
$\theta_{n}^{0}$) population EM sequence. In particular, the ball of
contraction for the sample EM sequence is the \emph{same} as that for the
population EM sequence; and the convergence rate of the sample EM
sequence is also well approximated by the population convergence rate
with high probability. For concrete models, our theory gives quantitative
characterization of the low-dimensional regime with respect to approximation
error $\epsilon>0$ and tolerance $\delta>0$ as $\mathcal{R}_{\ell}(\epsilon,\delta)=\{(p,n)\mid\left|\overline{K}_{n}-\overline{\kappa}\right|<\epsilon\text{ with probability at least }1-\delta\}$.

\emph{The study of the convergence behavior of the EM algorithm in
low-dimensional regime can basically be reduced to the study of the
population EM sequence.}
\end{enumerate}

\subsection{Related Works}

Our work was inspired by an insightful Population-Sample based analysis
in \cite{Balakrishnan} and we built upon many classical works on
the EM algorithm. The major differences of our theory to previous
works are in the following respects:
\begin{enumerate}
\item We focus on the study of a different problem in the convergence analysis
of the EM algorithm. Previous works studied the convergence rate of
the EM algorithm on an arbitrary but \emph{fixed} sample data set.
We study the dynamics of the convergence rate when the EM algorithm
is performed on multiple data sets (with possibly different sample
sizes) from the same population distribution, and quantify the intrinsic
connection between the data generating distribution and the convergence
rate of the sample EM algorithm. The central objects in our analysis
are \emph{derived random variables }(defined in the sequel) from the
data generating distribution $\mathbb{P}_{\theta^{*}}$. As we shall
see, the population means of these random variables characterize the
convergence of the population EM sequence, while their empirical means
characterize the convergence of the empirical EM sequence.
\item Classical works on the EM algorithm (e.g. \cite{Dempster,Redner,Meng-1,Meng-2})
analyzed the convergence rate of the EM algorithm asymptotically.
Recent work of Balakrishnan et al. \cite{Balakrishnan} proved geometric
convergence results for sample EM algorithm when initialized within
the basin of contraction. They directly leveraged the $\kappa$-contractivity
of the population $M$-operator to obtain the sample level convergence
result, hence the convergence rate is essentially the population level
(asymptotic) rate. In this paper, we characterize the finite sample
level convergence rate of the EM algorithm as a random variable, which
adaptively yields the convergence rate of the EM algorithm for each
finite sample set.
\item From the technical aspects, the main tools in classical analysis
of the convergence of the EM algorithm are information matrices and
the rate matrix (i.e. Jacobian matrix of the $M$-operator). Balakrishnan
et al. \cite{Balakrishnan} exploited the KKT conditions which characterize
the optimality of $\theta^{*}$ and $M(\theta)$ to derive the $\kappa$-contractivity
of the population $M$-operator. In this paper, we do not follow the $M$-operator
approach in previous works \cite{Balakrishnan,Wang,Yi}. Instead,
we directly leverage the optimality of the EM sequence in each $M$-step
of the EM iteration for both population and sample (empirical) EM
sequences. This approach allows us to prove a basic contraction inequality
\prettyref{eq:ora-contr-ineq}, which can be viewed as a generalization
of the inequality in Theorem 4 of \cite{Balakrishnan}. Meanwhile,
this approach overcomes the difficulty of verifying conditions involving
$M$-operators, e.g. the First-Order Stability or the (uniform) deviation
bounds of sample $M$-operators to population $M$-operator etc. Another
technical difference is that, under natural concentration assumptions,
the quantity characterizing the statistical error in our theory is
guaranteed to converge to zero in probability as the sample size $n\to\infty$.
This observation allows us to avoid the difficulty in bounding an
empirical process of $M$-operators, and prove the statistical consistency
of the EM algorithm not only for specific models, but also at a general
theoretical level.
\end{enumerate}
\vspace{10pt}

The remainder of this paper is organized as follows. Following Notations
and Conventions, we briefly review the EM algorithm in \prettyref{sec:rev-em-alg}.
Then we formulate our convergence theory in two parts: \prettyref{subsec:ora-conv}
contains the theory of oracle convergence; \prettyref{subsec:emp-conv}
contains the theory of empirical convergence and the consistency of
the EM algorithm. In \prettyref{sec:app-cla-mod}, we apply our theory
to three canonical models: the Gaussian Mixture Model (\prettyref{subsec:gau-mix}),
the Mixture of Linear Regressions (\prettyref{subsec:mix-lin-reg});
and Linear Regression with Missing Covariates (\prettyref{subsec:lin-reg-mis-cov}).
We conclude the paper with Discussion (\prettyref{sec:disc}) and defer the detailed
proofs for the canonical models to the Appendix.
\newpage
\paragraph{Notations and Conventions}
\begin{itemize}
\item For $p\ge1$ and $x\in\mathbb{R}^{p}$, let $\left\Vert x\right\Vert =\left(\sum_{j=1}^{p}\left|x^{j}\right|^{2}\right)^{\frac{1}{2}}$
be the $L^{2}$-norm of $x$.
\item For $r>0$ and $\theta\in\mathbb{R}^{p}$, let $B_{r}(\theta)\coloneqq\left\{ x\in\mathbb{R}^{p}\mid\left\Vert x-\theta\right\Vert <r\right\} $
be the open ball; and $\overline{B}_{r}(\theta)\coloneqq\left\{ x\in\mathbb{R}^{p}\mid\left\Vert x-\theta\right\Vert \le r\right\} $
be the closed ball; and $B_{r}^{\times}(\theta)\coloneqq B_{r}(\theta)\backslash\{\theta\}$
be the punctured open ball; and $\mathbb{S}^{p-1}\coloneqq\left\{ x\in\mathbb{R}^{p}\mid\left\Vert x\right\Vert =1\right\} $
be the standard unit sphere in $\mathbb{R}^{p}$.
\item A random variable $Y$ is a real-valued Borel measurable function
on a probability measure space $(\mathscr{S},\mathscr{E},\mathbb{P})$
(see e.g. \cite{Billingsley,Dudley,Folland}). For $\varpi\in\mathscr{S}$,
the function value $y=Y(\varpi)\in\mathbb{R}$ is called a \emph{realization}
or \emph{sample} of $Y\sim\mathbb{P}_{Y}$. Hence for any Borel set
$\mathcal{B}\subseteq\mathbb{R}$, an expression like ``\emph{any
realization $y\in\mathcal{B}$ with probability at least $\delta$}''
is simply ``\emph{$\Pr\left\{ Y\in\mathcal{B}\right\} \ge\delta$}''
paraphrased.
\item For a real valued Borel measurable function $f$ on $\mathbb{R}$,
we follow the convention of abusing the notation $f(Y)$ for both
a function on the range (or realizations) of a random variable $Y$
and the random variable $f\circ Y$ (or as a functional of $Y$).
\end{itemize}

\section{\label{sec:rev-em-alg}Review of the EM Algorithm}

In this section, we briefly review the notations and indicate some
extensions to the basic theory of the EM algorithm.

\subsection{Log-Likelihood Function and Maximum Likelihood Estimate}

Suppose a set of \emph{independent and identically distributed} (i.i.d.)
random samples $\{y_{k}\}_{k=1}^{n}$ are observed from a distribution
$\mathbb{P}_{\theta^{*}}$ with an unknown parameter $\theta^{*}\in\Omega\subseteq\mathbb{R}^{p}$.
The goal is to estimate $\theta^{*}$ from these samples. In practice,
we assume the parametric distribution $\mathbb{P}_{\theta}$ has a
density function $p_{\theta}(y)$ for $\theta\in\Omega$ with respect
to the Lebesgue measure on $\mathbb{R}^{d}$.

We consider the samples $\{y_{k}\}_{k=1}^{n}$ as a \emph{realization}
of the i.i.d. copies $\{Y_{k}\}_{k=1}^{n}$ of $Y\sim\mathbb{P}_{\theta^{*}}$.
For the random variable $Y\sim\mathbb{P}_{\theta^{*}}$, we define
the \emph{stochastic log-likelihood functional}
\[
L(\theta;Y)\coloneqq\log p_{\theta}(Y),
\]
for $\theta\in\Omega$. Define the \emph{empirical log-likelihood
functional }as the empirical mean of the stochastic log-likelihood
functionals of the i.i.d. copies $Y_{k}$ ($k=1,\cdots,n$) of $Y$,
i.e.
\[
L_{n}(\theta;\{Y_{k}\})\coloneqq\frac{1}{n}\sum_{k=1}^{n}L(\theta;Y_{k}).
\]

A \emph{maximum likelihood estimate} (MLE) $\widehat{\theta}$ is
obtained by maximizing a\emph{ realization} of the empirical log-likelihood
functional, that is, $\widehat{\theta}\in\arg\max_{\theta\in\Omega}L_{n}(\theta;\{y_{k}\})$
where the maximizer of $L_{n}(\theta;\{y_{k}\})$ may not be unique.
The \emph{expected } \emph{(oracle) log-likelihood function} is the
expectation of the stochastic log-likelihood functional
\begin{equation}
L_{*}(\theta)\coloneqq\mathbb{E}_{\theta^{*}}L(\theta;Y)=\int_{\mathbb{R}^{d}}\log\left(p_{\theta}(y)\right)p_{\theta^{*}}(y)dy.\label{eq:ora-log-lik}
\end{equation}

A fundamental property of the expected log-likelihood function is
the following result.
\begin{prop}
\label{prop:sel-con-loglik}The true population parameter $\theta^{*}$
is a global maximizer of $L_{*}(\theta)$ over $\Omega$. Namely,
\[
\theta^{*}\in\arg\max_{\theta\in\Omega}L_{*}\left(\theta\right).
\]
\end{prop}
\begin{proof}
See \prettyref{subsec:pf-sel-con-loglik}.
\end{proof}
\begin{rem*}
The expected log-likelihood function and the above result are well-known
in the statistics literature (e.g. \cite{Conniffe}) . A proof is given
only for the completeness. Further, if $\theta^{*}$ is an interior
point of $\Omega$ and the function $L_{*}:\Omega\to\mathbb{R}$ is
differentiable\footnote{In the sense of possessing first order partial derivatives.}
in a neighborhood of $\theta^{*}$ then $\nabla_{1}L_{*}(\theta^{*})=0$.
Further, if $\theta^{*}$ is the \emph{unique} maximizer in an open
neighborhood in which $L_{*}$ is twice continuously differentiable,
then in addition to $\nabla_{1}L_{*}(\theta^{*})=0$, the Hessian
matrix of $L_{*}$ at $\theta^{*}$ is negative definite or $\nabla_{1}\nabla_{1}^{\intercal}L_{*}(\theta^{*})\prec0$.
This is equivalent to the positive definiteness of the Fisher information
matrix $I(\theta^{*})$. See \prettyref{eq:obs-inf-mat} and \prettyref{eq:fis-inf-mat}.
\end{rem*}

\subsection{The EM Algorithm and \texorpdfstring{$Q$}{Q}-Functions}

The EM Algorithm is often applied to  maximum likelihood estimation
in latent variable models. The basic assumptions and formulation of
the EM algorithm is briefly summarized as follows.

In latent variable models, the random variable $Y\sim\mathbb{P}_{\theta^{*}}$
is considered as the observed part of a pair $(Y,Z)$ in which $Z$
is the latent or hidden variable. Suppose for $\theta\in\Omega$,
the complete joint density $f_{\theta}(y,z)$ of $\left(Y,Z\right)\in\mathcal{Y}\times\mathcal{Z}\subseteq\mathbb{R}^{d}$
is known, then the density $p_{\theta}(y)$ of $Y$ is the marginalization
$p_{\theta}(y)=\int_{\mathcal{Z}}f_{\theta}(y,z)dz$. Define the conditional
density of $Z$ given $y\in\mathcal{Y}$ as $k_{\theta}(z|y)\coloneqq\frac{f_{\theta}(y,z)}{p_{\theta}(y)}$
for $z\in\mathcal{Z}(y)\coloneqq\left\{ z\mid\left(y,z\right)\in\mathcal{Y}\times\mathcal{Z}\right\} $,
then the stochastic log-likelihood function satisfies
\begin{equation}
L(\theta';y)=\log p_{\theta'}(y)=\log f_{\theta'}(y,z)-\log k_{\theta'}(z|y).\label{eq:int-1-em}
\end{equation}

Now taking conditional expectation of $Z$ given $y$ at parameter
$\theta$ (i.e. multiplying $k_{\theta}(z|y)$ on both sides and integrating
with respect to $z$), one has
\[
L(\theta';y)=Q(\theta'|\theta;y)-H(\theta'|\theta;y)
\]
where $H(\theta'|\theta;y):=\mathbb{E}_{\theta}\left[\log k_{\theta'}(Z|y)\mid y\right]$
and we define the \emph{stochastic $Q$-function}
\[
Q(\theta'|\theta;y)\coloneqq\mathbb{E}_{\theta}\left[\log f_{\theta'}(y,Z)\mid y\right]=\int_{\mathcal{Z}(y)}\log\left(f_{\theta'}(y,z)\right)k_{\theta}(z|y)dz.
\]

Balakrishnan et al. \cite{Balakrishnan} introduced the sample and population $Q$-functions to study the EM algorithm from a Population-Sample based perspective. They defined the \emph{sample (empirical) $Q$-function} as
\[
Q_{n}(\theta'|\theta;\{y_{k}\})\coloneqq\frac{1}{n}\sum_{k=1}^{n}Q(\theta'|\theta;y_{k})=\frac{1}{n}\sum_{k=1}^{n}\int_{\mathcal{Z}(y_{k})}\log\left(f_{\theta'}(y_{k},z)\right)k_{\theta}(z|y_{k})dz,
\]
which is the empirical mean of the stochastic $Q$-functions of a set of
i.i.d. realizations $\{y_{k}\}_{k=1}^{n}$ of $Y\sim\mathbb{P}_{\theta^{*}}$
and also the \emph{population} (or \emph{oracle})\emph{ $Q$-function}
\[
Q_{*}(\theta'|\theta)\coloneqq\int_{\mathcal{Y}}\left(\int_{\mathcal{Z}(y)}\log\left(f_{\theta'}(y,z)\right)k_{\theta}(z|y)dz\right)p_{\theta^{*}}(y)dy
\]
for $\theta',\theta\in\Omega$, which is the expectation (or population mean)
of the stochastic $Q$-function.

A sequence $\{\theta^{t}_{n}\}_{t\ge0}$ generated by maximizing a sample
$Q$-function recursively is called a \emph{sample EM sequence}. A
sequence $\{\theta^{t}\}_{t\ge0}$ generated by maximizing a population
$Q$-function recursively is called a \emph{population (or oracle)
EM sequence}.

Note that the sample $Q$-function is a quantity we can actually compute
in real applications, since its definition involves only a set of samples from
the population distribution $\mathbb{P}_{\theta^{*}}$, while the computation of the
oracle $Q$-function requires knowledge of the true population parameter
$\theta^{*}$.

To develop a data-adaptive theory, we extend these important concepts
in our analysis: First we define the \emph{stochastic $Q$-functional}
\[
Q(\theta'|\theta;Y)\coloneqq\mathbb{E}_{\theta}\left[\log f_{\theta'}(Y,Z)\mid Y\right]
\]
as the basic \emph{derived} random variable of $Y\sim\mathbb{P}_{\theta^{*}}$,
then the stochastic $Q$-function is simply a realization of $Q(\theta'|\theta;Y)$,
the sample $Q$-function is a realization of the \emph{empirical $Q$-functional}
\[
Q_{n}(\theta'|\theta;\{Y_{k}\})\coloneqq\frac{1}{n}\sum_{k=1}^{n}Q(\theta'|\theta;Y_{k}),
\]
which is the empirical mean of the stochastic $Q$-functionals of
i.i.d. copies $\{Y_{k}\}_{k=1}^{n}$ of $Y\sim\mathbb{P}_{\theta^{*}}$,
and the oracle $Q$-function is the population mean of $Q(\theta'|\theta;Y)$.
Then given an initial point $\Theta_{n}^{0}\in B_{r}(\theta^{*})$,
we define the \emph{empirical EM sequence} $\{\Theta_{n}^{t}\}_{t\ge0}$
as
\begin{equation}
\Theta_{n}^{t+1}\in\arg\max\{Q_{n}(\Theta'|\Theta_{n}^{t};\{Y_{k}\})\mid\Theta'\in B_{R}(\theta^{*})\}\text{ for }t\ge0.\label{eq:emp-em-sequ}
\end{equation}
It is not difficult to see that $\Theta_{n}^{t}$ is a measurable
function of $(Y_{1},\cdots,Y_{n})$ for each $t\ge0$, hence a random
variable. Now the sample EM sequence $\{\theta_{n}^{t}\}_{t\ge0}$
is a realization of the empirical EM sequence $\{\Theta_{n}^{t}\}_{t\ge0}$.
For a concrete example, consider the Gaussian Mixture model (see \prettyref{subsec:gau-mix}):
For $\Theta_{n}^{0}\in B_{r}(\theta^{*})$, we have
\[
\Theta_{n}^{t+1}=\frac{1}{n}\sum_{k=1}^{n}\tanh\left(\frac{\left\langle \Theta_{n}^{t},Y_{k}\right\rangle }{\sigma^{2}}\right)Y_{k},
\]
where each $Y_{k}$ is an i.i.d. copy of $Y\sim\frac{1}{2}\mathcal{N}(\theta^{*},\sigma^{2}I_{p})+\frac{1}{2}\mathcal{N}(-\theta^{*},\sigma^{2}I_{p})$
for $1\le k\le n$.

A fundamental property of the oracle $Q$-function, referred as \emph{self-consistency}
in \cite{McLachlan,Balakrishnan}, is the following well-known result.
\begin{prop}
\label{prop:self-const}The true population parameter $\theta^{*}$
is a global maximizer of the oracle $Q$-function on $\Omega$. Namely,
\[
\theta^{*}\in\arg\max_{\theta'\in\Omega}Q_{*}\left(\theta'|\theta^{*}\right).
\]
\end{prop}
\begin{proof}
See \prettyref{subsec:pf-self-const}.
\end{proof}
\begin{rem*}
The proof is given only for the completeness. Note that if $\theta^{*}$
is an interior point of $\Omega$ and the function $Q_{*}(\cdot|\theta^{*}):\Omega\to\mathbb{R}$
is differentiable in a neighborhood of $\theta^{*}$, then the self-consistency
implies that,
\begin{equation}
\nabla_{1}Q_{*}\left(\theta^{*}|\theta^{*}\right)=0\label{eq:self-const}
\end{equation}
which holds true in our \emph{local} analysis of the convergence of
oracle EM sequences.
\end{rem*}

\section{\label{sec:main}Theory for the Convergence of the EM Algorithm}

In this section, we formulate our theoretical framework for the convergence
of the EM algorithm. The main results consist of the optimal oracle
convergence theorem (\prettyref{thm:oct}), the optimal empirical
convergence theorem (\prettyref{thm:oect}) and the optimal rate convergence
theorem (\prettyref{thm:orct}). We also prove the consistency of
the EM algorithm (\prettyref{thm:cema}).

\subsection{\label{subsec:ora-conv}The Oracle Convergence of the EM Algorithm}

We analyze the convergence of oracle EM sequences in this section.
We first define the \emph{derived} random quantities whose population
means characterize the oracle convergence, then we define the set
of contraction parameters and deduce the oracle contraction inequality
which leads to the main theorem.

\subsubsection{Definitions\label{subsec:grv-crv-sev}}

For $Y\sim\mathbb{P}_{\theta^{*}}$, where $\theta^{*}\in\Omega$ is the unknown true population parameter, we define three derived random quantities, the\emph{ gradient difference
random vector }(GRV)
\begin{equation}
\Gamma(\theta;Y)\coloneqq\nabla_{1}Q(\theta^{*}|\theta;Y)-\nabla_{1}Q(\theta^{*}|\theta^{*};Y),\label{eq:grad-ran-vec}
\end{equation}
the \emph{concavity random variable }(CRV)
\begin{equation}
V(\theta'|\theta;Y)\coloneqq Q(\theta'|\theta;Y)-Q(\theta^{*}|\theta;Y)-\left\langle \nabla_{1}Q(\theta^{*}|\theta;Y),\theta'-\theta^{*}\right\rangle ,\label{eq:conc-ran-var}
\end{equation}
and the \emph{statistical error vector }(SEV)
\begin{equation}
\mathcal{E}(Y)\coloneqq\nabla_{1}Q(\theta^{*}|\theta^{*};Y).\label{eq:stat-err-vec}
\end{equation}
As we shall see, the convergence of an oracle EM sequence is characterized by their
population means
\[
\mathbb{E}_{\theta^{*}}\Gamma(\theta;Y)=\nabla_{1}Q_{*}(\theta^{*}|\theta)-\nabla_{1}Q_{*}(\theta^{*}|\theta^{*})\text{ and}
\]
\begin{equation}
\mathbb{E}_{\theta^{*}}V(\theta'|\theta;Y)=Q_{*}(\theta'|\theta)-Q_{*}(\theta^{*}|\theta)-\left\langle \nabla_{1}Q_{*}(\theta^{*}|\theta),\theta'-\theta^{*}\right\rangle .\label{eq:exp-grv-crv}
\end{equation}
Note that
\begin{equation}
\mathbb{E}_{\theta^{*}}\mathcal{E}(Y)=\mathbb{E}_{\theta^{*}}\nabla_{1}Q(\theta^{*}|\theta^{*};Y)=\nabla_{1}Q_{*}(\theta^{*}|\theta^{*})=0,\label{eq:exp-stat-err}
\end{equation}
which follows from \prettyref{prop:self-const} and the remark on
the self-consistency of the oracle $Q$-function.

\subsubsection{The Contraction Parameters}

In our theory, the convergence behavior of an oracle EM sequence in
a given ball $B_{r}(\theta^{*})$ is characterized by a pair of parameters
$\left(\gamma,\nu\right)$ and we consider \emph{all }possible parameters
for \emph{any} ball centered at the true population parameter $\theta^{*}$.
Specifically, for $0<r\le R$, we define the following sets,
\begin{equation}
\mathcal{G}(r)\coloneqq\{\gamma>0\mid\left\Vert \mathbb{E}_{\theta^{*}}\Gamma(\theta;Y)\right\Vert \le\gamma\left\Vert \theta-\theta^{*}\right\Vert \text{ for }\theta\in B_{r}(\theta^{*})\}\text{ and}\label{eq:G-r}
\end{equation}
\begin{equation}
\mathcal{V}(r,R)\coloneqq\{\nu>0\mid\mathbb{E}_{\theta^{*}}V(\theta'|\theta;Y)\le-\nu\left\Vert \theta'-\theta^{*}\right\Vert ^{2}\text{ for }\left(\theta',\theta\right)\in B_{R}(\theta^{*})\times B_{r}(\theta^{*})\}.\label{eq:V-r-R}
\end{equation}

Thus in view of \prettyref{eq:exp-grv-crv}, for each $\gamma\in$$\mathcal{G}(r)$,
the oracle $Q$-function satisfies a \emph{gradient stability} ($\gamma$-GS)
condition:
\begin{equation}
\left\Vert \nabla_{1}Q_{*}(\theta^{*}|\theta)-\nabla_{1}Q_{*}(\theta^{*}|\theta^{*})\right\Vert \le\gamma\left\Vert \theta-\theta^{*}\right\Vert \text{ for }\theta\in B_{r}(\theta^{*}),\label{eq:grad-stab-ora}
\end{equation}
and for each $\nu\in\mathcal{V}(r,R)$, the oracle $Q$-function satisfies
a \emph{local uniform strong concavity} ($\nu$-LUSC) condition:
\begin{equation}
Q_{*}(\theta'|\theta)-Q_{*}(\theta^{*}|\theta)-\left\langle \nabla_{1}Q_{*}(\theta^{*}|\theta),\theta'-\theta^{*}\right\rangle \le-\nu\left\Vert \theta'-\theta^{*}\right\Vert ^{2}\label{eq:stro-conc-ora}
\end{equation}
for $\left(\theta',\theta\right)\in B_{R}(\theta^{*})\times B_{r}(\theta^{*})$.

Note the gradient stability condition \prettyref{eq:grad-stab-ora}
is different from the Gradient Smoothness Condition and the First-order
Stability Condition introduced in \cite{Balakrishnan}. Our gradient
stability condition is equivalent to the gradient $\nabla_{1}Q_{*}\left(\theta^{*}|\cdot\right)$
being Lipschitz continuous at $\theta^{*}$ with parameter $\gamma$.

The local uniform strong concavity condition is different from the
Strong Concavity condition in \cite{Balakrishnan,Wang,Yi}. It requires
that the oracle $Q$-function $Q_{*}(\theta'|\theta)$ is $\nu$-strongly
concave with respect to $\theta'$ at the point $\theta^{*}$, and
that it holds uniformly for all $\theta\in B_{r}(\theta^{*})$. This
condition is easily verified when $Q_{*}(\theta'|\theta)$ is a quadratic
function in $\theta'$ and independent of $\theta$, which is the
case for Gaussian Mixture (\prettyref{subsec:gau-mix}) and Mixture
of Linear Regression (\prettyref{subsec:mix-lin-reg}). From the theoretical
perspective, the local uniform strong concavity condition is motivated
by the following proposition.
\begin{prop}
\label{prop:fisher-non-empt-V}If the Fisher information matrix $I(\theta)$
of the parametric density $p_{\theta}(y)$ is positive definite at
$\theta^{*}$, then there exist $0<r\le R$ such that $\mathcal{V}(r,R)\neq\varnothing$.
\end{prop}
\begin{proof}
See the proof in \prettyref{subsec:infor-mat}.
\end{proof}
Intuitively, for $r>0$, the set $\mathcal{G}(r)$ consists of all
$\gamma>0$, such that the oracle $Q$-function satisfies a $\gamma$-GS
condition in $B_{r}(\theta^{*})$; and for $0<r\le R$, the set $\mathcal{V}(r,R)$
consists of all $\nu>0$ such that the oracle $Q$-function satisfies
a $\nu$-LUSC condition in $B_{R}(\theta^{*})\times B_{r}(\theta^{*})$.
It is easy to see that $\mathcal{G}(r_{1})\subseteq\mathcal{G}(r_{2})$
if $r_{1}\ge r_{2}\ge0$ and $\mathcal{V}(r,R_{1})\subseteq\mathcal{V}(r,R_{2})$
if $R_{1}\ge R_{2}\ge r\ge0$.

There is no \emph{a priori} guarantee that these sets are non-empty
for given $0<r\le R$, but if this is the case, then the following
lemma completely characterizes these sets.
\begin{lem}
\label{lem:G=00003D00003D000026V}Let $\mathcal{G}(r)$ and $\mathcal{V}(r,R)$
be defined as above.

$(a)$ If $\mathcal{G}(r)\neq\varnothing$ for some $r>0$, then $\mathcal{G}(r)=\left[\overline{\gamma},+\infty\right)$
where
\begin{equation}
\overline{\gamma}\coloneqq\sup\left\{ \frac{\left\Vert \mathbb{E}_{\theta^{*}}\Gamma(\theta;Y)\right\Vert }{\left\Vert \theta-\theta^{*}\right\Vert }\mid\theta\in B_{r}^{\times}(\theta^{*})\right\} \in\mathbb{R}\label{eq:gam-bar};
\end{equation}

$(b)$ If $\mathcal{V}(r,R)\neq\varnothing$ for some $R\ge r>0$,
then $\mathcal{V}(r,R)=\left(0,\overline{\nu}\right]$ where
\begin{equation}
\overline{\nu}\coloneqq\inf\left\{ -\frac{\mathbb{E}_{\theta^{*}}V(\theta'|\theta;Y)}{\left\Vert \theta'-\theta^{*}\right\Vert ^{2}}\mid\left(\theta',\theta\right)\in B_{R}^{\times}(\theta^{*})\times B_{r}(\theta^{*})\right\} \in\mathbb{R}\label{eq:nu-bar}.
\end{equation}
\end{lem}
\begin{proof}
$(a)$ For any $\gamma\in\mathcal{G}(r)\neq\varnothing$, we have
by definition
\[
\left\Vert \mathbb{E}_{\theta^{*}}\Gamma(\theta;Y)\right\Vert \le\gamma\left\Vert \theta-\theta^{*}\right\Vert \text{ for }\theta\in B_{r}(\theta^{*}),
\]
and hence
\[
\frac{\left\Vert \mathbb{E}_{\theta^{*}}\Gamma(\theta;Y)\right\Vert }{\left\Vert \theta-\theta^{*}\right\Vert }\le\gamma\text{ for }\theta\in B_{r}^{\times}(\theta^{*}).
\]
It then follows that
\[
\overline{\gamma}=\sup\left\{ \frac{\left\Vert \mathbb{E}_{\theta^{*}}\Gamma(\theta;Y)\right\Vert }{\left\Vert \theta-\theta^{*}\right\Vert }\mid\theta\in B_{r}^{\times}(\theta^{*})\right\} \le\gamma
\]
and hence $\overline{\gamma}\le\inf\mathcal{G}(r)$, since $\gamma\in\mathcal{G}(r)$
is arbitrary. Now by definition
\[
\frac{\left\Vert \mathbb{E}_{\theta^{*}}\Gamma(\theta;Y)\right\Vert }{\left\Vert \theta-\theta^{*}\right\Vert }\le\overline{\gamma}\text{ for }\theta\in B_{r}^{\times}(\theta^{*})
\]
and in view of the fact that $\mathbb{E}_{\theta^{*}}\Gamma(\theta^{*};Y)=0$,
we see $\overline{\gamma}\in\mathcal{G}(r)$ and hence $\overline{\gamma}=\min\mathcal{G}(r)$.
The result follows by noticing that if $\gamma\in\mathcal{G}(r)$,
then $\gamma'\in\mathcal{G}(r)$ for any $\gamma'>\gamma$.

$(b)$ For any $\nu\in\mathcal{V}(r,R)\neq\varnothing$, we have by
definition
\[
\mathbb{E}_{\theta^{*}}V(\theta'|\theta;Y)\le-\nu\left\Vert \theta'-\theta^{*}\right\Vert ^{2}\text{ for }\left(\theta',\theta\right)\in B_{R}(\theta^{*})\times B_{r}(\theta^{*}),
\]
and hence
\[
-\frac{\mathbb{E}_{\theta^{*}}V(\theta'|\theta;Y)}{\left\Vert \theta'-\theta^{*}\right\Vert ^{2}}\ge\nu\text{ for }\left(\theta',\theta\right)\in B_{R}^{\times}(\theta^{*})\times B_{r}(\theta^{*}).
\]
It then follows that
\[
\overline{\nu}=\inf\left\{ -\frac{\mathbb{E}_{\theta^{*}}V(\theta'|\theta)}{\left\Vert \theta'-\theta^{*}\right\Vert ^{2}}\mid\left(\theta',\theta\right)\in B_{R}^{\times}(\theta^{*})\times B_{r}(\theta^{*})\right\} \ge\nu
\]
and hence $\overline{\nu}\ge\sup\mathcal{V}(r,R)$, since $\nu\in\mathcal{V}(r,R)$
is arbitrary. Now by definition
\[
-\frac{\mathbb{E}_{\theta^{*}}V(\theta'|\theta;Y)}{\left\Vert \theta'-\theta^{*}\right\Vert ^{2}}\ge\overline{\nu}\text{ for }\left(\theta',\theta\right)\in B_{R}^{\times}(\theta^{*})\times B_{r}(\theta^{*})
\]
and in view of the fact that $\mathbb{E}_{\theta^{*}}V(\theta^{*}|\theta;Y)=0$,
we see $\overline{\nu}\in\mathcal{V}(r,R)$ and hence $\overline{\nu}=\max\mathcal{V}(r,R)$.
The result follows by noticing that if $\nu\in\mathcal{V}(r,R)$,
then $\nu'\in\mathcal{V}(r,R)$ for any $0<\nu'<\nu$.
\end{proof}

\subsubsection{The Oracle Contraction Inequality}

The pair of parameters $(\gamma,\nu)\in\mathcal{G}(r)\times\mathcal{V}(r,R)$
gives rise to an inequality which lies in the core of our oracle convergence
theory.
\begin{prop}
\label{prop:ora-ineq}If $\mathcal{G}(r)\times\mathcal{V}(r,R)\neq\varnothing$
for some $0<r\le R$, then for any $\theta\in B_{r}(\theta^{*})$
and $\theta'\in B_{R}(\theta^{*})$ such that
\[
Q_{*}(\theta'|\theta)\ge Q_{*}(\theta^{*}|\theta),
\]
there holds the inequality
\begin{equation}
\left\Vert \theta'-\theta^{*}\right\Vert \le\frac{\gamma}{\nu}\left\Vert \theta-\theta^{*}\right\Vert \label{eq:ora-contr-ineq}
\end{equation}
for each pair $(\gamma,\nu)\in\mathcal{G}(r)\times\mathcal{V}(r,R)$.
\end{prop}
\begin{proof}
By definitions \prettyref{eq:G-r} and \prettyref{eq:V-r-R}, for
any $(\gamma,\nu)\in\mathcal{G}(r)\times\mathcal{V}(r,R)$, we have
\[
\left\Vert \mathbb{E}_{\theta^{*}}\Gamma(\theta;Y)\right\Vert \le\gamma\left\Vert \theta-\theta^{*}\right\Vert \text{ for }\theta\in B_{r}(\theta^{*})\text{ and}
\]
\[
\mathbb{E}_{\theta^{*}}V(\theta'|\theta;Y)\le-\nu\left\Vert \theta'-\theta^{*}\right\Vert ^{2}\text{ for }\left(\theta',\theta\right)\in B_{R}^{\times}(\theta^{*})\times B_{r}(\theta^{*}).
\]
Then it follows that
\begin{align*}
0 & \le Q_{*}(\theta'|\theta)-Q_{*}(\theta^{*}|\theta)\\
 & \stackrel{(a)}{=}\mathbb{E}_{\theta^{*}}V(\theta'|\theta;Y)+\left\langle \nabla_{1}Q_{*}(\theta^{*}|\theta),\theta'-\theta^{*}\right\rangle \\
 & \stackrel{(b)}{\le}\mathbb{E}_{\theta^{*}}V(\theta'|\theta;Y)+\left\Vert \nabla_{1}Q_{*}(\theta^{*}|\theta)\right\Vert \cdot\left\Vert \theta'-\theta^{*}\right\Vert \\
 & \stackrel{(c)}{=}\mathbb{E}_{\theta^{*}}V(\theta'|\theta;Y)+\left\Vert \mathbb{E}_{\theta^{*}}\Gamma(\theta;Y)\right\Vert \cdot\left\Vert \theta'-\theta^{*}\right\Vert \\
 & \le-\nu\left\Vert \theta'-\theta^{*}\right\Vert ^{2}+\gamma\left\Vert \theta-\theta^{*}\right\Vert \cdot\left\Vert \theta'-\theta^{*}\right\Vert
\end{align*}
where $(a)$ follows from the definition of $\mathbb{E}_{\theta^{*}}V(\theta'|\theta;Y)$;
$(b)$ follows from the Cauchy-Schwartz inequality; and $(c)$ follows
from the definition of $\mathbb{E}_{\theta^{*}}\Gamma(\theta;Y)$ and the self-consistency \prettyref{eq:self-const}. Hence the proposition follows.
\end{proof}

\subsubsection{\label{subsec:main-thm-oct}The Optimal Oracle Convergence Theorem}

We note \prettyref{eq:ora-contr-ineq} holds for \emph{any} pair of
$(\gamma,\nu)\in\mathcal{G}(r)\times\mathcal{V}(r,R)$, not only those
$\gamma<\nu$. However, we are more interested in the case when it
is indeed a contraction. For this purpose, let
\[
\mathcal{T}\coloneqq\{(x,y)\in\mathbb{R}^{+}\times\mathbb{R}^{+}\mid x<y\}
\]
be the open upper-triangle of the first quadrant and define the \emph{set
of contraction parameters}
\[
\mathcal{C}(r,R)\coloneqq\left(\mathcal{G}(r)\times\mathcal{V}(r,R)\right)\cap\mathcal{T}=\{(\gamma,\nu)\mid\gamma\in\mathcal{G}(r),\nu\in\mathcal{V}(r,R)\text{ such that }\gamma<\nu\},
\]
and we say $0<r\le R$ are\emph{ radii of contraction} if $\mathcal{C}(r,R)\not=\varnothing$.

If $0<r\le R$ are radii of contraction, then in view of \prettyref{lem:G=00003D00003D000026V},
we have
\[
\mathcal{C}(r,R)=\left(\left[\overline{\gamma},+\infty\right)\times\left(0,\overline{\nu}\right]\right)\cap\mathcal{T},
\]
and we call $\left(\overline{\gamma},\overline{\nu}\right)$ the \emph{optimal
pair }since the ratio $\frac{\overline{\gamma}}{\overline{\nu}}\le\frac{\gamma}{\nu}<1$
for any $(\gamma,\nu)\in\mathcal{C}(r,R)$, which then gives the \emph{optimal
rate of oracle convergence} with respect to the radii of contraction
$r\le R$. Indeed, this is the content of our main theorem of this
section.
\begin{thm}[Optimal Oracle Convergence Theorem]
\label{thm:oct}Suppose $0<r\le R$ are radii of contraction, then
given initial point $\theta^{0}\in B_{r}(\theta^{*})$, any oracle
EM sequence $\{\theta^{t}\}_{t\ge0}$ such that
\begin{equation}
\theta^{t+1}\in\arg\max\{Q_{*}(\theta'|\theta^{t})\mid\theta'\in B_{R}(\theta^{*})\}\text{ for }t\ge0,\label{eq:ora-em-seq}
\end{equation}
satisfies the inequality
\begin{equation}
\left\Vert \theta^{t}-\theta^{*}\right\Vert \le\overline{\kappa}^{t}\left\Vert \theta^{0}-\theta^{*}\right\Vert \label{eq:ora-conv}
\end{equation}
where $\overline{\kappa}\coloneqq\frac{\overline{\gamma}}{\overline{\nu}}\le\frac{\gamma}{\nu}<1$
for any $(\gamma,\nu)\in\mathcal{C}(r,R)$, is the optimal rate of
oracle convergence with respect to $r\le R$.
\end{thm}
\begin{proof}
We only need to show that
\begin{equation}
\left\Vert \theta^{t}-\theta^{*}\right\Vert \le\left(\frac{\gamma}{\nu}\right)^{t}\left\Vert \theta^{0}-\theta^{*}\right\Vert \label{eq:int-1-oct}
\end{equation}
holds for any $(\gamma,\nu)\in\mathcal{C}(r,R)$ and $t\in\mathbb{N}$,
and from which the result follows. We proceed by induction. It is
clear that \prettyref{eq:int-1-oct} holds for $t=0$. Assume it holds
for $t\ge0$ then $\theta^{t}\in B_{r}(\theta^{*})$ since $\frac{\gamma}{\nu}<1$,
and by definition
\[
Q_{*}(\theta^{t+1}|\theta^{t})\ge Q_{*}(\theta^{*}|\theta^{t})
\]
and $\theta^{t+1}\in B_{R}(\theta^{*})$. It follows from \prettyref{prop:ora-ineq}
and induction hypothesis that
\[
\left\Vert \theta^{t+1}-\theta^{*}\right\Vert \le\frac{\gamma}{\nu}\left\Vert \theta^{t}-\theta^{*}\right\Vert \le\left(\frac{\gamma}{\nu}\right)^{t+1}\left\Vert \theta^{0}-\theta^{*}\right\Vert
\]
and hence \prettyref{eq:int-1-oct} holds for $t+1$ and the proof
is complete.
\end{proof}
\begin{rem*}
The theorem above can be viewed as a stronger version of the population
convergence result in Theorem 4 of \cite{Balakrishnan}. It gives
a \emph{family} of deterministic convergence inequalities for oracle
EM sequences, one for each pair of $(\gamma,\nu)\in\mathcal{C}(r,R)$.
It also asserts that the oracle EM sequence converges geometrically
at the \emph{optimal} rate $\overline{\kappa}$ with respect to given
radii of contraction $0<r\le R$. Although in specific models and
real applications, we only \emph{calculate }one or a class of convergence
rates $\kappa$ for some ball of contraction, (see Sections \ref{subsec:ora-cov-gmm},
\ref{subsec:ora-cov-mlr} and \ref{subsec:ora-cov-rmc}), the oracle
EM sequence always converges at the optimal rate $\overline{\kappa}\le\kappa$
with respect to that ball of contraction.
\end{rem*}

\subsection{\label{subsec:emp-conv}The Empirical Convergence and Consistency
of the EM Algorithm}

The empirical EM sequence is constructed iteratively
by maximizing the empirical $Q$-functional $Q_{n}\left(\theta'|\theta;\{Y_{k}\}\right)$,
which is the empirical approximation on a finite set $\{Y_{k}\}_{k=1}^{n}$
of i.i.d. copies of $Y\sim\mathbb{P}_{\theta^{*}}$, to the oracle
(population) $Q$-function.

Our intuition is that, due to the concentration of measure phenomenon
of random variables, the convergence behavior of the empirical EM
sequence ought to ``concentrate'' on the convergence behavior of
the corresponding oracle EM sequence, with high probability.

Hence the results established in the oracle convergence theorem, namely
the radii of contraction and the set of contraction parameters are
\emph{oracle }information that we can exploit to help illuminate the
convergence behavior of the empirical EM sequence. To substantiate
this idea with mathematical rigor, we prove the optimal empirical
convergence theorem in this section and as a consequence, a theorem
on the statistical consistency of the EM algorithm. We first introduce
the empirical versions of GRV, CRV and SEV.

\subsubsection{Basic Definitions and Assumptions\label{subsec:grv-crv-sev-emp}}

For a set $\{Y_{k}\}_{k=1}^{n}$ of i.i.d. copies of $Y\sim\mathbb{P}_{\theta^{*}}$,
we define the \emph{empirical gradient difference random vector} as
\[
\Gamma_{n}(\theta;\{Y_{k}\})\coloneqq\frac{1}{n}\sum_{k=1}^{n}\Gamma(\theta;Y_{k})=\nabla_{1}Q_{n}(\theta^{*}|\theta;\{Y_{k}\})-\nabla_{1}Q_{n}(\theta^{*}|\theta^{*};\{Y_{k}\}),
\]
the \emph{empirical concavity random variable} as
\begin{align*}
V_{n}(\theta'|\theta;\{Y_{k}\}) & \coloneqq\frac{1}{n}\sum_{k=1}^{n}V(\theta'|\theta;Y_{k})\\
 & =Q_{n}(\theta'|\theta;\{Y_{k}\})-Q_{n}(\theta^{*}|\theta;\{Y_{k}\})-\left\langle \nabla_{1}Q_{n}(\theta^{*}|\theta;\{Y_{k}\}),\theta'-\theta^{*}\right\rangle ,
\end{align*}
and also the \emph{empirical} \emph{statistical error vector} as
\[
\mathcal{E}_{n}(\{Y_{k}\})\coloneqq\frac{1}{n}\sum_{k=1}^{n}\mathcal{E}(Y_{k})=\nabla_{1}Q_{n}(\theta^{*}|\theta^{*};\{Y_{k}\}).
\]

In order to exploit the oracle information from the population version
of these quantities, we postulate the following assumptions on the
empirical versions of the GRV, CRV and SEV.
\begin{description}
\item [{Assumptions\label{Assumptions}}]~

For $\delta\in(0,1)$ and a set $\{Y_{k}\}_{k=1}^{n}$ of i.i.d. copies
of $Y\sim\mathbb{P}_{\theta^{*}}$:
\begin{enumerate}[label=(\textbf{A\arabic*}),ref={A\arabic*}]
\item \label{A1}There exits $\varepsilon_{1}(\delta,r,n,p)\ge0$ such
that
\[
\left\Vert \Gamma_{n}(\theta;\{Y_{k}\})-\mathbb{E}_{\theta^{*}}\Gamma(\theta;Y)\right\Vert \le\varepsilon_{1}(\delta,r,n,p)\left\Vert \theta-\theta^{*}\right\Vert
\]
for $\theta\in B_{r}(\theta^{*})$ with probability at least $1-\delta$.
\item \label{A2}There exits $\varepsilon_{2}(\delta,r,R,n,p)\ge0$ such
that
\[
\left|V_{n}(\theta'|\theta;\{Y_{k}\})-\mathbb{E}_{\theta^{*}}V(\theta'|\theta;Y)\right|\le\varepsilon_{2}(\delta,r,R,n,p)\left\Vert \theta'-\theta^{*}\right\Vert ^{2}
\]
for $\left(\theta',\theta\right)\in B_{R}(\theta^{*})\times B_{r}(\theta^{*})$
with probability at least $1-\delta$.
\item \label{A3}There exists $\varepsilon_{s}(\delta,r,R,n,p)>0$ such
that
\[
\left\Vert \mathcal{E}_{n}(\{Y_{k}\})\right\Vert \le\varepsilon_{s}(\delta,r,R,n,p)
\]
with probability at least $1-\delta$.
\end{enumerate}
\end{description}
\begin{rem*}
For the measurability issue involved in the assumptions, see \prettyref{subsec:mea}.
These assumptions are natural concentration inequalities for random
variables or vectors, they are readily verified in canonical example
models. And in view of the Law of Large Numbers, we have that $\varepsilon_{1}(\delta,r,n,p)\to0$,
$\varepsilon_{2}(\delta,r,R,n,p)\to0$ and $\varepsilon_{s}(\delta,r,R,n,p)\to0$
as the sample size $n\to\infty$.
\end{rem*}

\subsubsection{\label{subsec:def-oecr}The Optimal Empirical Convergence Rate}

Now we proceed to define the central object of our data-adaptive analysis
of the EM algorithm. For $0<r\le R$ and i.i.d. copies $\{Y_{k}\}_{k=1}^{n}$
of $Y\sim\mathbb{P}_{\theta^{*}}$, define the (possibly) extended
real-valued (with range $\mathbb{R}\cup\{\pm\infty\}$) random variables\footnote{For the measurability issue, see \prettyref{subsec:mea}}
\[
\overline{\varGamma}_{n}\coloneqq\sup\left\{ \frac{\left\Vert \Gamma_{n}(\theta;\{Y_{k}\})\right\Vert }{\left\Vert \theta-\theta^{*}\right\Vert }\mid\theta\in B_{r}^{\times}(\theta^{*})\right\} ,\quad\overline{E}_{n}\coloneqq\left\Vert \mathcal{E}_{n}(\{Y_{k}\})\right\Vert \text{ and }
\]
\begin{equation}
\overline{V}_{n}\coloneqq\inf\left\{ -\frac{V_{n}(\theta'|\theta;\{Y_{k}\})}{\left\Vert \theta'-\theta^{*}\right\Vert ^{2}}\mid\left(\theta',\theta\right)\in B_{R}^{\times}(\theta^{*})\times B_{r}(\theta^{*})\right\} \label{eq:opt-emp-con-paras}
\end{equation}

The following lemma asserts that these random variables assume finite
values and are properly bounded with high probability under our assumptions.
\begin{lem}
\label{lem:bounds}Suppose $\delta\in(0,1)$ and $0<r\le R$. If
assumptions (\ref{A1}), (\ref{A2}) and (\ref{A3}) hold true, then
for any $(\gamma,\nu)\in\mathcal{G}(r)\times\mathcal{V}(r,R)$, with
probability at least $1-\delta$, these random variables satisfy $\overline{\varGamma}_{n}\le\gamma_{n}$,
$\overline{V}_{n}\ge\nu_{n}$ and $\overline{E}_{n}\le\varepsilon_{s}(\delta,r,R,n,p)$
where $\gamma_{n}\coloneqq\gamma+\varepsilon_{1}(\delta,r,n,p)\text{ and }\nu_{n}\coloneqq\nu-\varepsilon_{2}(\delta,r,R,n,p).$
\end{lem}
\begin{proof}
Since $(\gamma,\nu)\in\mathcal{G}(r)\times\mathcal{V}(r,R)$ and by
definitions \prettyref{eq:G-r} and \prettyref{eq:V-r-R}, one has
\[
\left\Vert \mathbb{E}_{\theta^{*}}\Gamma(\theta;Y)\right\Vert \le\gamma\left\Vert \theta-\theta^{*}\right\Vert \text{ for }\theta\in B_{r}(\theta^{*})\text{ and }
\]
\[
\mathbb{E}_{\theta^{*}}V(\theta'|\theta;Y)\le-\nu\left\Vert \theta'-\theta^{*}\right\Vert ^{2}\text{ for }\left(\theta',\theta\right)\in B_{R}(\theta^{*})\times B_{r}(\theta^{*}).
\]
Then by assumption (\ref{A1}), for the set $\{Y_{k}\}_{k=1}^{n}$
of i.i.d. copies of $Y\sim\mathbb{P}_{\theta^{*}}$ and $\theta\in B_{r}(\theta^{*})$,
\begin{align*}
\left\Vert \Gamma_{n}(\theta;\{Y_{k}\})\right\Vert  & \le\left\Vert \mathbb{E}_{\theta^{*}}\Gamma(\theta;Y)\right\Vert +\left\Vert \Gamma_{n}(\theta;\{Y_{k}\})-\mathbb{E}_{\theta^{*}}\Gamma(\theta;Y)\right\Vert \\
 & \le\gamma\left\Vert \theta-\theta^{*}\right\Vert +\varepsilon_{1}(\delta,r,n,p)\left\Vert \theta-\theta^{*}\right\Vert \\
 & =\gamma_{n}\left\Vert \theta-\theta^{*}\right\Vert
\end{align*}
with probability at least $1-\delta/3$. It follows that
\[
\overline{\varGamma}_{n}=\sup\left\{ \frac{\left\Vert \Gamma_{n}(\theta;\{Y_{k}\})\right\Vert }{\left\Vert \theta-\theta^{*}\right\Vert }\mid\theta\in B_{r}^{\times}(\theta^{*})\right\} \le\gamma_{n}.
\]

Likewise, by assumption (\ref{A2}) and for $\left(\theta',\theta\right)\in B_{R}(\theta^{*})\times B_{r}(\theta^{*})$,
we have
\begin{align*}
V_{n}(\theta'|\theta;\{Y_{k}\}) & \le\mathbb{E}_{\theta^{*}}V(\theta'|\theta;Y)+\left|V_{n}(\theta'|\theta;\{Y_{k}\})-\mathbb{E}_{\theta^{*}}V(\theta'|\theta;Y)\right|\\
 & \le-\nu\left\Vert \theta'-\theta^{*}\right\Vert ^{2}+\varepsilon_{2}(\delta,r,R,n,p)\left\Vert \theta'-\theta^{*}\right\Vert ^{2}\\
 & =-\nu_{n}\left\Vert \theta'-\theta^{*}\right\Vert ^{2}
\end{align*}
with probability at least $1-\delta/3$. It follows that
\[
\overline{V}_{n}=\inf\left\{ -\frac{V_{n}(\theta'|\theta;\{Y_{k}\})}{\left\Vert \theta'-\theta^{*}\right\Vert ^{2}}\mid\left(\theta',\theta\right)\in B_{R}^{\times}(\theta^{*})\times B_{r}(\theta^{*})\right\} \ge\nu_{n}.
\]

Moreover by assumption (\ref{A3}), $\overline{E}_{n}=\left\Vert \mathcal{E}_{n}(\{Y_{k}\})\right\Vert \le\varepsilon_{s}(\delta,r,R,n,p)$
with probability at least $1-\delta/3$. Then the lemma is proved by
applying a union bound.
\end{proof}
If $\delta\in(0,1)$ and $0<r\le R$ are radii of contraction, then
the above lemma holds for the optimal pair $(\overline{\gamma},\overline{\nu})\in\mathcal{C}(r,R)$
and set $\overline{\gamma}_{n}\coloneqq\overline{\gamma}+\varepsilon_{1}(\delta,r,n,p)$
and $\overline{\nu}_{n}\coloneqq\overline{\nu}-\varepsilon_{2}(\delta,r,R,n,p).$
We call $(\overline{\gamma}_{n},\overline{\nu}_{n})$ the \emph{optimal
pair} of the empirical contraction parameters. Define the event
\begin{equation}
\mathscr{E}_{n}\coloneqq\left\{ \varpi\mid\overline{\varGamma}_{n}\le\overline{\gamma}_{n},\overline{V}_{n}\ge\overline{\nu}_{n}\text{ and }\overline{E}_{n}\le\varepsilon_{s}(\delta,r,R,n,p)\right\} ,\label{eq:evt-E_n}
\end{equation}
then the above lemma implies that $\Pr\mathscr{E}_{n}\ge1-\delta$.
Now we define the random variable
\begin{equation}
\overline{K}_{n}\coloneqq\begin{cases}
\min\left\{ \frac{\overline{\varGamma}_{n}}{\overline{V}_{n}},\overline{\kappa}_{n}\right\}  & \text{ if }0<\overline{V}_{n}<\infty\\
\overline{\kappa}_{n} & \text{ otherwise}
\end{cases}\label{eq:def-Kn}
\end{equation}
as the \emph{optimal empirical convergence rate}, where $\overline{\kappa}_{n}\coloneqq\frac{\overline{\gamma}_{n}}{\overline{\nu}_{n}}$.
Under the event $\mathscr{E}_{n}$, we have that $0<\frac{\overline{\varGamma}_{n}}{\overline{V}_{n}}\le\overline{\kappa}_{n}$, hence $\overline{K}_{n}=\frac{\overline{\varGamma}_{n}}{\overline{V}_{n}}$.
And by definition $0<\overline{K}_{n}\le\overline{\kappa}_{n}$.

\subsubsection{The Optimal Empirical Convergence Theorem}

The following proposition gives the \emph{optimal empirical contraction
inequality}, which lies in the core of our empirical convergence theory.
\begin{prop}
\label{prop:opt-ineq}Let $\delta\in(0,1)$ and $\{Y_{k}\}_{k=1}^{n}$
be a set of i.i.d. copies of $Y\sim\mathbb{P}_{\theta^{*}}$. Suppose
$0<r\le R$ are radii of contraction and $\left(\overline{\gamma},\overline{\nu}\right)\in\mathcal{C}(r,R)$
is the optimal pair. If assumptions (\ref{A1}), (\ref{A2}) and
(\ref{A3}) hold true and the sample size $n$ is sufficiently large
such that $\overline{\nu}_{n}>0$, then for any $\theta\in B_{r}(\theta^{*})$
and $\theta'\in B_{R}(\theta^{*})$ such that
\[
Q_{n}(\theta'|\theta;\{Y_{k}\})\ge Q_{n}(\theta^{*}|\theta;\{Y_{k}\}),
\]
there holds the inequality
\begin{equation}
\left\Vert \theta'-\theta^{*}\right\Vert \le\overline{K}_{n}\left\Vert \theta-\theta^{*}\right\Vert +\frac{\overline{E}_{n}}{\overline{V}_{n}}\label{eq:opt-contr-ineq}
\end{equation}
with probability at least $1-\delta$.
\end{prop}
\begin{proof}
By a similar argument to that of \prettyref{prop:ora-ineq}, we have
\begin{align*}
0 & \le Q_{n}(\theta'|\theta;\{Y_{k}\})-Q_{n}(\theta^{*}|\theta;\{Y_{k}\})\\
 & \stackrel{(a)}{=}V_{n}(\theta'|\theta;\{Y_{k}\})+\left\langle \nabla_{1}Q_{n}(\theta^{*}|\theta;\{Y_{k}\}),\theta'-\theta^{*}\right\rangle \\
 & \stackrel{(b)}{\le}V_{n}(\theta'|\theta;\{Y_{k}\})+\left\Vert \nabla_{1}Q_{n}(\theta^{*}|\theta;\{Y_{k}\})\right\Vert \cdot\left\Vert \theta'-\theta^{*}\right\Vert \\
 & \stackrel{(c)}{=}V_{n}(\theta'|\theta;\{Y_{k}\})+\left\Vert \Gamma_{n}(\theta;\{Y_{k}\})+\mathcal{E}_{n}(\{Y_{k}\})\right\Vert \cdot\left\Vert \theta'-\theta^{*}\right\Vert \\
 & \stackrel{(d)}{\le}-\overline{V}_{n}\left\Vert \theta'-\theta^{*}\right\Vert ^{2}+\left(\overline{\Gamma}_{n}\left\Vert \theta-\theta^{*}\right\Vert +\overline{E}_{n}\right)\cdot\left\Vert \theta'-\theta^{*}\right\Vert
\end{align*}
where $(a)$ follows from the definition of $V_{n}(\theta'|\theta;\{Y_{k}\})$;
$(b)$ follows from the Cauchy-Schwartz inequality; $(c)$ follows
from the definitions of $\Gamma_{n}(\theta;\{Y_{k}\})$ and $\mathcal{E}_{n}(\{Y_{k}\})$; and $(d)$ follows from the definition of the random variables $\overline{\Gamma}_{n}$,
$\overline{V}_{n}$ and $\overline{E}_{n}$. Conditioning on the event
$\mathscr{E}_{n}$, we have $\overline{V}_{n}\ge\overline{\nu}_{n}>0$
, then we can perform the division by $\overline{V}_{n}$ on both sides of above inequality, and obtain the desired result.
\end{proof}
\begin{rem*}
Since $\varepsilon_{2}(\delta,r,R,n,p)\to0$ as $n\to\infty$, we
have $\overline{\nu}_{n}=\overline{\nu}-\varepsilon_{2}(\delta,r,R,n,p)>0$
when $n$ is sufficiently large.
\end{rem*}
Before we state the main theorem, we prove one more technical lemma
for an event bound.
\begin{lem}
\label{lem:evt-bound}Let $\delta\in(0,1)$ and $\{Y_{k}\}_{k=1}^{n}$
be a set of i.i.d. copies of $Y\sim\mathbb{P}_{\theta^{*}}$. Suppose
$0<r\le R$ are radii of contraction and $\left(\overline{\gamma},\overline{\nu}\right)\in\mathcal{C}(r,R)$
is the optimal pair. If assumptions (\ref{A1}), (\ref{A2}) and
(\ref{A3}) hold true and the sample size $n$ is sufficiently large
such that
\begin{equation}
\varepsilon_{s}(\delta,r,R,n,p)+r\varepsilon_{1}(\delta,r,n,p)+r\varepsilon_{2}(\delta,r,R,n,p)<r\left(\overline{\nu}-\overline{\gamma}\right),\label{eq:cond-n-evt-bnd}
\end{equation}
then $\mathscr{E}_{n}\subseteq\left\{ \varpi\mid\overline{E}_{n}<r\left(\overline{V}_{n}-\overline{\Gamma}_{n}\right)\right\} $.
\end{lem}
\begin{proof}
By definition \prettyref{eq:evt-E_n}, under the event $\mathscr{E}_{n}$,
we have $\overline{\varGamma}_{n}\le\overline{\gamma}_{n}=\overline{\gamma}+\varepsilon_{1}(\delta,r,n,p)$,
$\overline{V}_{n}\ge\overline{\nu}_{n}=\overline{\nu}-\varepsilon_{2}(\delta,r,R,n,p)$
and $\overline{E}_{n}\le\varepsilon_{s}(\delta,r,R,n,p)$, then simple
calculation yields that
\begin{align*}
\overline{E}_{n} & \le\varepsilon_{s}(\delta,r,R,n,p)\stackrel{(a)}{<}r\left(\overline{\nu}-\overline{\gamma}\right)-r\varepsilon_{1}(\delta,r,n,p)-r\varepsilon_{2}(\delta,r,R,n,p)\\
 & =r\left(\overline{\nu}-\varepsilon_{2}(\delta,r,R,n,p)-\left(\overline{\gamma}+\varepsilon_{1}(\delta,r,n,p)\right)\right)\\
 & =r\left(\overline{\nu}_{n}-\overline{\gamma}_{n}\right)\le r\left(\overline{V}_{n}-\overline{\varGamma}_{n}\right)
\end{align*}
where $(a)$ follows from assumption \prettyref{eq:cond-n-evt-bnd}
and the result is proved.
\end{proof}
\begin{rem*}
We note \prettyref{eq:cond-n-evt-bnd} implies that $\overline{\nu}_{n}=\overline{\nu}-\varepsilon_{2}(\delta,r,R,n,p)>\overline{\gamma}_{n}+\frac{1}{r}\varepsilon_{s}(\delta,r,R,n,p)>0$.
\end{rem*}
Now we state and prove the main theorem.
\begin{thm}[Optimal Empirical Convergence Theorem]
\label{thm:oect}Let $\delta\in(0,1)$ and $\{Y_{k}\}_{k=1}^{n}$
be a set of i.i.d. copies of $Y\sim\mathbb{P}_{\theta^{*}}$. Suppose
$0<r\le R$ are radii of contraction and $\left(\overline{\gamma},\overline{\nu}\right)\in\mathcal{C}(r,R)$
is the optimal pair. If assumptions (\ref{A1}), (\ref{A2}) and
(\ref{A3}) hold true and the sample size $n$ is sufficiently large
such that
\begin{equation}
\varepsilon_{s}(\delta,r,R,n,p)+r\varepsilon_{1}(\delta,r,n,p)+r\varepsilon_{2}(\delta,r,R,n,p)<r(\overline{\nu}-\overline{\gamma}),\label{eq:cond-for-n}
\end{equation}
then given an initial point $\Theta_{n}^{0}\in B_{r}(\theta^{*})$,
the empirical EM sequence $\{\Theta_{n}^{t}\}_{t\ge0}$ such that
\[
\Theta_{n}^{t+1}\in\arg\max\{Q_{n}(\Theta'|\Theta_{n}^{t};\{Y_{k}\})\mid\Theta'\in B_{R}(\theta^{*})\}\text{ for }t\ge0
\]
satisfies the inequality
\begin{equation}
\left\Vert \Theta_{n}^{t}-\theta^{*}\right\Vert \le\left(\overline{K}_{n}\right)^{t}\left\Vert \Theta_{n}^{0}-\theta^{*}\right\Vert +\frac{\overline{E}_{n}}{\overline{V}_{n}-\overline{\Gamma}_{n}}\label{eq:opt-conv}
\end{equation}
with probability at least $1-\delta$.
\end{thm}
\begin{proof}
By \prettyref{lem:bounds} we have $\Pr\mathscr{E}_{n}\ge1-\delta$,
where
\begin{equation}
\mathscr{E}_{n}=\left\{ \varpi\mid\overline{\varGamma}_{n}\le\overline{\gamma}_{n},\overline{V}_{n}\ge\overline{\nu}_{n}\text{ and }\overline{E}_{n}\le\varepsilon_{s}(\delta,r,R,n,p)\right\} .\label{eq:cond-event}
\end{equation}
Conditioning on this event and by \prettyref{lem:evt-bound}, we have
$\overline{K}_{n}=\frac{\overline{\Gamma}_{n}}{\overline{V}_{n}}\le\overline{\kappa}_{n}<1$
and $\overline{E}_{n}<r\left(\overline{V}_{n}-\overline{\Gamma}_{n}\right)$.
Now we claim that:\emph{ for $t\in\mathbb{N}$, the empirical EM sequence
satisfies}
\begin{equation}
\left\Vert \Theta_{n}^{t+1}-\theta^{*}\right\Vert \le\overline{K}_{n}\left\Vert \Theta_{n}^{t}-\theta^{*}\right\Vert +\frac{\overline{E}_{n}}{\overline{V}_{n}}.\label{eq:int-1-oect}
\end{equation}

We prove this claim by induction. Note $\Theta_{n}^{0}\in B_{r}(\theta^{*})$
and $\Theta_{n}^{1}\in B_{R}(\theta^{*})$ by definition, and since
$Q_{n}(\Theta_{n}^{1}|\Theta_{n}^{0};\{Y_{k}\})\ge Q_{n}(\theta^{*}|\Theta_{n}^{0};\{Y_{k}\})$
and $\overline{\nu}_{n}>0$, it follows from \prettyref{prop:opt-ineq}
that
\[
\left\Vert \Theta_{n}^{1}-\theta^{*}\right\Vert \le\overline{K}_{n}\left\Vert \Theta_{n}^{0}-\theta^{*}\right\Vert +\frac{\overline{E}_{n}}{\overline{V}_{n}}\text{ and }\left\Vert \Theta_{n}^{1}-\theta^{*}\right\Vert <\frac{\overline{\Gamma}_{n}}{\overline{V}_{n}}\cdot r+\frac{\overline{E}_{n}}{\overline{V}_{n}}<r
\]
under the event $\mathscr{E}_{n}$. Hence \prettyref{eq:int-1-oect}
holds for $t=0$ and $\Theta_{n}^{1}\in B_{r}(\theta^{*})$.

Now assume \prettyref{eq:int-1-oect} holds for $t\ge0$ and $\Theta_{n}^{t+1}\in B_{r}(\theta^{*})$,
then for
\[
\Theta_{n}^{t+2}\in\arg\max\{Q_{n}(\Theta'|\Theta_{n}^{t+1};\{Y_{k}\})\mid\Theta'\in B_{R}(\theta^{*})\},
\]
we have $\Theta_{n}^{t+2}\in B_{R}(\theta^{*})$ and $Q_{n}(\Theta_{n}^{t+2}|\Theta_{n}^{t+1};\{Y_{k}\})\ge Q_{n}(\theta^{*}|\Theta_{n}^{t+1};\{Y_{k}\})$.
Then by \prettyref{prop:opt-ineq},
\[
\left\Vert \Theta_{n}^{t+2}-\theta^{*}\right\Vert \le\overline{K}_{n}\left\Vert \Theta_{n}^{t+1}-\theta^{*}\right\Vert +\frac{\overline{E}_{n}}{\overline{V}_{n}}\text{ and }\left\Vert \Theta_{n}^{t+2}-\theta^{*}\right\Vert <\frac{\overline{\Gamma}_{n}}{\overline{V}_{n}}\cdot r+\frac{\overline{E}_{n}}{\overline{V}_{n}}<r
\]
under the event $\mathscr{E}_{n}$. Hence \prettyref{eq:int-1-oect}
holds for $t+1$ and $\Theta_{n}^{t+2}\in B_{r}(\theta^{*})$. We
conclude that \prettyref{eq:int-1-oect} holds for all $t\in\mathbb{N}$
and the claim is proved.

Now it remains to show \prettyref{eq:opt-conv}. We proceed by induction
again. It clearly holds for $t=0$; assume it holds for $t\ge0$,
then by \prettyref{eq:int-1-oect} and the induction hypothesis,
\begin{align*}
\left\Vert \Theta_{n}^{t+1}-\theta^{*}\right\Vert  & \le\overline{K}_{n}\left\Vert \Theta_{n}^{t}-\theta^{*}\right\Vert +\frac{\overline{E}_{n}}{\overline{V}_{n}}\\
 & \le\overline{K}_{n}\left(\left(\overline{K}_{n}\right)^{t}\left\Vert \Theta_{n}^{0}-\theta^{*}\right\Vert +\frac{\overline{E}_{n}}{\overline{V}_{n}-\overline{\Gamma}_{n}}\right)+\frac{\overline{E}_{n}}{\overline{V}_{n}}\\
 & =\left(\overline{K}_{n}\right)^{t+1}\left\Vert \Theta_{n}^{0}-\theta^{*}\right\Vert +\frac{\overline{E}_{n}}{\overline{V}_{n}-\overline{\Gamma}_{n}}.
\end{align*}
Hence it holds for $t+1$ and by induction it holds for all $t\in\mathbb{N}$
and the proof is complete.
\end{proof}
\begin{rem*}
In view of definition \prettyref{eq:opt-emp-con-paras}, the random
variables $\overline{\varGamma}_{n},\overline{V}_{n}$ and hence $\overline{K}_{n}$
are data-adaptive. For each realization $\{y_{k}\}_{k=1}^{n}$ of
i.i.d. copies $\{Y_{k}\}_{k=1}^{n}$ of $Y\sim\mathbb{P}_{\theta^{*}}$,
the above theorem produces a realization $k_{n}$ of the optimal empirical
convergence rate $\overline{K}_{n}$. The sample EM sequence constructed
from the realization $Q_{n}(\theta'|\theta;\{y_{k}\})$ converges
geometrically at the rate of $k_{n}$. Hence $\overline{K}_{n}$ quantitatively
characterizes the propagation of the randomness from the underlying
data generating distribution $\mathbb{P}_{\theta^{*}}$ to the convergence
rate of the empirical EM sequence.
\end{rem*}
\begin{rem*}
Under the event $\mathscr{E}_{n}$, we have $\overline{\varGamma}_{n}\le\overline{\gamma}_{n}<\gamma_{n}$,
$\overline{V}_{n}\ge\overline{\nu}_{n}>\nu_{n}$ and $\overline{E}_{n}\le\varepsilon_{s}(\delta,r,R,n,p)$,
hence $\frac{\overline{E}_{n}}{\overline{V}_{n}-\overline{\Gamma}_{n}}\le\frac{\varepsilon_{s}(\delta,r,R,n,p)}{\nu_{n}-\gamma_{n}}$.
Then \prettyref{eq:opt-conv} implies that
\begin{equation}
\left\Vert \Theta_{n}^{t}-\theta^{*}\right\Vert \le\left(\overline{K}_{n}\right)^{t}\left\Vert \Theta_{n}^{0}-\theta^{*}\right\Vert +\frac{\varepsilon_{s}(\delta,r,R,n,p)}{\nu_{n}-\gamma_{n}}\label{eq:emp-conv}
\end{equation}
with probability at least $1-\delta$. This inequality is sometimes
more convenient when we apply the optimal empirical convergence theorem.
\end{rem*}
We note that since $\nu_{n}\to\nu$, $\gamma_{n}\to\gamma$, $\varepsilon_{s}(\delta,r,R,n,p)\to0$ and $\overline{K}_{n}\le\overline{\kappa}_{n}\ssearrow\overline{\kappa}$
as the sample size $n\to\infty$, then intuitively, the empirical
inequality \prettyref{eq:emp-conv} ``converges'' to the oracle
inequality in the form \prettyref{eq:ora-conv}, hence the limit of
an empirical EM sequence should give a consistent estimate to $\theta^{*}$.
Indeed, as a consequence of the self-consistency of the oracle $Q$-function
\prettyref{eq:exp-stat-err} and the above observation, we have the
following result.
\begin{thm}[Consistency of the EM algorithm]
\label{thm:cema}Suppose $\delta\in(0,1)$ and $0<r\le R$ are radii
of contraction, $(\gamma,\nu)\in\mathcal{C}(r,R)$. If assumptions
(\ref{A1}), (\ref{A2}) and (\ref{A3}) hold true, then there exists
an $N\in\mathbb{N}$ such that whenever the sample size $n>N$, for
each set $\{Y_{k}\}_{k=1}^{n}$ of i.i.d. copies of $Y\sim\mathbb{P}_{\theta^{*}}$
and the empirical EM sequence $\{\Theta_{n}^{t}\}_{t\ge0}$ therefrom,
if $\lim_{t\to\infty}\Theta_{n}^{t}\eqqcolon\widetilde{\Theta}_{n}\in\Omega$
for each $n>N$, then
\begin{equation}
\left\Vert \widetilde{\Theta}_{n}-\theta^{*}\right\Vert \le\frac{\varepsilon_{s}(\delta,r,R,n,p)}{\nu_{n}-\gamma_{n}}\label{eq:const-bnd}
\end{equation}
with probability at least $1-\delta$. Hence $\lim_{n\to\infty}\widetilde{\Theta}_{n}=\theta^{*}$
in probability.
\end{thm}
\begin{proof}
Let $N$ be the smallest $n$ such that condition \prettyref{eq:cond-for-n}
holds, then for each $n>N$, by \prettyref{thm:oect} and the fact
that $\overline{K}_{n}\le\overline{\kappa}_{n}<1$, the empirical
EM sequence $\{\Theta_{n}^{t}\}_{t\ge0}$ constructed from $\{Y_{k}\}_{k=1}^{n}$
satisfies
\[
\left\Vert \Theta_{n}^{t}-\theta^{*}\right\Vert \le\overline{\kappa}_{n}^{t}\left\Vert \Theta_{n}^{0}-\theta^{*}\right\Vert +\frac{\varepsilon_{s}(\delta,r,R,n,p)}{\nu_{n}-\gamma_{n}}\text{ for }t\in\mathbb{N},
\]
with probability at least $1-\delta$. Then let $t\to\infty$ in the
above inequality and notice $\Theta_{n}^{t}\to\widetilde{\Theta}_{n}$
as $t\to\infty$, we obtain \prettyref{eq:const-bnd}. Since for $\delta>0$,
$\varepsilon_{s}(\delta,r,R,n,p)\to0$ and $\nu_{n}-\gamma_{n}\to\nu-\gamma>0$
as $n\to\infty$, it follows from \prettyref{eq:const-bnd} that $\lim_{n\to\infty}\widetilde{\Theta}_{n}=\theta^{*}$
in probability.
\end{proof}
\begin{rem*}
The classical work of Wu \cite{Wu} proved that, under the unimodal assumption
and other regularity conditions on the log-likelihood function, the
sample EM sequence converges to the MLE. In this case, the statistical
consistency of the EM algorithm can be guaranteed by that of the MLE.
Balakrishnan et al. \cite{Balakrishnan} obtained convergence results
of the sample EM sequence to the statistical error ball of the true
population parameter $\theta^{*}$ in canonical models, which implies
the statistical consistency of the sample EM sequences in these cases.
In their work, the statistical error is characterized by the following
deviation bound
\[
\sup_{\theta\in B_{r}(\theta^{*})}\left\Vert M_{n}(\theta)-M(\theta)\right\Vert \le\varepsilon_{M}^{\text{unif}}(n,\delta).
\]
In the general case, we do not know whether this uniform deviation can be bounded
by an infinitesimal $\varepsilon_{M}^{\text{unif}}(n,\delta)$ as
the sample size $n\to\infty$, since it may not necessarily be true that
the sample $M$-operator $M_{n}(\theta)$ is the empirical mean and
the population $M$-operator $M(\theta)$ is the corresponding population
mean.

In our theory, the statistical error is characterized by the norm
of the empirical mean $\mathcal{E}_{n}(\{Y_{k}\})$ of $\mathcal{E}(Y_{k})=\nabla_{1}Q(\theta^{*}|\theta^{*};Y_{k})$
for $1\le k\le n$. Since $\mathbb{E}_{\theta^{*}}\mathcal{E}(Y)=\nabla_{1}Q_{*}(\theta^{*}|\theta^{*})=0$
by the self-consistency \prettyref{eq:self-const}, it is then guaranteed
that $\left\Vert \mathcal{E}_{n}(\{Y_{k}\})\right\Vert \le\varepsilon_{s}(\delta,r,R,n,p)\to0$
as $n\to\infty$. Hence the above theorem gives a theoretical guarantee for the consistency of the limit point $\widetilde{\Theta}_{n}$ of the empirical EM sequence not only for canonical models but also for the general case, and $\varepsilon_{s}(\delta,r,R,n,p)$ is exactly the convergence rate of the statistical error of the empirical EM sequence.
\end{rem*}
\begin{rem*}
We do not claim $\widetilde{\Theta}_{n}$ as an MLE or stationary
point of a log-likelihood function. Instead, we believe any point
within the statistical error ball of $\theta^{*}$ serves equivalently
as a consistent estimate. In practical applications, we do not even
need the \emph{well-defined} convergence of the empirical EM sequence
$\{\Theta_{n}^{t}\}_{t\ge0}$ to some point $\widetilde{\Theta}_{n}\in\Omega$.
Indeed, by \prettyref{eq:emp-conv} when the number of iterations
$T$ is sufficiently large, the optimization error would be so small
that any point $\Theta_{n}^{t}$ for $t\ge T$ is almost within the
statistical precision to $\theta^{*}$.
\end{rem*}

\subsubsection{\label{subsec:opt-emp-conv}The Optimal Rate Convergence Theorem}

Now we prove a non-asymptotic concentration bound for the optimal
empirical convergence rate $\overline{K}_{n}$ on the optimal oracle
convergence rate $\overline{\kappa}$, which then implies that $\overline{K}_{n}\to\overline{\kappa}$
in probability as the sample size $n\to\infty$.

We first characterize the concentration property of the empirical
contraction parameters on their population versions, which is the following
result on \emph{concentration of contraction parameters.}
\begin{prop}
\label{prop:gamma-nu-bound}Suppose assumptions (\ref{A1}), (\ref{A2})
and (\ref{A3}) hold true, $\delta\in(0,1)$ and that $\mathcal{G}(r)\times\mathcal{V}(r,R)\neq\varnothing$,
then
\[
\left|\overline{\varGamma}_{n}-\overline{\gamma}\right|\le\varepsilon_{1}(\delta,r,n,p)\text{ and }\left|\overline{V}_{n}-\overline{\nu}\right|\le\varepsilon_{2}(\delta,r,R,n,p)
\]
with probability at least $1-\delta$.
\end{prop}
\begin{proof}
In view of \prettyref{lem:bounds}, for given $(\gamma,\nu)\in\mathcal{G}(r)\times\mathcal{V}(r,R)\neq\varnothing$,
we have
\[
\overline{\varGamma}_{n}\le\gamma_{n}<+\infty\text{ and }\overline{V}_{n}\ge\nu_{n}>-\infty
\]
with probability at least $1-\delta$. Conditioning on this event
and by \prettyref{eq:gam-bar} and \prettyref{eq:opt-emp-con-paras}, we have
\begin{align*}
\left|\overline{\varGamma}_{n}-\overline{\gamma}\right| & \stackrel{(a)}{\le}\sup\left\{ \left|\frac{\left\Vert \Gamma_{n}(\theta;\{Y_{k}\})\right\Vert }{\left\Vert \theta-\theta^{*}\right\Vert }-\frac{\left\Vert \mathbb{E}_{\theta^{*}}\Gamma(\theta;Y)\right\Vert }{\left\Vert \theta-\theta^{*}\right\Vert }\right|\mid\theta\in B_{r}^{\times}(\theta^{*})\right\} \\
 & \le\sup\left\{ \frac{\left\Vert \Gamma_{n}(\theta;\{Y_{k}\})-\mathbb{E}_{\theta^{*}}\Gamma(\theta;Y)\right\Vert }{\left\Vert \theta-\theta^{*}\right\Vert }\mid\theta\in B_{r}^{\times}(\theta^{*})\right\} \\
 & \le\varepsilon_{1}(\delta,r,n,p),
\end{align*}
where $(a)$ follows from \prettyref{lem:fun-bounds}$\prettyref{enu:fun-bnd-a}$. Similarly, by \prettyref{eq:nu-bar} and \prettyref{eq:opt-emp-con-paras}, we have
\begin{align*}
\left|\overline{V}_{n}-\overline{\nu}\right| & \stackrel{(a)}{\le}\sup\left\{ \left|\frac{\mathbb{E}_{\theta^{*}}V(\theta'|\theta)}{\left\Vert \theta'-\theta^{*}\right\Vert ^{2}}-\frac{V_{n}(\theta'|\theta)}{\left\Vert \theta'-\theta^{*}\right\Vert ^{2}}\right|\mid\left(\theta',\theta\right)\in B_{R}^{\times}(\theta^{*})\times B_{r}(\theta^{*})\right\} \\
 & =\sup\left\{ \frac{\left|\mathbb{E}_{\theta^{*}}V(\theta'|\theta)-V_{n}(\theta'|\theta)\right|}{\left\Vert \theta'-\theta^{*}\right\Vert ^{2}}\mid\left(\theta',\theta\right)\in B_{R}^{\times}(\theta^{*})\times B_{r}(\theta^{*})\right\} \\
 & \le\varepsilon_{2}(\delta,r,R,n,p),
\end{align*}
where $(a)$ follows from \prettyref{lem:fun-bounds}$\prettyref{enu:fun-bnd-b}$, and the proof is complete.
\end{proof}
Now we state and prove the concentration theorem.
\begin{thm}[Optimal Rate Convergence Theorem]
\label{thm:orct}Suppose assumptions (\ref{A1}), (\ref{A2})
and (\ref{A3}) hold true, $\delta\in(0,1)$ and $0<r\le R$ are radii
of contraction, let $(\overline{\gamma},\overline{\nu})\in\mathcal{C}(r,R)$
be the optimal pair of the oracle convergence. If the sample size
$n$ is sufficiently large such that $\varepsilon_{2}(\delta,r,R,n,p)<\frac{1}{2}\overline{\nu}$,
then
\begin{equation}
\left|\overline{K}_{n}-\overline{\kappa}\right|\le\frac{2}{\overline{\nu}}\left(\varepsilon_{1}(\delta,r,n,p)+\overline{\kappa}\varepsilon_{2}(\delta,r,R,n,p)\right)\label{eq:kn-k}
\end{equation}
with probability at least $1-\delta$. Hence $\overline{K}_{n}\to\overline{\kappa}$
in probability as $n\to\infty$.
\end{thm}
\begin{proof}
In view of \prettyref{prop:gamma-nu-bound}, we have $\left|\overline{\varGamma}_{n}-\overline{\gamma}\right|\le\varepsilon_{1}(\delta,r,n,p)$
and $\left|\overline{V}_{n}-\overline{\nu}\right|\le\varepsilon_{2}(\delta,r,R,n,p)$
with probability at least $1-\delta$. Conditioning on this event,
we have
\begin{align*}
\left|\overline{K}_{n}-\overline{\kappa}\right| & =\frac{\left|\left(\overline{\varGamma}_{n}-\overline{\gamma}\right)\overline{\nu}+\left(\overline{\nu}-\overline{V}_{n}\right)\overline{\gamma}\right|}{\overline{V}_{n}\overline{\nu}}\\
 & \le\frac{1}{\overline{V}_{n}}\left(\left|\overline{\varGamma}_{n}-\overline{\gamma}\right|+\frac{\overline{\gamma}}{\overline{\nu}}\left|\overline{\nu}-\overline{V}_{n}\right|\right)\\
 & \le\frac{1}{\overline{V}_{n}}\left(\varepsilon_{1}(\delta,r,n,p)+\overline{\kappa}\varepsilon_{2}(\delta,r,R,n,p)\right),
\end{align*}
and since $\overline{V}_{n}\ge\overline{\nu}-\varepsilon_{2}(\delta,r,R,n,p)>\frac{1}{2}\overline{\nu}$,
the bound \prettyref{eq:kn-k} follows. Moreover for $\delta>0$,
we have $\varepsilon_{1}(\delta,r,n,p)\to0$ and $\varepsilon_{2}(\delta,r,R,n,p)\to0$
as $n\to\infty$, it follows from \prettyref{eq:kn-k} that $\lim_{n\to\infty}\overline{K}_{n}=\overline{\kappa}$
in probability.
\end{proof}
In view of definition \prettyref{eq:def-Kn} and the theorem above,
$\overline{K}_{n}$ is upper bounded by $\overline{\kappa}_{n}$ and
concentrated on the optimal oracle convergence rate $\overline{\kappa}<\overline{\kappa}_{n}$.
The relationship of the real numbers $\overline{\kappa}_{n}$, $\overline{\kappa}$
and the random variable $\overline{K}_{n}$ can be illustrated in
the following schematic diagram,

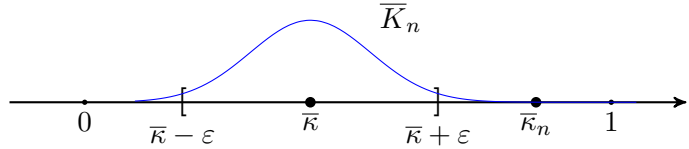
\begin{figure}[H]
\begin{tikzpicture}[scale=1]

\draw[->,thick] (-1,0)--(8,0);

\fill (7,0) circle (1pt) node[below]{$1$};
\fill (0,0) circle (1pt) node[below]{$0$};
\fill (6,0) circle (2pt) node[below]{$\overline{\kappa}_n$}; \fill (3,0) circle (2pt) node[below]{$\overline{\kappa}$}; \node  at (1.3,0) {$\boldsymbol{[}$};
\node  at (1.3,0) [below=5pt] {$\overline{\kappa}-\varepsilon$};
\node  at (4.7,0) [below=5pt] {$\overline{\kappa}+\varepsilon$};
\node  at (4.7,0) {$\boldsymbol{]}$};
\node[above] at (4.2,0.8) {$\overline{K}_n$};

%\draw[->] (2.7,0.6)--(2,0.2);
%\draw[->] (3.3,0.6)--(4,0.2);
\begin{axis}
[every axis plot post/.append style={mark=none,domain=0:5,samples=50,smooth},
height=1.2cm,
width=8cm,
scale only axis,
axis x line*=bottom,
% no box around the plot, only x and y axis
axis y line*=left,
% the * suppresses the arrow tips
%enlargelimits=upper,
% extend the axes a bit to the right and top
hide axis,
ymin=0,
%ymax=0.4
]

\addplot {gauss(1.75,0.6)};
%\addplot {gauss(1,0.75)};
\end{axis}
%%%%%%%%%%%%%%%%%%

\end{tikzpicture} \centering \caption{Concentration and upper bound of $\overline{K}_{n}$}
\label{figure:K-kap}
\end{figure}
\noindent where $\varepsilon\to0$ and $\overline{\kappa}_{n}\ssearrow\overline{\kappa}$
as $n\to\infty$, hence the distribution of the random variable $\overline{K}_{n}$
collapses on $\overline{\kappa}$ when the sample size $n$ is sufficiently
large.

\section{\label{sec:app-cla-mod}Applications to Canonical Models}

In this section, we apply our theory to the EM algorithm on three
canonical models: the Gaussian Mixture Model, the Mixture of Linear
Regressions and the Regression with Missing Covariates to obtain specific
results for these models.
\begin{description}
\item [{Notations}] The following notations are used throughout this section:
\begin{itemize}
\item We use $c,C,C_{1},C_{2},C_{3}\cdots$ to denote a numerical constant.
\item For $\theta^{*}\neq0$, let $\eta\coloneqq\frac{\left\Vert \theta^{*}\right\Vert }{\sigma}$
be the \emph{signal to noise ratio} (SNR). Let $\omega\coloneqq\frac{r}{\left\Vert \theta^{*}\right\Vert }$
be the \emph{relative contraction radius} (RCR) and let $K\coloneqq\sigma+\left\Vert \theta^{*}\right\Vert =\sigma\left(1+\eta\right)$.
\item Let $L=\mathcal{N}_{\frac{1}{2}}\left(\mathbb{S}^{p-1}\right)<5^{p}$
be the $\frac{1}{2}$-covering number of $\mathbb{S}^{p-1}$. (see
\prettyref{subsec:cov-num})
\item Let $\phi\left(x;\mu,\Sigma\right)$ be the density function of the
multivariate normal distribution $\mathcal{N}\left(\mu,\Sigma\right)$,
where $\mu\in\mathbb{R}^{p}$ and $\Sigma\in\mathbb{R}^{p\times p}$.
\end{itemize}
\end{description}

\subsection{\label{subsec:gau-mix}Gaussian Mixture Model}

Consider the balanced symmetric Gaussian mixture model
\[
Y=Z\cdot\theta^{*}+W,
\]
where $Z$ is a Rademacher random variable, $W\sim\mathcal{N}(0,\sigma^{2}I_{p})$
is the Gaussian noise with variance $\sigma^{2}$ and $\theta^{*}\in\mathbb{R}^{p}\ (p\ge1)$.
Suppose $Y$ is observed and $Z$ is a latent variable, the complete
joint density of $(Y,Z)$ is
\[
f_{\theta^{*}}(y,z)=\frac{1}{2}\phi(y-z\cdot\theta^{*};0,\sigma^{2}I_{p}),
\]
and marginalization over $Z$ gives the density of $Y$ as a Gaussian
mixture
\[
g_{\theta^{*}}(y)=\frac{1}{2}\phi(y-\theta^{*};0,\sigma^{2}I_{p})+\frac{1}{2}\phi(y+\theta^{*};0,\sigma^{2}I_{p}).
\]

Suppose a set of i.i.d. realizations $\{y_{k}\}_{k=1}^{n}$ of $Y$
are observed from the mixture density, the goal is to estimate
the unknown true population parameter $\theta^{*}\in\Omega\subseteq\mathbb{R}^{p}$,
while the variance $\sigma^{2}$ is assumed known.

Standard calculation of the EM algorithm yields the stochastic $Q$-function
\[
Q(\theta'|\theta;y)=-\frac{1}{2\sigma^{2}}\left(w_{\theta}(y)\left\Vert y-\theta'\right\Vert ^{2}+(1-w_{\theta}(y))\left\Vert y+\theta'\right\Vert ^{2}\right)-\log2\left(\sqrt{2\pi}\sigma\right)^{p}
\]
where $w_{\theta}(y)\coloneqq\varsigma\left(\frac{2\theta^{\intercal}y}{\sigma^{2}}\right)$,
and $\varsigma(t)\coloneqq\frac{1}{1+e^{-t}}$ is the logistic function.
Then the gradient
\[
\nabla_{1}Q(\theta'|\theta;y)=\frac{1}{\sigma^{2}}\left[(2w_{\theta}(y)-1)y-\theta'\right],
\]
and hence the GRV
\begin{align}
\Gamma(\theta;Y) & =\nabla_{1}Q(\theta^{*}|\theta;Y)-\nabla_{1}Q(\theta^{*}|\theta^{*};Y)\nonumber \\
 & =\frac{2}{\sigma^{2}}\left[w_{\theta}(Y)-w_{\theta^{*}}(Y)\right]Y.\label{eq:Gam-gmm}
\end{align}
Since $Q(\theta'|\theta;y)$ is quadratic in $\theta'$, the CRV can
be computed as
\begin{equation}
V(\theta'|\theta;Y)=-\frac{1}{2\sigma^{2}}\left\Vert \theta'-\theta^{*}\right\Vert ^{2}\label{eq:V-gmm}
\end{equation}
by \prettyref{lem:quad-taylor}. Then the SEV
\begin{equation}
\mathcal{E}(Y)=\frac{1}{\sigma^{2}}\left[(2w_{\theta^{*}}(Y)-1)Y-\theta^{*}\right].\label{eq:E-gmm}
\end{equation}

\subsubsection{\label{subsec:ora-cov-gmm}Oracle Convergence}

We first characterize the sets $\mathcal{G}(r)$ and $\mathcal{V}(r,R)$.
It is clear from \prettyref{eq:V-gmm} that
\[
\mathbb{E}_{\theta^{*}}V(\theta'|\theta;Y)=-\frac{1}{2\sigma^{2}}\left\Vert \theta'-\theta^{*}\right\Vert ^{2},
\]
and hence $\mathcal{V}(r,R)=\left(0,\overline{\nu}\right]$, where
$\overline{\nu}=\frac{1}{2\sigma^{2}}$ for any $0<r\le R$. As for
the set $\mathcal{G}(r)$ we need to bound
\[
\mathbb{E}_{\theta^{*}}\Gamma(\theta;Y)=\frac{2}{\sigma^{2}}\mathbb{E}_{\theta^{*}}\left[(w_{\theta}(Y)-w_{\theta^{*}}(Y))Y\right].
\]
To this end, we cite the following technical result from \cite{Balakrishnan}
(Lemma 2).
\begin{lem}
If $\theta^{*}\neq0$ and the signal to noise ratio $\eta$ is sufficiently
large, then
\[
\left\Vert \frac{2}{\sigma^{2}}\mathbb{E}_{\theta^{*}}\left[(w_{\theta}(Y)-w_{\theta^{*}}(Y))Y\right]\right\Vert \le\gamma\left(\eta\right)\left\Vert \theta-\theta^{*}\right\Vert \text{ for }\theta\in B_{r}(\theta^{*}),
\]
where $r=\frac{\left\Vert \theta^{*}\right\Vert }{4}$ and $\gamma\left(\eta\right)\coloneqq\frac{1}{\sigma^{2}}e^{-c\eta^{2}}$.
\end{lem}
Hence for $r=\frac{\left\Vert \theta^{*}\right\Vert }{4}$, we have
$\mathcal{G}(r)=\left[\overline{\gamma},+\infty\right)$ where $\overline{\gamma}\le\gamma\left(\eta\right)$.
It is clear that $\gamma\left(\eta\right)<\overline{\nu}=\frac{1}{2\sigma^{2}}$
when $\eta$ is sufficiently large. In that case, $0<r<+\infty$ are
radii of contraction and $\left(\gamma\left(\eta\right),\overline{\nu}\right)\in\mathcal{C}\left(r,+\infty\right)\neq\varnothing$.
We then apply the oracle convergence theorem to get the following
result for the Gaussian Mixture Model.
\begin{cor}
For the Gaussian Mixture Model, if $\eta$ is sufficiently large such
that $\kappa\coloneqq2e^{-c\eta^{2}}<1$, then $0<r<+\infty$ where
$r=\frac{\left\Vert \theta^{*}\right\Vert }{4}$, are radii of contraction.
For each pair $\left(\gamma\left(\eta\right),\frac{1}{2\sigma^{2}}\right)\in\mathcal{C}\left(r,+\infty\right)\neq\varnothing$
and initial point $\theta^{0}\in B_{r}(\theta^{*})$, any oracle EM
sequence $\{\theta^{t}\}_{t\ge0}$ such that
\[
\theta^{t+1}\in\arg\max_{\theta'\in\Omega}Q_{*}(\theta'|\theta^{t})\text{ for }t\ge0
\]
satisfies the inequality
\begin{equation}
\left\Vert \theta^{t}-\theta^{*}\right\Vert \le\overline{\kappa}^{t}\left\Vert \theta^{0}-\theta^{*}\right\Vert ,\label{eq:ora-conv-1}
\end{equation}
where $\overline{\kappa}\coloneqq\frac{\overline{\gamma}}{\overline{\nu}}\le\kappa<1$,
is the optimal oracle convergence rate.
\end{cor}

\subsubsection{Empirical Convergence}

For empirical convergence results, we need to find specific forms
of the $\varepsilon$-bounds in the Assumptions.
\begin{lem}
\label{lem:epsilon-1-gmm}For $\delta\in(0,1)$ and $r>0$, if $n>c\log\left(L/\delta\right)$,
then\footnote{Note $\log(L/\delta)\le O(p)$, see \prettyref{subsec:cov-num}.}
\begin{equation}
\left\Vert \Gamma_{n}(\theta;\{Y_{k}\})-\mathbb{E}_{\theta^{*}}\Gamma(\theta;Y)\right\Vert \le C\frac{K^{2}}{\sigma^{2}}\sqrt{\frac{\log(L/\delta)}{n}}\left\Vert \theta-\theta^{*}\right\Vert \text{ for }\theta\in B_{r}(\theta^{*})\label{eq:epsilon-1-gmm}
\end{equation}
with probability at least $1-\delta$.
\end{lem}
\begin{proof}
See \prettyref{subsec:pf-eps-1-gmm}.
\end{proof}
Secondly, in view of \prettyref{eq:V-gmm} we have
\begin{equation}
V(\theta'|\theta;Y)=V_{n}(\theta'|\theta;\{Y_{k}\})=\mathbb{E}_{\theta^{*}}V(\theta'|\theta;Y)=-\frac{1}{2\sigma^{2}}\left\Vert \theta'-\theta^{*}\right\Vert ^{2},\label{eq:epsilon-2-gmm}
\end{equation}
and hence $\varepsilon_{2}(\delta,r,R,n,p)=0$. Thirdly, for the $\varepsilon$-bound
on statistical error, we have the following result.
\begin{lem}
\label{lem:epsilon-s-gmm}For $\delta\in(0,1)$, there holds the inequality
\begin{equation}
\left\Vert \mathcal{E}_{n}(\{Y_{k}\})\right\Vert \le C\frac{K}{\sigma^{2}}\sqrt{\frac{\log(L/\delta)}{n}}\label{eq:epsilon-s-gmm}
\end{equation}
with probability at least $1-\delta$.
\end{lem}
\begin{proof}
See \prettyref{subsec:proof-of-epsilon-s}.
\end{proof}
From \prettyref{eq:epsilon-1-gmm} and \prettyref{eq:epsilon-s-gmm}
in the above lemmas, it is clear that the Assumptions (\ref{A1}$\sim$\ref{A3})
are satisfied with the following
\[
\varepsilon_{1}(\delta,r,n,p)=C_{1}\frac{K^{2}}{\sigma^{2}}\sqrt{\frac{\log(L/\delta)}{n}},\quad\varepsilon_{2}(\delta,r,R,n,p)=0\quad\text{and}
\]
\begin{equation}
\varepsilon_{s}(\delta,r,R,n,p)=C_{2}\frac{K}{\sigma^{2}}\sqrt{\frac{\log(L/\delta)}{n}}.\label{eq:all-eps-gmm}
\end{equation}
We obtain the following data-adaptive empirical convergence result
for the Gaussian Mixture Model.
\begin{cor}
Suppose $\{Y_{k}\}_{k=1}^{n}$ is a set of i.i.d. copies of the Gaussian
mixture $Y\sim g_{\theta^{*}}(y)$ and $\delta\in(0,1)$. If $\eta$
is sufficiently large such that $\kappa\coloneqq2e^{-c\eta^{2}}<1$
and the sample size $n$ satisfies
\begin{equation}
n>\frac{\log(L/\delta)}{\left(1-\kappa\right)^{2}}\left(C_{1}K^{2}+C_{2}\left(1+\frac{1}{\eta}\right)\right)^{2},\label{eq:cond-for-n-gmm}
\end{equation}
then for $r=\frac{\left\Vert \theta^{*}\right\Vert }{4}$ and an initial
point $\theta^{0}\in B_{r}(\theta^{*})$, with probability at least
$1-\delta$, the empirical EM sequence $\{\Theta_{n}^{t}\}_{t\ge0}$
such that
\[
\Theta_{n}^{t+1}\in\arg\max\{Q_{n}(\Theta'|\Theta_{n}^{t};\{y_{k}\}\mid\Theta'\in\mathbb{R}^{p})\text{ for }t\ge0,
\]
satisfies the inequality
\begin{equation}
\left\Vert \Theta_{n}^{t}-\theta^{*}\right\Vert \le\left(\overline{K}_{n}\right)^{t}\left\Vert \Theta_{n}^{0}-\theta^{*}\right\Vert +\frac{C_{3}K}{1-\kappa_{n}}\sqrt{\frac{\log(L/\delta)}{n}},\label{eq:emp-conv-gmm}
\end{equation}
where the optimal empirical convergence rate $\overline{K}_{n}$ satisfies
that
\[
\overline{K}_{n}\le\kappa+C_{4}K^{2}\sqrt{\frac{\log(L/\delta)}{n}}<1\text{ and }\left|\overline{K}_{n}-\overline{\kappa}\right|\le CK^{2}\sqrt{\frac{\log(L/\delta)}{n}}.
\]
\end{cor}
\begin{proof}
We first note condition \prettyref{eq:cond-for-n-gmm} implies the
lower bound of $n$ in \prettyref{lem:epsilon-1-gmm}. Hence to apply
the empirical convergence theorem, we only need to verify that
\[
\varepsilon_{s}(\delta,r,R,n,p)+r\varepsilon_{1}(\delta,r,n,p)+r\varepsilon_{2}(\delta,r,R,n,p)<r\left(\overline{\nu}-\gamma\left(\eta\right)\right)
\]
holds true whenever $n$ satisfies \prettyref{eq:cond-for-n-gmm},
but this is trivial.

Then note $\varepsilon_{2}(\delta,r,R,n,p)=0$ and $\overline{\nu}_{n}=\frac{1}{2\sigma^{2}}$,
hence the concentration bound of $\overline{K}_{n}$ follows from
the optimal rate convergence theorem.
\end{proof}

\subsection{\label{subsec:mix-lin-reg}Mixture of Linear Regressions}

Consider the Mixture of Linear Regressions model with two balanced
symmetric components in which the covariate-response $(Y,X)$ are
linked via
\begin{equation}
Y=\left\langle X,Z\cdot\theta^{*}\right\rangle +W,\label{eq:mlr-model}
\end{equation}
where $Z$ is a Rademacher variable, $W\sim\mathcal{N}(0,\sigma^{2})$
is the Gaussian noise and $X\sim\mathcal{N}(0,I_{p})$ is a Gaussian
covariate. Given a set of i.i.d realizations $\{(y_{k},x_{k})\}_{k=1}^{n}$
generated by \prettyref{eq:mlr-model}, the goal is to estimate
the unknown population parameter $\theta^{*}\in\mathbb{R}^{p}$.

In above setting, we observe the covariate-response pair $(Y,X)$
while $Z$ is a latent variable. The complete joint density is
\[
f_{\theta}(y,x,z)=\frac{1}{2}\phi(y-\left\langle x,z\cdot\theta\right\rangle ;0,\sigma^{2})\phi(x;0,I_{p}),
\]
where $(y,x)\in\mathbb{R}\times\mathbb{R}^{p}$ and $z\in\{-1,1\}$.
Then it is a standard procedure to obtain the stochastic $Q$-function
\begin{align*}
Q(\theta'|\theta;(y,x)) & =-\frac{1}{2\sigma^{2}}\left(w_{\theta}(y,x)\left(y-\left\langle x,\theta'\right\rangle \right)^{2}+(1-w_{\theta}(y,x))\left(y+\left\langle x,\theta'\right\rangle \right)^{2}\right)\\
 & -\frac{1}{2}\left\Vert x\right\Vert ^{2}-\log2\left(\sqrt{2\pi}\sigma\right)^{p+1},
\end{align*}
where $w_{\theta}(y,x)\coloneqq\varsigma\left(\frac{2y\left\langle x,\theta\right\rangle }{\sigma^{2}}\right)$
and $\varsigma(t)\coloneqq\frac{1}{1+e^{-t}}$ is the logistic function.
Then the gradient
\[
\nabla_{1}Q(\theta'|\theta;(y,x))=\frac{1}{\sigma^{2}}\left[(2w_{\theta}(y,x)-1)y-\left\langle x,\theta'\right\rangle \right]x,
\]
and hence the GRV
\begin{align}
\Gamma(\theta;(Y,X)) & =\nabla_{1}Q(\theta^{*}|\theta;(Y,X))-\nabla_{1}Q(\theta^{*}|\theta^{*};(Y,X))\nonumber \\
 & =\frac{2}{\sigma^{2}}\left[w_{\theta}(Y,X)-w_{\theta^{*}}(Y,X)\right]YX.\label{eq:Gam-mlr}
\end{align}
Since $Q(\theta'|\theta;(y,x))$ is quadratic in $\theta'$, the CRV
can be computed as
\begin{equation}
V(\theta'|\theta;(Y,X))=-\frac{1}{2\sigma^{2}}(\theta'-\theta^{*})^{\intercal}XX^{\intercal}(\theta'-\theta^{*})\label{eq:V-mlr}
\end{equation}
by \prettyref{lem:quad-taylor}. Then the SEV
\begin{equation}
\mathcal{E}(Y,X)=\frac{1}{\sigma^{2}}\left[(2w_{\theta^{*}}(Y,X)-1)Y-\left\langle X,\theta^{*}\right\rangle \right]X.\label{eq:E-mlr}
\end{equation}

\subsubsection{\label{subsec:ora-cov-mlr}Oracle Convergence}

We first characterize the sets $\mathcal{G}(r)$ and $\mathcal{V}(r,R)$.
Since $X\sim\mathcal{N}(0,I_{p})$ and by \prettyref{eq:V-mlr},
\begin{equation}
\mathbb{E}_{\theta^{*}}V(\theta'|\theta;(Y,X))=-\frac{1}{2\sigma^{2}}\left\Vert \theta'-\theta^{*}\right\Vert ^{2},\label{eq:EV-mlr}
\end{equation}
hence $\mathcal{V}(r,R)=\left(0,\overline{\nu}\right]$, where $\overline{\nu}=\frac{1}{2\sigma^{2}}$
for any $0<r\le R$. As for the set $\mathcal{G}(r)$ we need to bound
\[
\mathbb{E}_{\theta^{*}}\Gamma(\theta;(Y,X))=\frac{2}{\sigma^{2}}\mathbb{E}_{\theta^{*}}\left[\left(w_{\theta}(Y,X)-w_{\theta^{*}}(Y,X)\right)YX\right].
\]
To this end, we cite a technical result from \cite{Yi} (Lemma 7 in
the Supplement), see also Lemma 3 in \cite{Balakrishnan} for an alternative.
\begin{lem}
If $\omega\in\left(0,\frac{1}{4}\right]$ and $r=\omega\left\Vert \theta^{*}\right\Vert $,
then
\[
\left\Vert \mathbb{E}_{\theta^{*}}\Gamma(\theta;(Y,X))\right\Vert \le\gamma\left(\omega,\eta\right)\left\Vert \theta-\theta^{*}\right\Vert \text{ for }\theta\in B_{r}(\theta^{*}),
\]
where $\gamma\left(\omega,\eta\right)\coloneqq\frac{1}{\sigma^{2}}\left(7.3\omega+\frac{17}{\eta}\right)$.
\end{lem}
Hence $\mathcal{G}(r)=\left[\overline{\gamma},+\infty\right)$, where
$\overline{\gamma}\le\gamma\left(\omega,\eta\right)$ for $r=\omega\left\Vert \theta^{*}\right\Vert $
and $\omega\in\left(0,\frac{1}{4}\right]$. It is clear that $\gamma\left(\omega,\eta\right)<\overline{\nu}=\frac{1}{2\sigma^{2}}$,
when $\eta$ is sufficiently large and $\omega$ is sufficiently small.
In that case, $0<\omega\left\Vert \theta^{*}\right\Vert <+\infty$
are radii of contraction and $\left(\gamma\left(\omega,\eta\right),\overline{\nu}\right)\in\mathcal{C}(\omega\left\Vert \theta^{*}\right\Vert ,+\infty)\neq\varnothing$.
We then apply the oracle convergence theorem to get the following
corollary.
\begin{cor}
For the Mixture of Linear Regressions model, if $\eta$ is sufficiently
large and $\omega$ is sufficiently small such that $7.3\omega+\frac{17}{\eta}<\frac{1}{2}$,
then $0<r<+\infty$ where $r=\omega\left\Vert \theta^{*}\right\Vert $
are radii of contraction. For each pair $\left(\gamma\left(\omega,\eta\right),\frac{1}{2\sigma^{2}}\right)\in\mathcal{C}(r,+\infty)$
and initial point $\theta^{0}\in B_{r}(\theta^{*})$, any oracle EM
sequence $\{\theta^{t}\}_{t\ge0}$ such that
\[
\theta^{t+1}\in\arg\max_{\theta'\in\Omega}Q_{*}(\theta'|\theta^{t})\text{ for }t\ge0,
\]
satisfies the inequality
\begin{equation}
\left\Vert \theta^{t}-\theta^{*}\right\Vert \le\overline{\kappa}^{t}\left\Vert \theta^{0}-\theta^{*}\right\Vert ,\label{eq:ora-conv-mlr}
\end{equation}
where $\overline{\kappa}\coloneqq\frac{\overline{\gamma}}{\overline{\nu}}\le2\left(7.3\omega+\frac{17}{\eta}\right)<1$,
is the optimal oracle convergence rate.
\end{cor}

\subsubsection{Empirical Convergence}

For empirical convergence results, we need to find specific forms
of the $\varepsilon$-bounds in the Assumptions.
\begin{lem}
\label{lem:epsilon-1-mlr}For $\delta\in(0,1)$ and $r>0$, there
holds the inequality
\begin{equation}
\left\Vert \Gamma_{n}(\theta;\{(Y_{k},X_{k})\})-\mathbb{E}_{\theta^{*}}\Gamma(\theta;(Y,X))\right\Vert \le C\frac{\log(L/\delta)}{\sigma^{2}n^{\frac{1}{2}-\epsilon}}\left\Vert \theta-\theta^{*}\right\Vert \text{ for }\theta\in B_{r}(\theta^{*})\label{eq:epsilon-1-mlr}
\end{equation}
with probability at least $1-\delta$.
\end{lem}
\begin{proof}
See \prettyref{subsec:pf-eps-1-mlr}.
\end{proof}
\begin{lem}
\label{lem:epsilon-2-mlr}For $\delta\in(0,1)$ and $r>0$, if $n>c\log\left(1/\delta\right)$
then
\begin{equation}
\left|V_{n}(\theta'|\theta;\{(Y_{k},X_{k})\})-\mathbb{E}_{\theta^{*}}V(\theta'|\theta;(Y,X))\right|\le\frac{C}{\sigma^{2}}\sqrt{\frac{\log(1/\delta)}{n}}\left\Vert \theta'-\theta^{*}\right\Vert ^{2}\text{ for }\theta',\theta\in B_{r}(\theta^{*})\label{eq:epsilon-2-mlr}
\end{equation}
with probability at least $1-\delta$.
\end{lem}
\begin{proof}
See \prettyref{subsec:pf-eps-2-mlr}.
\end{proof}
\begin{lem}
\label{lem:epsilon-s-mlr}For $\delta\in(0,1)$ and $n>c\log\left(L/\delta\right)$,
there holds the inequality
\begin{equation}
\left\Vert \mathcal{E}_{n}(\{(Y_{k},X_{k})\})\right\Vert \le\frac{C}{\sigma}\left(1+2\eta\right)\sqrt{\frac{\log(L/\delta)}{n}}\label{eq:epsilon-s-mlr}
\end{equation}
with probability at least $1-\delta$.
\end{lem}
\begin{proof}
See \prettyref{subsec:pf-eps-s-mlr}.
\end{proof}
From \prettyref{eq:epsilon-1-mlr}, \prettyref{eq:epsilon-2-mlr}
and \prettyref{eq:epsilon-s-mlr} in above lemmas, it is clear that
the Assumptions (\ref{A1}$\sim$\ref{A3}) are satisfied with the
following
\[
\varepsilon_{1}(\delta,r,n,p)=\frac{C_{1}}{\sigma^{2}}\frac{\log(L/\delta)}{n^{\frac{1}{2}-\epsilon}},\quad\varepsilon_{2}(\delta,r,R,n,p)=\frac{C_{2}}{\sigma^{2}}\sqrt{\frac{\log(1/\delta)}{n}}\quad\text{and}
\]
\begin{equation}
\varepsilon_{s}(\delta,r,R,n,p)=\frac{C_{3}}{\sigma}\left(1+2\eta\right)\sqrt{\frac{\log(L/\delta)}{n}}.\label{eq:all-eps-mlr}
\end{equation}
We obtain the following data-adaptive empirical convergence result
for the Mixture of Linear Regressions model.
\begin{cor}
Suppose $\{(Y_{k},X_{k})\}_{k=1}^{n}$ is a set of i.i.d. copies of
$(Y,X)\sim\mathbb{P}_{\theta^{*}}$ and $\delta\in(0,1)$. If $\eta$
is sufficiently large and $\omega$ is sufficiently small such that
$7.3\omega+\frac{17}{\eta}<\frac{1}{2}$ and the sample size $n$
satisfies
\begin{equation}
n>\frac{\log(1/\delta)}{\left(1-\kappa\right)^{2}}\left(C_{1}\sqrt{\log(1/\delta)}\cdot pn^{\epsilon}+C_{2}+C_{3}\frac{\left(1+2\eta\right)}{\omega\eta}\sqrt{p}\right)^{2}\label{eq:cond-for-n-mlr}
\end{equation}
where $\kappa\coloneqq\frac{\gamma\left(\omega,\eta\right)}{\overline{\nu}}=2\left(7.3\omega+\frac{17}{\eta}\right)$
and $r=\omega\left\Vert \theta^{*}\right\Vert $. Then given an initial
point $\Theta_{n}^{0}\in B_{r}(\theta^{*})$, with probability at
least $1-\delta$, any empirical EM sequence $\{\Theta_{n}^{t}\}_{t\ge0}$
such that
\[
\Theta_{n}^{t+1}\in\arg\max_{\theta'\in\Omega}Q_{n}(\Theta'|\Theta_{n}^{t};\{(Y_{k},X_{k})\})\text{ for }t\ge0,
\]
satisfies the inequality
\begin{equation}
\left\Vert \Theta_{n}^{t}-\theta^{*}\right\Vert \le\left(\overline{K}_{n}\right)^{t}\left\Vert \Theta_{n}^{0}-\theta^{*}\right\Vert +\frac{C_{3}\left(1+2\eta\right)}{\sigma\left(\nu_{n}-\gamma_{n}\right)}\sqrt{\frac{\log(L/\delta)}{n}}\label{eq:emp-conv-mlr}
\end{equation}
where
\[
\gamma_{n}\coloneqq\gamma\left(\omega,\eta\right)+\frac{C_{1}}{\sigma^{2}}\frac{\log(L/\delta)}{n^{\frac{1}{2}-\epsilon}},\quad\nu_{n}\coloneqq\frac{1}{2\sigma^{2}}-\frac{C_{2}}{\sigma^{2}}\sqrt{\frac{\log(1/\delta)}{n}},
\]
and the optimal empirical convergence rate $\overline{K}_{n}$ satisfies
that
\[
\overline{K}_{n}\le\frac{\gamma_{n}}{\nu_{n}}<1\text{ and }\left|\overline{K}_{n}-\overline{\kappa}\right|\le\left(C_{1}\sqrt{\log(1/\delta)}\cdot pn^{\epsilon}+C_{2}\kappa\right)\sqrt{\frac{\log(1/\delta)}{n}}.
\]
\end{cor}
\begin{proof}
Note condition \prettyref{eq:cond-for-n-mlr} implies the lower bounds
of $n$ in \prettyref{lem:epsilon-2-mlr} and \prettyref{lem:epsilon-s-mlr}.
Hence to apply the empirical convergence theorem, we only need to
verify that
\[
\varepsilon_{s}(\delta,r,R,n,p)+r\varepsilon_{1}(\delta,r,n,p)+r\varepsilon_{2}(\delta,r,R,n,p)<r\left(\overline{\nu}-\gamma\left(\omega,\eta\right)\right)
\]
holds true whenever $n$ satisfies \prettyref{eq:cond-for-n-mlr},
which is not difficult noticing that $\log(L/\delta)<Cp\log(1/\delta)$.

Moreover, \prettyref{eq:cond-for-n-mlr} also implies that $\varepsilon_{2}(\delta,r,R,n,p)=\frac{C_{2}}{\sigma^{2}}\sqrt{\frac{\log(1/\delta)}{n}}<\frac{1}{2}\overline{\nu}=\frac{1}{4\sigma^{2}}$,
hence the concentration bound of $\overline{K}_{n}$ follows from
the optimal rate convergence theorem.
\end{proof}
\begin{rem*}
In view of \prettyref{eq:V-mlr} and the definitions in \prettyref{eq:opt-emp-con-paras},
we find $\overline{V}_{n}=\frac{1}{2\sigma^{2}}\lambda_{n}^{\min}$,
where $\lambda_{n}^{\min}$ is the smallest eigenvalue of the empirical
mean of the Gaussian covariance matrix $XX^{\intercal}$, while we
do not know a closed form for $\overline{\Gamma}_{n}$ or $\overline{K}_{n}$
yet in this model.
\end{rem*}

\subsection{\label{subsec:lin-reg-mis-cov}Linear Regression with Missing Covariates}

Consider the linear regression in which the covariate-response $(Y,X)$
are linked via
\begin{equation}
Y=\left\langle \theta^{*},X\right\rangle +W,\label{eq:lr-model}
\end{equation}
where $W\sim\mathcal{N}(0,\sigma^{2})$ is the observational noise
and the covariate $X\sim\mathcal{N}(0,I_{p})$ under Gaussian design.
Instead of observing the complete data $(Y,X)$, we have each coordinate
$X^{j}\ (j=1,\cdots,p)$ of the covariate missing completely at random
with a probability $\epsilon\in[0,1)$.

It is not difficult to see that there is a 1-to-1 correspondence between
the set of missing patterns and the set of binary vectors $\{\circ,1\}^{p}$,
where $\circ$ is just $0$ written differently and hence $\circ\cdot a=\circ$
and $\circ+a=a$ for $a\in\mathbb{R}$.\footnote{This notational variation is necessary to distinguish missing coordinates
from those having value $0$.} Indeed, given $\tau\in\{\circ,1\}^{p}$ we say $X^{j}$ is missing
iff $\tau^{j}=\circ$. The missing pattern $\tau$ is a discrete random
variable with distribution
\begin{equation}
\psi(\tau)=\epsilon^{p-\left|\tau\right|}(1-\epsilon)^{\left|\tau\right|}\text{ for }\tau\in\{\circ,1\}^{p}\label{eq:dist-psi},
\end{equation}
where $\left|\cdot\right|$ denotes the number of $1$'s in $\tau$.
Let $s\coloneqq\ensuremath{\ones}-\tau$ be the \emph{complement}
of $\tau$ in $\{\circ,1\}^{p}$, where $\ones$ denotes the vector
with all coordinates $1$, and we also introduce the following notation
for convenience: for a (random) vector $x\in\mathbb{R}^{p}$ and $\tau\in\{\circ,1\}^{p}$,
denote by $x_{\tau}=x\odot\tau$ the Hadamard product of $x$ and
$\tau$, hence $x=x_{s}+x_{\tau}$.

Then in the missing covariates regression model, the observed variable
is $(Y,X_{s})$. Note the missing pattern $\tau=\ones-s$ is determined
by the observed $X_{s}$ by checking the coordinates marked as $\circ$.

Suppose a set of i.i.d samples $\{(y_{k},x_{k})\}_{k=1}^{n}$ and $\{\tau_k\}_{k=1}^n$ are generated by \prettyref{eq:lr-model} and \prettyref{eq:dist-psi} respectively, while we only observe $\left\{ (y_{k},(x_{k})_{s_{k}})\right\} _{k=1}^{n}$ where $s_k=\ones-\tau_k$, and wish to estimate the true population parameter $\theta^{*}$.

It is not hard to write down the complete joint density
\[
f_{\theta}(y,x_{s},x_{\tau})=\phi(y-\theta_{s}^{\intercal}x_{s}-\theta_{\tau}^{\intercal}x_{\tau};0,\sigma^{2})\phi(x_{s}+x_{\tau};0,I_{p})\psi(s),
\]
where $(y,x_{s},x_{\tau})\in\mathbb{R}\times\mathbb{R}^{p}\times\mathbb{R}^{p}$.

The density of the observed pair $(Y,X_{s})$ is the marginalization
\begin{eqnarray*}
g_{\theta}(y,x_{s}) & = & \int_{\mathbb{R}^{p}}f_{\theta}(y,x_{s},x_{\tau})dx_{\tau}\\
 & = & \phi(y-\theta_{s}^{\intercal}x_{s};0,\sigma^{2}+\left\Vert \theta_{\tau}\right\Vert ^{2})\phi(x_{s};0,\text{diag}\{s\})\psi(s),
\end{eqnarray*}
where the integration is over coordinates of $x_{\tau}$ \emph{not}
marked as $\circ$.

The conditional density of the latent variable $X_{\tau}$ is
\[
k_{\theta}(x_{\tau}|y,x_{s})=f_{\theta}(y,x_{s},x_{\tau})/g_{\theta}(y,x_{s})=\phi\left(x_{\tau};b_{\theta},A_{\theta}\right),
\]
which is Gaussian with mean vector
\[
b_{\theta}(y,x_{s})\coloneqq\mathbb{E}_{\theta}\left[X_{\tau}|y,x_{s}\right]=\frac{y-\theta_{s}^{\intercal}x_{s}}{\sigma^{2}+\left\Vert \theta_{\tau}\right\Vert ^{2}}\theta_{\tau},
\]
and covariance matrix
\begin{equation}
A_{\theta}(\tau)\coloneqq\mathbb{E}_{\theta}\left[\left(X_{\tau}-b_{\theta}(y,x_{s})\right)\left(X_{\tau}-b_{\theta}(y,x_{s})\right)^{\intercal}|y,x_{s}\right]=\text{diag}\{\tau\}-\frac{1}{\sigma^{2}+\left\Vert \theta_{\tau}\right\Vert ^{2}}\theta_{\tau}\theta_{\tau}^{\intercal},\label{eq:A-theta-rmc}
\end{equation}
which is also the conditional covariance matrix of the vector $X$
given $x_{s}$ and $y$, since
\[
X_{\tau}-b_{\theta}(y,x_{s})=X-\mathbb{E}_{\theta}\left[X|y,x_{s}\right].
\]

Denote the conditional mean of the covariate $X$ by
\begin{equation}
\mu_{\theta}(y,x_{s})\coloneqq\mathbb{E}_{\theta}\left[X|y,x_{s}\right]=\mathbb{E}_{\theta}\left[x_{s}+X_{\tau}|y,x_{s}\right]=x_{s}+b_{\theta}(y,x_{s}),\label{eq:mu-theta-rmc}
\end{equation}
and the conditional mean of the matrix $XX^{\intercal}$ by
\begin{equation}
\Sigma_{\theta}(y,x_{s})\coloneqq\mathbb{E}_{\theta}\left[XX^{\intercal}|y,x_{s}\right]=\mu_{\theta}(y,x_{s})\mu_{\theta}(y,x_{s})^{\intercal}+A_{\theta}(\tau).\label{eq:sigma-theta-rmc}
\end{equation}

It is a routine procedure to calculate the stochastic $Q$-function
\begin{eqnarray*}
Q(\theta'|\theta;(y,x_{s})) & = & \mathbb{E}_{\theta}\left[\log f_{\theta'}(y,x_{s},X_{\tau})|y,x_{s}\right]\\
 & = & -\frac{1}{2\sigma^{2}}\mathbb{E}_{\theta}\left[y^{2}-2y\theta'^{\intercal}X+\theta'^{\intercal}XX^{\intercal}\theta'|y,x_{s}\right]\\
 &  & -\frac{1}{2}\mathbb{E}_{\theta}\left[X^{\intercal}X|y,x_{s}\right]-p\log\sqrt{2\pi}+\log\psi(s)\\
 & = & -\frac{1}{2\sigma^{2}}\left(y^{2}-2\theta'^{\intercal}\mu_{\theta}(y,x_{s})y+\theta'^{\intercal}\Sigma_{\theta}(y,x_{s})\theta'\right)\\
 &  & -\frac{1}{2}\text{tr}\Sigma_{\theta}(y,x_{s})-p\log\sqrt{2\pi}+\log\psi(s)
\end{eqnarray*}
and the gradient
\[
\nabla_{1}Q(\theta'|\theta;(y,x_{s}))=\frac{1}{\sigma^{2}}\left[y\mu_{\theta}(y,x_{s})-\Sigma_{\theta}(y,x_{s})\theta'\right].
\]
Hence the GRV
\begin{equation}
\Gamma(\theta;(Y,X_{s}))=\frac{1}{\sigma^{2}}\left[Y\left(\mu_{\theta}(Y,X_{s})-\mu_{\theta^{*}}(Y,X_{s})\right)-\left(\Sigma_{\theta}(Y,X_{s})-\Sigma_{\theta^{*}}(Y,X_{s})\right)\theta^{*}\right].\label{eq:Gam-rmc}
\end{equation}
Since $Q(\theta'|\theta;(y,x_{s}))$ is quadratic in $\theta'$, the
CRV is
\begin{equation}
V(\theta'|\theta;(Y,X_{s}))=-\frac{1}{2\sigma^{2}}(\theta'-\theta^{*})^{\intercal}\Sigma_{\theta}(Y,X_{s})(\theta'-\theta^{*})\label{eq:V-rmc}
\end{equation}
by \prettyref{lem:quad-taylor}. Then the SEV
\begin{equation}
\mathcal{E}(Y,X_{s})=\frac{1}{\sigma^{2}}\left[Y\mu_{\theta^{*}}(Y,X_{s})-\Sigma_{\theta^{*}}(Y,X_{s})\theta^{*}\right].\label{eq:E-rmc}
\end{equation}

\subsubsection{\label{subsec:ora-cov-rmc}Oracle Convergence}

For $r>0$ and $\theta^{*}\neq0$, let $\xi\coloneqq\left(1+\omega\right)\eta^{2}$,
where the RCR $\omega$ and SNR $\eta$ are defined in the notations.
To characterize the sets $\mathcal{G}(r)$ and $\mathcal{V}(r,R)$,
we have the following bounds for the population mean of GRV and CRV.
\begin{lem}
\label{lem:bound-Gam-rmc}For the Linear Regression with Missing Covariates
model,
\[
\left\Vert \mathbb{E}\Gamma(\theta;(Y,X_{s}))\right\Vert \le\gamma\left(\omega,\eta\right)\left\Vert \theta-\theta^{*}\right\Vert \text{ for }\theta\in B_{r}\left(\theta^{*}\right)
\]
where
\begin{equation}
\gamma\left(\omega,\eta\right)\coloneqq\frac{1}{\sigma^{2}}\left(\left(\omega\xi^{2}+\left(3\omega+2\right)\xi+1\right)\epsilon+\xi\sqrt{\epsilon\left(1-\epsilon\right)}\right),\label{eq:gam-rmc}
\end{equation}
and the expectation is taken over $(Y,X_{s})$ and the missing pattern
$\tau=\ones-s$.
\end{lem}
\begin{proof}
See \prettyref{subsec:pf-bnd-Gam-rmc}.
\end{proof}
\begin{rem*}
It follows that $\gamma\left(\omega,\eta\right)\in\mathcal{G}(r)\neq\varnothing$
for $r>0$.
\end{rem*}
\begin{lem}
\label{lem:bound-V-rmc}For the Linear Regression with Missing Covariates
model,
\[
\mathbb{E}V(\theta'|\theta;(Y,X_{s}))\le-\nu\left(\omega,\eta\right)\left\Vert \theta'-\theta^{*}\right\Vert ^{2}\text{ for }\left(\theta',\theta\right)\in\mathbb{R}^{p}\times B_{r}\left(\theta^{*}\right)
\]
where
\begin{equation}
\nu\left(\omega,\eta\right)\coloneqq\frac{1}{2\sigma^{2}}\left(1-2\omega\xi\sqrt{\epsilon\left(1-\epsilon\right)}-\left(1+\omega\right)\xi\epsilon\right),\label{eq:nu-rmc}
\end{equation}
and the expectation is taken over $(Y,X_{s})$ and the missing pattern
$\tau=\ones-s$.
\end{lem}
\begin{proof}
See \prettyref{subsec:pf-bnd-V-rmc}.
\end{proof}
\begin{rem*}
It follows that $\nu\left(\omega,\eta\right)\in\mathcal{V}(r,+\infty)\neq\varnothing$
for $r>0$.
\end{rem*}
In view of above lemmas, for $\theta^{*}\neq0$, if the probability
of missingness $\epsilon$ and the RCR $\omega$ are sufficiently
small and the SNR $\eta$ is bounded above, then $\gamma\left(\omega,\eta\right)\ll\frac{1}{2\sigma^{2}}$
and $\nu\left(\omega,\eta\right)\approx\frac{1}{2\sigma^{2}}$ and
in that case, $0<r<+\infty$, where $r=\omega\left\Vert \theta^{*}\right\Vert $
are radii of contraction and $\left(\gamma\left(\omega,\eta\right),\nu\left(\omega,\eta\right)\right)\in\mathcal{C}\left(r,+\infty\right)\neq\varnothing$
is a pair of contraction parameters. By imposing conditions that ensure
$\gamma\left(\omega,\eta\right)<\nu\left(\omega,\eta\right)$, we
can obtain various forms of oracle convergence results via the oracle
convergence theorem. Among them we state and prove the following corollary.
\begin{cor}
For the Linear Regression with Missing Covariates model, if $\theta^{*}\neq0$
and
\begin{equation}
\frac{1}{\sqrt{1+\omega}}<\eta<\frac{1}{3\left(1+\omega\right)\sqrt[4]{\epsilon}},\label{eq:eta-oct-rmc}
\end{equation}
then $0<r<+\infty$ where $r=\omega\left\Vert \theta^{*}\right\Vert $
are radii of contraction and $\left(\gamma\left(\omega,\eta\right),\nu\left(\omega,\eta\right)\right)\in\mathcal{C}(r,+\infty)$.
Then given initial point $\theta^{0}\in B_{r}(\theta^{*})$, any oracle
EM sequence $\{\theta^{t}\}_{t\ge0}$ such that
\[
\theta^{t+1}\in\arg\max_{\theta'\in\Omega}Q_{*}(\theta'|\theta^{t})\text{ for }t\ge0,
\]
satisfies the inequality
\begin{equation}
\left\Vert \theta^{t}-\theta^{*}\right\Vert \le\overline{\kappa}^{t}\left\Vert \theta^{0}-\theta^{*}\right\Vert \label{eq:ora-conv-rmc}
\end{equation}
where $\overline{\kappa}\coloneqq\frac{\overline{\gamma}}{\overline{\nu}}\le\frac{\gamma\left(\omega,\eta\right)}{\nu\left(\omega,\eta\right)}<1$,
is the optimal oracle convergence rate with respect to $r<+\infty$.
\end{cor}
\begin{proof}
We only need to verify that \prettyref{eq:eta-oct-rmc} implies $\gamma\left(\omega,\eta\right)<\nu\left(\omega,\eta\right)$,
which is trivial, and the corollary follows from the oracle convergence
theorem.
\end{proof}
\begin{rem*}
The condition \prettyref{eq:eta-oct-rmc} imposes an upper bound on
the probability of missingness $\epsilon$, namely $\sqrt{\epsilon}<\frac{1}{9\left(1+\omega\right)}<\frac{1}{9}$,
hence $\epsilon<\frac{1}{81}$.
\end{rem*}

\subsubsection{Empirical Convergence}

For empirical convergence results, we need to find the specific forms
of the $\varepsilon$-bounds in the Assumptions. To ease notations,
we use $Z_{k}\coloneqq(Y_{k},(X_{k})_{s_{k}})$ to denote an i.i.d.
copy of $\left(Y,X_{s}\right)$ throughout this section.
\begin{lem}
\label{lem:epsilon-1-rmc}For $\delta\in(0,1)$ and $r>0$, if $n>c\log\left(L/\delta\right)$
then
\begin{equation}
\left\Vert \Gamma_{n}\left(\theta;\left\{ Z_{k}\right\} \right)-\mathbb{E}\Gamma\left(\theta;\left(Y,X_{s}\right)\right)\right\Vert \le\frac{C\left(\omega,\eta\right)}{\sigma^{2}}\sqrt{\frac{\log(L/\delta)}{n}}\left\Vert \theta-\theta^{*}\right\Vert \text{ for }\theta\in B_{r}(\theta^{*})\label{eq:epsilon-1-rmc}
\end{equation}
with probability at least $1-\delta$, where
\begin{align*}
C\left(\omega,\eta\right) & =C_{1}\left(\eta\left(1+\eta\right)\left(2+\omega\right)+1\right)\eta\left(1+\eta\right)\left(1+\omega\right)\\
 & +C_{2}\left(\eta\left(1+\eta\right)\left(2+\omega\right)+1\right)\left(1+\omega\right)\eta^{2}\\
 & +C_{3}\left(\left(1+\omega\right)\eta^{2}+1\right)\left(2+\omega\right)\eta^{2}\\
 & =O\left(\left(1+\omega\right)^{2}\left(1+\eta\right)^{4}\right).
\end{align*}
\end{lem}
\begin{proof}
See \prettyref{subsec:pf-eps-1-rmc}.
\end{proof}
\begin{lem}
\label{lem:epsilon-2-rmc}For $\delta\in(0,1)$ and $r>0$, if $n>c\log\left(L/\delta\right)$,
then
\begin{equation}
\left|V_{n}\left(\theta'|\theta;\left\{ Z_{k}\right\} \right)-\mathbb{E}V(\theta'|\theta;(Y,X_{s}))\right|\le\frac{C}{\sigma^{2}}\sqrt{\frac{\log(L/\delta)}{n}}\left\Vert \theta'-\theta^{*}\right\Vert ^{2}\text{ for }\theta',\theta\in B_{r}(\theta^{*})\label{eq:epsilon-2-rmc}
\end{equation}
with probability at least $1-\delta$.
\end{lem}
\begin{proof}
See \prettyref{subsec:pf-esp-2-rmc}.
\end{proof}
\begin{lem}
\label{lem:epsilon-s-rmc}For $\delta\in(0,1)$ and $n>c\log\left(L/\delta\right)$,
then
\begin{equation}
\left\Vert \mathcal{E}_{n}\left(\left\{ Z_{k}\right\} \right)\right\Vert \le C\frac{\left(1+\eta\right)}{\sigma}\sqrt{\frac{\log(L/\delta)}{n}}\label{eq:epsilon-s-rmc}
\end{equation}
with probability at least $1-\delta$.
\end{lem}
\begin{proof}
See \prettyref{subsec:pf-esp-s-rmc}.
\end{proof}
From \prettyref{eq:epsilon-1-rmc}, \prettyref{eq:epsilon-2-rmc}
and \prettyref{eq:epsilon-s-rmc} in above lemmas,  it is clear that
the Assumptions (\ref{A1}$\sim$\ref{A3}) are satisfied with the
following
\[
\varepsilon_{1}(\delta,r,n,p)=\frac{C\left(\omega,\eta\right)}{\sigma^{2}}\sqrt{\frac{\log(L/\delta)}{n}},\quad\varepsilon_{2}(\delta,r,R,n,p)=\frac{C_{1}}{\sigma^{2}}\sqrt{\frac{\log(L/\delta)}{n}}\quad\text{and}
\]
\[
\varepsilon_{s}(\delta,r,R,n,p)=C_{2}\frac{\left(1+\eta\right)}{\sigma}\sqrt{\frac{\log(L/\delta)}{n}}.
\]
We obtain the following data-adaptive empirical convergence result
for the Linear Regression with Missing Covariates model.
\begin{cor}
Suppose $\{(Y_{k},(X_{k})_{s_{k}})\}_{k=1}^{n}$ is a set of i.i.d.
copies of $\left(Y,X_{s}\right)$ and $\delta\in(0,1)$. If $\theta^{*}\neq0$,
\begin{equation}
\frac{1}{\sqrt{1+\omega}}<\eta<\frac{1}{3\left(1+\omega\right)\sqrt[4]{\epsilon}},\label{eq:eta-orct-rmc}
\end{equation}
and the sample size $n$ is sufficiently large such that
\begin{equation}
n>\left(C_{1}+C_{2}\frac{1+\eta}{\omega\eta}+C\left(\omega,\eta\right)\right)^{2}\log(L/\delta).\label{eq:cond-for-n-rmc}
\end{equation}
Then given an initial point $\Theta_{n}^{0}\in B_{r}(\theta^{*})$,
with probability at least $1-\delta$, the empirical EM sequence $\{\Theta_{n}^{t}\}_{t\ge0}$
such that
\[
\Theta_{n}^{t+1}\in\arg\max_{\Theta'\in\Omega}Q_{n}(\Theta'|\Theta_{n}^{t};\{(Y_{k},(X_{k})_{s_{k}})\})\text{ for }t\ge0,
\]
satisfies the inequality
\begin{equation}
\left\Vert \Theta_{n}^{t}-\theta^{*}\right\Vert \le\left(\overline{K}_{n}\right)^{t}\left\Vert \Theta_{n}^{0}-\theta^{*}\right\Vert +\frac{C_{2}K}{\sigma^{2}\left(\nu_{n}-\gamma_{n}\right)}\sqrt{\frac{\log(L/\delta)}{n}}\label{eq:emp-conv-rmc}
\end{equation}
where
\[
\gamma_{n}:=\gamma\left(\omega,\eta\right)+\frac{C\left(\omega,\eta\right)}{\sigma^{2}}\sqrt{\frac{\log(L/\delta)}{n}}\text{ and }\nu_{n}:=\nu\left(\omega,\eta\right)-\frac{C_{1}}{\sigma^{2}}\sqrt{\frac{\log(L/\delta)}{n}}
\]
and the optimal empirical convergence rate $\overline{K}_{n}$ satisfies
that
\[
\overline{K}_{n}\le\frac{\gamma_{n}}{\nu_{n}}<1\text{ and }\left|\overline{K}_{n}-\overline{\kappa}\right|\le\left(C\left(\omega,\eta\right)+C_{1}\kappa\left(\omega,\eta\right)\right)\sqrt{\frac{\log(L/\delta)}{n}}.
\]
\end{cor}
\begin{proof}
Note condition \prettyref{eq:cond-for-n-rmc} implies the lower bounds
of $n$ in \prettyref{lem:epsilon-1-rmc}, \prettyref{lem:epsilon-2-rmc}
and \prettyref{lem:epsilon-s-rmc}. Hence to apply the empirical convergence
theorem, we only need to verify that
\[
\varepsilon_{s}(\delta,r,R,n,p)+r\varepsilon_{1}(\delta,r,n,p)+r\varepsilon_{2}(\delta,r,R,n,p)<r\left(\nu\left(\omega,\eta\right)-\gamma\left(\omega,\eta\right)\right)
\]
holds true whenever $n$ satisfies \prettyref{eq:cond-for-n-rmc},
but this is trivial.

Moreover, it is not difficult to see that $\nu\left(\omega,\eta\right)>\frac{1}{3\sigma^{2}}$
under the condition \prettyref{eq:eta-orct-rmc}, and condition \prettyref{eq:cond-for-n-rmc}
implies that $\varepsilon_{2}(\delta,r,R,n,p)=\frac{C_{1}}{\sigma^{2}}\sqrt{\frac{\log(L/\delta)}{n}}<\frac{1}{6\sigma^{2}}<\frac{1}{2}\nu\left(\omega,\eta\right)\le\frac{1}{2}\overline{\nu}$,
hence the concentration bound of $\overline{K}_{n}$ follows from
the optimal rate convergence theorem.
\end{proof}

\section{\label{sec:disc}Discussion}

In this paper, we have proved that for given radii of contraction $0<r\le R$,
the oracle EM sequence $\{\theta^{t}\}_{t\ge0}$ converges geometrically
to the true population parameter $\theta^{*}$ at the optimal rate
$\overline{\kappa}$ with respect to $r\le R$. This is a \emph{deterministic}
result.

As illustrated in \prettyref{sec:app-cla-mod}, we can often obtain
some upper bounds $\kappa$ for the optimal rate in concrete models.
Although we may not be able to calculate $\overline{\kappa}$ in closed
form, the oracle EM sequence is smart enough to converge optimally.

Similar remarks apply to the empirical convergence, where we showed
that given oracle convergence with respect to radii of contraction
$r\le R$, an empirical EM sequence converges geometrically at the
rate $k_{n}$ as a realization of the optimal empirical convergence
rate $\overline{K}_{n}$, which is a random variable upper bounded
by $\overline{\kappa}_{n}$ and concentrated on $\overline{\kappa}$,
see \prettyref{figure:K-kap}. This is a \emph{probabilistic} result.

The concentration inequality \prettyref{eq:kn-k} is how we find a
reconciliation of our theory with the classical theories on the asymptotic
convergence rate of the EM algorithm, i.e. when the sample size $n$
is sufficiently large, the random fluctuations of $\overline{K}_{n}$
are so small that it behaves almost like the constant $\overline{\kappa}$.

The idea of considering an MLE as a maximizer of a realization of
the empirical log-likelihood functional of i.i.d. random variables
and the EM algorithm as a realization of an iterative process for
approximating the true population parameter $\theta^{*}$ can be further
applied in optimization problems involving iterative procedures, in
which the data generative model is probabilistic. In such a scenario,
by exploiting the oracle deterministic convergence results and the
concentration of measure phenomena of random variables, it is foreseeable
that one can obtain similar convergence results as in this paper.

%It is not surprising that our theory works in the low dimension regime
%since the concentration properties entails that one must have $n\gg p$.

\pdfbookmark[1]{Acknowledgments}{sec:ack}

\section*{Acknowledgments}

This work was partially supported by grant NO. 61501389 from National
Natural Science Foundation of China (NSFC), grants HKBU-22302815 and
HKBU-12316116 from Hong Kong Research Grant Council, and grant FRG2/15-16/011
from Hong Kong Baptist University.

\appendix

\section{Proofs for the Gaussian Mixture Model}

We give proofs for the Gaussian Mixture model in this section.

\subsection{Preliminaries}

We first prove the sub-gaussianity of the random vector $Y$ defined
in the model.
\begin{lem}
\label{lem:u.Y-sub-gau}Let $W\sim\mathcal{N}(0,\sigma^{2}I_{p})$
and $Y=Z\cdot\theta^{*}+W$ be defined in the Gaussian Mixture model,
then for any $u\in\mathbb{R}^{p}$ the random variable $u^{\intercal}Y$
is sub-gaussian with Orlicz norm $\left\Vert u^{\intercal}Y\right\Vert _{\psi_{2}}\le K\left\Vert u\right\Vert $.
\end{lem}
\begin{proof}
It is clear that $u^{\intercal}W$ is a zero-mean Gaussian random
variable with variance
\[
\text{Var}\left(u^{\intercal}W\right)=\sum_{j=1}^{p}\left|u^{j}\right|^{2}\sigma^{2}=\left\Vert u\right\Vert ^{2}\sigma^{2}.
\]
Hence $\left\Vert u^{\intercal}W\right\Vert _{\psi_{2}}\le\sigma\left\Vert u\right\Vert $
and since $\left\Vert u^{\intercal}Y\right\Vert _{\psi_{2}}\le\left\Vert Z\cdot u^{\intercal}\theta^{*}\right\Vert _{\psi_{2}}+\left\Vert u^{\intercal}W\right\Vert _{\psi_{2}}\le\left\Vert \theta^{*}\right\Vert \left\Vert u\right\Vert +\sigma\left\Vert u\right\Vert $,
the lemma follows.
\end{proof}

\subsection{\label{subsec:pf-eps-1-gmm}Proof of \prettyref{lem:epsilon-1-gmm}}
\begin{proof}
By the Mean Value Theorem, we have
\[
w_{\theta}(Y)-w_{\theta^{*}}(Y)=\varsigma\left(\frac{2\theta^{\intercal}Y}{\sigma^{2}}\right)-\varsigma\left(\frac{2\theta^{*\intercal}Y}{\sigma^{2}}\right)=\frac{2}{\sigma^{2}}\varsigma'\left(\frac{2\vartheta^{\intercal}Y}{\sigma^{2}}\right)Y^{\intercal}\left(\theta-\theta^{*}\right)
\]
where $\vartheta$ is a point on the line segment joining $\theta$
and $\theta^{*}$. In view of \prettyref{eq:Gam-gmm}, for any $u\in\mathbb{S}^{p-1}$,
\[
\left|u^{\intercal}\Gamma(\theta;Y)\right|=\frac{2}{\sigma^{2}}\left|\varsigma'\left(\frac{2\vartheta^{\intercal}Y}{\sigma^{2}}\right)\left(\theta-\theta^{*}\right)^{\intercal}YY^{\intercal}u\right|\stackrel{(a)}{\le}\frac{1}{2\sigma^{2}}\left|\left(\theta-\theta^{*}\right)^{\intercal}Y\cdot u^{\intercal}Y\right|,
\]
where $(a)$ follows from the fact that $\varsigma'(t)=\varsigma(t)(1-\varsigma(t))\le\frac{1}{4}$
for $t\in\mathbb{R}$. Then by \prettyref{lem:u.Y-sub-gau} and \prettyref{lem:orl-norm}\prettyref{enu:prod-sub-gau},
$u^{\intercal}\Gamma(\theta;Y)$ is sub-exponential with Orlicz norm
\begin{align*}
\left\Vert u^{\intercal}\Gamma(\theta;Y)\right\Vert _{\psi_{1}} & \le\frac{1}{2\sigma^{2}}\left\Vert \left(\theta-\theta^{*}\right)^{\intercal}Y\cdot u^{\intercal}Y\right\Vert _{\psi_{1}}\\
 & \le\frac{C}{2\sigma^{2}}\left\Vert \left(\theta-\theta^{*}\right)^{\intercal}Y\right\Vert _{\psi_{2}}\cdot\left\Vert u^{\intercal}Y\right\Vert _{\psi_{2}}\\
 & \le\frac{CK^{2}}{2\sigma^{2}}\left\Vert \theta-\theta^{*}\right\Vert
\end{align*}
and the result follows from \prettyref{lem:constr-sub-exp} on the
concentration of sub-exponential random vectors.
\end{proof}

\subsection{\label{subsec:proof-of-epsilon-s}Proof of \prettyref{lem:epsilon-s-gmm}}
\begin{proof}
Let $A=\frac{1}{\sigma^{2}}(2w_{\theta^{*}}(Y)-1)Y$, then simple
calculation yields that $\mathbb{E}_{\theta^{*}}A=\frac{1}{\sigma^{2}}\theta^{*}$.
Hence $\mathcal{E}(Y)=A-\mathbb{E}_{\theta^{*}}A$ is the centered
random vector and $\mathcal{E}_{n}(\{Y_{k}\})$ is the empirical mean
of $\mathcal{E}(Y)$. Since $0<w_{\theta^{*}}(Y)<1$, for $u\in\mathbb{S}^{p-1}$,
we have
\[
\left|u^{\intercal}A\right|=\left|\frac{1}{\sigma^{2}}(2w_{\theta^{*}}(Y)-1)u^{\intercal}Y\right|\le\frac{1}{\sigma^{2}}\left|u^{\intercal}Y\right|.
\]
Hence $u^{\intercal}A$ is sub-gaussian with $\left\Vert u^{\intercal}A\right\Vert _{\psi_{2}}\le\frac{K}{\sigma^{2}}$
by \prettyref{lem:u.Y-sub-gau}, and the result follows from \prettyref{lem:constr-sub-gau}
on the concentration of sub-gaussian random vectors.
\end{proof}

\section{Proofs for Mixture of Linear Regressions}

We give proofs for the Mixture of Linear Regressions model in this
section.

\subsection{Preliminaries}

We first prove some properties of the random variates defined in the
model.
\begin{lem}
\label{lem:Y-X-props-mlr}Let $X\sim\mathcal{N}(0,I_{p})$ and $Y=Z\cdot X^{\intercal}\theta^{*}+W$
be defined in the Mixture of Linear Regressions model, then
\begin{itemize}
\item[$(a)$ ] $u^{\intercal}X$ is Gaussian with Orlicz norm $\left\Vert u^{\intercal}X\right\Vert _{\psi_{2}}\le\left\Vert u\right\Vert $
for $u\in\mathbb{R}^{p}$;
\item[$(b)$ ] $\left\Vert X\right\Vert $ is sub-gaussian with Orlicz norm at most
$\sqrt{2p}$;
\item[$(c)$ ] $Y$ is sub-gaussian with Orlicz norm $\left\Vert Y\right\Vert _{\psi_{2}}\le K$
where $K\coloneqq\left\Vert \theta^{*}\right\Vert +\sigma$.
\end{itemize}
\end{lem}
\begin{proof}
$(a)$ It is easy to see that $\text{Var}\left(u^{\intercal}X\right)=\text{Var}\left(\sum_{j=1}^{p}u^{j}X^{j}\right)=\sum_{j=1}^{p}\left|u^{j}\right|^{2}=\left\Vert u\right\Vert ^{2}$
and the result follows. $(b)$ Denote $A=\left\Vert X\right\Vert $,
then we have
\[
\left\Vert A\right\Vert _{\psi_{2}}^{2}\le\left\Vert A^{2}\right\Vert _{\psi_{1}}=\left\Vert \sum_{j=1}^{p}\left(X^{j}\right)^{2}\right\Vert _{\psi_{1}}\le\sum_{j=1}^{p}\left\Vert \left(X^{j}\right)^{2}\right\Vert _{\psi_{1}}\le\sum_{j=1}^{p}2\left\Vert X^{j}\right\Vert _{\psi_{2}}^{2}\le2p.
\]
$(c)$ By definition, $\left\Vert Y\right\Vert _{\psi_{2}}\le\left\Vert Z\cdot X^{\intercal}\theta^{*}\right\Vert _{\psi_{2}}+\left\Vert W\right\Vert _{\psi_{2}}=\left\Vert X^{\intercal}\theta^{*}\right\Vert _{\psi_{2}}+\sigma$
and since $X^{\intercal}\theta^{*}$ is a zero-mean Gaussian with
variance $\text{Var}\left(X^{\intercal}\theta^{*}\right)=\left\Vert \theta^{*}\right\Vert ^{2}$,
the result follows.
\end{proof}

\subsection{\label{subsec:pf-eps-1-mlr}Proof of \prettyref{lem:epsilon-1-mlr}}
\begin{proof}
By the Mean Value Theorem, we have
\[
w_{\theta}(Y,X)-w_{\theta^{*}}(Y,X)=\varsigma\left(\frac{2\theta^{\intercal}XY}{\sigma^{2}}\right)-\varsigma\left(\frac{2\theta^{*\intercal}XY}{\sigma^{2}}\right)=\frac{2}{\sigma^{2}}\varsigma'\left(\frac{2\vartheta^{\intercal}XY}{\sigma^{2}}\right)\left(\theta-\theta^{*}\right)^{\intercal}XY,
\]
where $\vartheta$ is a point on the line segment joining $\theta$
and $\theta^{*}$. Then in view of \prettyref{eq:Gam-mlr}, we have
for any $u\in\mathbb{S}^{p-1}$,
\begin{equation}
\left|u^{\intercal}\Gamma(\theta;(Y,X))\right|=\frac{2}{\sigma^{2}}\left|w_{\theta}(Y,X)-w_{\theta^{*}}(Y,X)\right|\cdot\left|u^{\intercal}XY\right|\stackrel{(a)}{\le}\frac{1}{2\sigma^{2}}\left\Vert \theta-\theta^{*}\right\Vert \cdot\left\Vert YX\right\Vert ^{2},\label{eq:int-1-mlr}
\end{equation}
where $(a)$ follows from the Cauchy-Schwartz inequality and the fact
that $\varsigma'(t)=\varsigma(t)(1-\varsigma(t))\le\frac{1}{4}$ for
$t\in\mathbb{R}$.

Define the random vector $A=\frac{2\sigma^{2}\Gamma(\theta;(Y,X))}{\left\Vert \theta-\theta^{*}\right\Vert }$
for $\theta\in B_{r}^{\times}(\theta^{*})$; let $A_{k}$ be the i.i.d.
copy of $A$ corresponding to $(Y_{k},X_{k})$ and let \textbf{$B_{n}=\frac{1}{n}\sum_{k=1}^{n}A_{k}-\mathbb{E}_{\theta^{*}}A$}.

In view of \prettyref{eq:int-1-mlr}, for $u\in\mathbb{S}^{p-1}$
and $t>0$, we have
\[
\Pr\left\{ \left|u^{\intercal}A\right|\ge t\right\} \le\Pr\left\{ \left\Vert YX\right\Vert ^{2}\ge t\right\} =\Pr\left\{ \left\Vert YX\right\Vert \ge\sqrt{t}\right\} \le C\exp\left(-ct^{\frac{1}{2}}\right),
\]
since $\left\Vert YX\right\Vert $ is sub-exponential with Orlicz
norm at most $CK\sqrt{2p}$ by \prettyref{lem:Y-X-props-mlr}. It
follows from Proposition 2.1.9 and its extensions in \cite{Tao} that
\[
\Pr\left\{ \left|u^{\intercal}B_{n}\right|\ge t\right\} \le C\exp\left(-ctn^{\frac{1}{2}-\epsilon}\right),
\]
where $0<\epsilon\ll\frac{1}{2}$ is a small constant. Then by discretization
of norm, for a $\frac{1}{2}$-net $\{u_{i}\}_{i=1}^{L}$ of $\mathbb{S}^{p-1}$,
\[
\left\Vert B_{n}\right\Vert \le2\max_{1\le i\le L}u_{i}^{\intercal}B_{n},
\]
and by the union bound and pigeonhole principle, we have
\[
\Pr\left\{ \left\Vert B_{n}\right\Vert \ge t\right\} \le\sum_{i=1}^{L}\Pr\left\{ \left|u_{i}^{\intercal}B_{n}\right|\ge\frac{t}{2}\right\} \le CL\exp\left(-\frac{1}{2}ctn^{\frac{1}{2}-\epsilon}\right).
\]
Then by equating the right hand side to $\delta$ and solving for
$t$, we obtain
\[
\left\Vert \Gamma_{n}(\theta;\{(Y_{k},X_{k})\})-\mathbb{E}_{\theta^{*}}\Gamma(\theta;(Y,X))\right\Vert \le C\frac{\log(L/\delta)}{\sigma^{2}n^{\frac{1}{2}-\epsilon}}\left\Vert \theta-\theta^{*}\right\Vert
\]
for $\theta\in B_{r}(\theta^{*})$ with probability at least $1-\delta$.
\end{proof}

\subsection{\label{subsec:pf-eps-2-mlr}Proof of \prettyref{lem:epsilon-2-mlr}}
\begin{proof}
In view of \prettyref{eq:V-mlr}, we have
\[
V(\theta'|\theta;\{(Y,X)\})=-\frac{1}{2\sigma^{2}}\left[\left(\theta'-\theta^{*}\right)^{\intercal}X\right]^{2}.
\]
By \prettyref{lem:Y-X-props-mlr}, the random variable $\left(\theta'-\theta^{*}\right)^{\intercal}X$
is Gaussian with Orlicz norm $\left\Vert \left(\theta'-\theta^{*}\right)^{\intercal}X\right\Vert _{\psi_{2}}\le\left\Vert \theta'-\theta^{*}\right\Vert $
and hence $V(\theta;\{(Y,X)\})$ is sub-exponential with Orlicz norm
\[
\left\Vert V(\theta'|\theta;\{(Y,X)\})\right\Vert _{\psi_{1}}\le\frac{C}{\sigma^{2}}\left\Vert \theta'-\theta^{*}\right\Vert ^{2}.
\]
The result follows from \prettyref{lem:constr-sub-exp} on concentration
of sub-exponential random variables.
\end{proof}

\subsection{\label{subsec:pf-eps-s-mlr}Proof of \prettyref{lem:epsilon-s-mlr}}
\begin{proof}
In view of \prettyref{eq:E-mlr}, for $u\in\mathbb{S}^{p-1}$, we
have
\[
\left|u^{\intercal}\mathcal{E}(Y,X)\right|\le\frac{1}{\sigma^{2}}\left(\left|u^{\intercal}X\right|\left|Y\right|+\left|X^{\intercal}\theta^{*}\right|\left|X^{\intercal}u\right|\right),
\]
since $\left|2w_{\theta^{*}}(Y,X)-1\right|<1$. Then by \prettyref{lem:Y-X-props-mlr},
$u^{\intercal}\mathcal{E}(Y,X)$ is sub-exponential with Orlicz norm
\[
\left\Vert u^{\intercal}\mathcal{E}(Y,X)\right\Vert _{\psi_{1}}\le\frac{C}{\sigma^{2}}\left(K+\left\Vert \theta^{*}\right\Vert \right)=\frac{C}{\sigma}\left(1+2\eta\right).
\]
The result follows from \prettyref{lem:constr-sub-exp} on concentration
of sub-exponential random variables.
\end{proof}

\section{Proofs for Linear Regression with Missing Covariates}

We give proofs for the Linear Regression with Missing Covariates model
in this section.

\subsection{Proofs for Oracle Convergence}

We prove lemmas for oracle convergence of the model, and we start
with some basic facts.

\subsubsection{Preliminaries}

We denote the expectation with respect to the random vector $\tau$
and its measurable functions by $\mathbb{E}_{\epsilon}\left[\cdot\right]$.
It is easy to see that $\mathbb{E}_{\epsilon}\left[\tau\right]=\epsilon\ones$,
$\mathbb{E}_{\epsilon}\left[s\right]=(1-\epsilon)\ones$ and more
generally, we have the following lemma.
\begin{lem}
\label{lem:mean-tau}$\mathbb{E}_{\epsilon}\left[x_{\tau}\right]=\epsilon x$,
$\mathbb{E}_{\epsilon}\left[\left\Vert x_{\tau}\right\Vert ^{2}\right]\le\epsilon\left\Vert x\right\Vert ^{2}$
and $\mathbb{E}_{\epsilon}\left[\left|x_{\tau}^{\intercal}y_{\tau}\right|^{2}\right]\le\epsilon\left|x^{\intercal}y\right|^{2}$
for $x,y\in\mathbb{R}^{p}$.
\end{lem}
\begin{proof}
These results follow from simple calculations.
\end{proof}
Now for a \emph{fixed} missing pattern $\tau\in\{\circ,1\}^{p}$ and
$s=\ones-\tau$, taking expectation with respect to $(Y,X_{s})$ and
by some calculation of multivariate Gaussian distribution, we have
\[
\mathbb{E}_{\theta^{*}}\left[Y\mu_{\theta}(Y,X_{s})\right]=\theta_{s}^{*}+\frac{\sigma^{2}+\left\Vert \theta_{\tau}^{*}\right\Vert ^{2}+\theta_{s}^{*\intercal}(\theta_{s}^{*}-\theta_{s})}{\sigma^{2}+\left\Vert \theta_{\tau}\right\Vert ^{2}}\theta_{\tau}\quad\text{and}
\]
\begin{align*}
\mathbb{E}_{\theta^{*}}\left[\mu_{\theta}(Y,X_{s})\mu_{\theta}(Y,X_{s})^{\intercal}\right] & =\text{diag}\{s\}\\
 & +\frac{1}{\sigma^{2}+\left\Vert \theta_{\tau}\right\Vert ^{2}}\left(\theta_{\tau}(\theta_{s}^{*}-\theta_{s})^{\intercal}+(\theta_{s}^{*}-\theta_{s})\theta_{\tau}^{\intercal}\right)\\
 & +\frac{\sigma^{2}+\left\Vert \theta_{\tau}^{*}\right\Vert ^{2}+\left\Vert \theta_{s}^{*}-\theta_{s}\right\Vert ^{2}}{\left(\sigma^{2}+\left\Vert \theta_{\tau}\right\Vert ^{2}\right)^{2}}\theta_{\tau}\theta_{\tau}^{\intercal},
\end{align*}
and it follows that
\begin{align}
\mathbb{E}_{\theta^{*}}\Sigma_{\theta}(Y,X_{s}) & =I_{p}\nonumber \\
 & +\frac{1}{\sigma^{2}+\left\Vert \theta_{\tau}\right\Vert ^{2}}\left(\theta_{\tau}(\theta_{s}^{*}-\theta_{s})^{\intercal}+(\theta_{s}^{*}-\theta_{s})\theta_{\tau}^{\intercal}\right)\nonumber \\
 & +\frac{\left\Vert \theta_{\tau}^{*}\right\Vert ^{2}-\left\Vert \theta_{\tau}\right\Vert ^{2}+\left\Vert \theta_{s}^{*}-\theta_{s}\right\Vert ^{2}}{\left(\sigma^{2}+\left\Vert \theta_{\tau}\right\Vert ^{2}\right)^{2}}\theta_{\tau}\theta_{\tau}^{\intercal}.\label{eq:ora-Sigma}
\end{align}
Since $\mathbb{E}_{\theta^{*}}\left[Y\mu_{\theta^{*}}(Y,X_{s})\right]=\theta^{*}$
and $\mathbb{E}_{\theta^{*}}\Sigma_{\theta^{*}}(Y,X_{s})=I_{p}$,
we have
\begin{align}
\sigma^{2}\cdot\mathbb{E}_{\theta^{*}}\Gamma(\theta;(Y,X_{s})) & =\mathbb{E}_{\theta^{*}}\left[Y\mu_{\theta}(Y,X_{s})-\Sigma_{\theta}(Y,X_{s})\theta^{*}\right]\nonumber \\
 & =\theta_{\tau}-\theta_{\tau}^{*}+\frac{\theta_{\tau}^{\intercal}\theta_{\tau}^{*}}{\sigma^{2}+\left\Vert \theta_{\tau}\right\Vert ^{2}}\left(\theta_{s}-\theta_{s}^{*}\right)+\frac{\zeta\cdot\theta_{\tau}}{\left(\sigma^{2}+\left\Vert \theta_{\tau}\right\Vert ^{2}\right)^{2}}\label{eq:ora-Gam-rmc}
\end{align}
where
\begin{equation}
\zeta\coloneqq\left(\sigma^{2}+\left\Vert \theta_{\tau}\right\Vert ^{2}-\theta_{\tau}^{\intercal}\theta_{\tau}^{*}\right)\left(\left\Vert \theta_{\tau}^{*}\right\Vert ^{2}-\left\Vert \theta_{\tau}\right\Vert ^{2}\right)-\theta_{\tau}^{\intercal}\theta_{\tau}^{*}\left\Vert \theta_{s}^{*}-\theta_{s}\right\Vert ^{2}.\label{eq:zeta-rmc}
\end{equation}
Further, we have the following bound for $\zeta$.
\begin{lem}
$\left|\zeta\right|\le\left(\sigma^{2}+\left\Vert \theta_{\tau}\right\Vert ^{2}\right)\left\Vert \theta_{\tau}-\theta_{\tau}^{*}\right\Vert ^{2}+2\sigma^{2}\left\Vert \theta_{\tau}\right\Vert \left\Vert \theta_{\tau}-\theta_{\tau}^{*}\right\Vert +\left\Vert \theta_{\tau}^{*}\right\Vert \left\Vert \theta_{\tau}\right\Vert \left\Vert \theta-\theta^{*}\right\Vert ^{2}$.
\end{lem}
\begin{proof}
By the triangle inequality and Cauchy-Schwartz inequality, we have
\begin{align*}
\left|\zeta\right| & \le\left(\sigma^{2}+\left\Vert \theta_{\tau}\right\Vert \left\Vert \theta_{\tau}-\theta_{\tau}^{*}\right\Vert \right)\left(\left\Vert \theta_{\tau}^{*}\right\Vert +\left\Vert \theta_{\tau}\right\Vert \right)\left\Vert \theta_{\tau}-\theta_{\tau}^{*}\right\Vert +\left\Vert \theta_{\tau}\right\Vert \left\Vert \theta_{\tau}^{*}\right\Vert \left\Vert \theta_{s}^{*}-\theta_{s}\right\Vert ^{2}\\
 & =\sigma^{2}\left\Vert \theta_{\tau}-\theta_{\tau}^{*}\right\Vert \left(\left\Vert \theta_{\tau}\right\Vert +\left\Vert \theta_{\tau}^{*}\right\Vert \right)+\left\Vert \theta_{\tau}\right\Vert ^{2}\left\Vert \theta_{\tau}-\theta_{\tau}^{*}\right\Vert ^{2}+\left\Vert \theta_{\tau}\right\Vert \left\Vert \theta_{\tau}^{*}\right\Vert \left\Vert \theta^{*}-\theta\right\Vert ^{2}\\
 & \le\sigma^{2}\left\Vert \theta_{\tau}-\theta_{\tau}^{*}\right\Vert \left(2\left\Vert \theta_{\tau}\right\Vert +\left\Vert \theta_{\tau}^{*}-\theta_{\tau}\right\Vert \right)+\left\Vert \theta_{\tau}\right\Vert ^{2}\left\Vert \theta_{\tau}-\theta_{\tau}^{*}\right\Vert ^{2}+\left\Vert \theta_{\tau}\right\Vert \left\Vert \theta_{\tau}^{*}\right\Vert \left\Vert \theta^{*}-\theta\right\Vert ^{2}\\
 & =\left(\sigma^{2}+\left\Vert \theta_{\tau}\right\Vert ^{2}\right)\left\Vert \theta_{\tau}-\theta_{\tau}^{*}\right\Vert ^{2}+2\sigma^{2}\left\Vert \theta_{\tau}\right\Vert \left\Vert \theta_{\tau}-\theta_{\tau}^{*}\right\Vert +\left\Vert \theta_{\tau}^{*}\right\Vert \left\Vert \theta_{\tau}\right\Vert \left\Vert \theta-\theta^{*}\right\Vert ^{2},
\end{align*}
and the result follows.
\end{proof}

\subsubsection{\label{subsec:pf-bnd-Gam-rmc}Proof of \prettyref{lem:bound-Gam-rmc}}
\begin{proof}
In view of \prettyref{eq:ora-Gam-rmc} and \prettyref{lem:mean-tau},
we have
\[
\sigma^{2}\cdot\mathbb{E}\Gamma(\theta;(Y,X_{s}))=\sigma^{2}\cdot\mathbb{E}_{\epsilon}\left[\mathbb{E}_{\theta^{*}}\Gamma(\theta;(Y,X_{s}))\right]=\epsilon\left(\theta-\theta^{*}\right)+T_{1}+T_{2},
\]
where
\[
T_{1}\coloneqq\mathbb{E}_{\epsilon}\left[\frac{\theta_{\tau}^{\intercal}\theta_{\tau}^{*}}{\sigma^{2}+\left\Vert \theta_{\tau}\right\Vert ^{2}}\left(\theta_{s}-\theta_{s}^{*}\right)\right]\quad\text{and}\quad T_{2}\coloneqq\mathbb{E}_{\epsilon}\left[\frac{\zeta\cdot\theta_{\tau}}{\left(\sigma^{2}+\left\Vert \theta_{\tau}\right\Vert ^{2}\right)^{2}}\right],
\]
and we can bound $T_{1}$ as
\begin{align}
\left\Vert T_{1}\right\Vert  & \le\frac{1}{\sigma^{2}}\mathbb{E}_{\epsilon}\left[\left|\theta_{\tau}^{\intercal}\theta_{\tau}^{*}\right|\left\Vert \theta_{s}-\theta_{s}^{*}\right\Vert \right]\le\frac{1}{\sigma^{2}}\mathbb{E}_{\epsilon}\left[\left|\theta_{\tau}^{\intercal}\theta_{\tau}^{*}\right|^{2}\right]^{\frac{1}{2}}\mathbb{E}_{\epsilon}\left[\left\Vert \theta_{s}-\theta_{s}^{*}\right\Vert ^{2}\right]^{\frac{1}{2}}\nonumber \\
 & \le\frac{1}{\sigma^{2}}\sqrt{\epsilon\left(1-\epsilon\right)}\left|\theta^{\intercal}\theta^{*}\right|\left\Vert \theta-\theta^{*}\right\Vert \le\left(1+\omega\right)\eta^{2}\sqrt{\epsilon\left(1-\epsilon\right)}\left\Vert \theta-\theta^{*}\right\Vert ,\label{eq:T-1}
\end{align}
and for $T_{2}$ we have
\[
\left\Vert T_{2}\right\Vert \le\mathbb{E}_{\epsilon}\left[\frac{\left|\zeta\right|\left\Vert \theta_{\tau}\right\Vert }{\left(\sigma^{2}+\left\Vert \theta_{\tau}\right\Vert ^{2}\right)^{2}}\right]\le S_{1}+S_{2}+S_{3}\text{ where }S_{1}\coloneqq\mathbb{E}_{\epsilon}\left[\frac{\left\Vert \theta_{\tau}-\theta_{\tau}^{*}\right\Vert ^{2}\left\Vert \theta_{\tau}\right\Vert }{\sigma^{2}+\left\Vert \theta_{\tau}\right\Vert ^{2}}\right]
\]
\[
S_{2}\coloneqq\mathbb{E}_{\epsilon}\left[\frac{2\sigma^{2}\left\Vert \theta_{\tau}\right\Vert ^{2}\left\Vert \theta_{\tau}-\theta_{\tau}^{*}\right\Vert }{\left(\sigma^{2}+\left\Vert \theta_{\tau}\right\Vert ^{2}\right)^{2}}\right],\quad S_{3}\coloneqq\mathbb{E}_{\epsilon}\left[\frac{\left\Vert \theta_{\tau}^{*}\right\Vert \left\Vert \theta_{\tau}\right\Vert ^{2}\left\Vert \theta-\theta^{*}\right\Vert ^{2}}{\left(\sigma^{2}+\left\Vert \theta_{\tau}\right\Vert ^{2}\right)^{2}}\right]
\]
and in view of \prettyref{lem:mean-tau} and Cauchy-Schwartz inequality,
we have
\[
S_{1}\le\frac{1}{\sigma^{2}}\mathbb{E}_{\epsilon}\left[\left\Vert \theta_{\tau}-\theta_{\tau}^{*}\right\Vert \left\Vert \theta_{\tau}\right\Vert \right]\left\Vert \theta-\theta^{*}\right\Vert \le\frac{\epsilon}{\sigma^{2}}\left\Vert \theta\right\Vert \left\Vert \theta-\theta^{*}\right\Vert ^{2}\le\omega\left(1+\omega\right)\eta^{2}\epsilon\left\Vert \theta-\theta^{*}\right\Vert ,
\]
\[
S_{2}\le\frac{2}{\sigma^{2}}\mathbb{E}_{\epsilon}\left[\left\Vert \theta_{\tau}-\theta_{\tau}^{*}\right\Vert \left\Vert \theta_{\tau}\right\Vert \right]\left\Vert \theta\right\Vert \le\frac{2\epsilon}{\sigma^{2}}\left\Vert \theta\right\Vert ^{2}\left\Vert \theta-\theta^{*}\right\Vert \le2\left(1+\omega\right)^{2}\eta^{2}\epsilon\left\Vert \theta-\theta^{*}\right\Vert \text{ and }
\]
\[
S_{3}\le\frac{1}{\sigma^{4}}\mathbb{E}_{\epsilon}\left[\left\Vert \theta_{\tau}^{*}\right\Vert \left\Vert \theta_{\tau}\right\Vert \right]\left\Vert \theta\right\Vert \left\Vert \theta-\theta^{*}\right\Vert ^{2}\le\frac{\epsilon}{\sigma^{4}}\left\Vert \theta^{*}\right\Vert \left\Vert \theta\right\Vert ^{2}\left\Vert \theta-\theta^{*}\right\Vert ^{2}\le\omega\left(1+\omega\right)^{2}\eta^{4}\epsilon\left\Vert \theta-\theta^{*}\right\Vert .
\]
Hence we have
\begin{equation}
\left\Vert T_{2}\right\Vert \le\left(\omega+2\left(1+\omega\right)+\omega\left(1+\omega\right)\eta^{2}\right)\left(1+\omega\right)\eta^{2}\epsilon\left\Vert \theta-\theta^{*}\right\Vert ,\label{eq:T-2}
\end{equation}
and therefore for $\theta\in B_{r}\left(\theta^{*}\right)$, we have
\[
\left\Vert \mathbb{E}\Gamma(\theta;(Y,X_{s}))\right\Vert \le\frac{1}{\sigma^{2}}\left(\epsilon\left\Vert \theta-\theta^{*}\right\Vert +\left\Vert T_{1}\right\Vert +\left\Vert T_{2}\right\Vert \right)\le\gamma\left(\omega,\eta\right)\left\Vert \theta-\theta^{*}\right\Vert
\]
where by \prettyref{eq:T-1} and \prettyref{eq:T-2},
\[
\gamma\left(\omega,\eta\right)=\frac{1}{\sigma^{2}}\left(\epsilon\left(\omega\xi^{2}+\left(3\omega+2\right)\xi+1\right)+\xi\sqrt{\epsilon\left(1-\epsilon\right)}\right),
\]
and the result follows.
\end{proof}

\subsubsection{\label{subsec:pf-bnd-V-rmc}Proof of \prettyref{lem:bound-V-rmc}}
\begin{proof}
In view of \prettyref{eq:V-rmc}, we have
\[
\mathbb{E}V(\theta'|\theta;(Y,X_{s}))=-\frac{1}{2\sigma^{2}}(\theta'-\theta^{*})^{\intercal}\mathbb{E}\Sigma_{\theta}(Y,X_{s})(\theta'-\theta^{*})
\]
and by \prettyref{eq:ora-Sigma}, it follows that
\[
\mathbb{E}\Sigma_{\theta}(Y,X_{s})=\mathbb{E}_{\epsilon}\left[\mathbb{E}_{\theta^{*}}\Sigma_{\theta}(Y,X_{s})\right]=I_{p}+\Sigma_{1}+\Sigma_{2}-\Sigma_{3},
\]
where
\[
\Sigma_{1}\coloneqq\mathbb{E}_{\epsilon}\left[\frac{1}{\sigma^{2}+\left\Vert \theta_{\tau}\right\Vert ^{2}}\left(\theta_{\tau}(\theta_{s}^{*}-\theta_{s})^{\intercal}+(\theta_{s}^{*}-\theta_{s})\theta_{\tau}^{\intercal}\right)\right],
\]
\[
\Sigma_{2}\coloneqq\mathbb{E}_{\epsilon}\left[\frac{\sigma^{2}+\left\Vert \theta_{\tau}^{*}\right\Vert ^{2}+\left\Vert \theta_{s}^{*}-\theta_{s}\right\Vert ^{2}}{\left(\sigma^{2}+\left\Vert \theta_{\tau}\right\Vert ^{2}\right)^{2}}\theta_{\tau}\theta_{\tau}^{\intercal}\right]\text{ and}
\]
\[
\Sigma_{3}\coloneqq\mathbb{E}_{\epsilon}\left[\frac{1}{\sigma^{2}+\left\Vert \theta_{\tau}\right\Vert ^{2}}\theta_{\tau}\theta_{\tau}^{\intercal}\right].
\]
For $u\in\mathbb{R}^{p}$ and $\theta\in B_{r}\left(\theta^{*}\right)$,
by \prettyref{lem:mean-tau} and Cauchy-Schwartz inequality, we have
\begin{align*}
\left|u^{\intercal}\Sigma_{1}u\right| & \le\frac{2}{\sigma^{2}}\mathbb{E}_{\epsilon}\left[\left|u_{\tau}^{\intercal}\theta_{\tau}\right|\left|\left(\theta_{s}^{*}-\theta_{s}\right)^{\intercal}u_{s}\right|\right]\le\frac{2}{\sigma^{2}}\mathbb{E}_{\epsilon}\left[\left|u_{\tau}^{\intercal}\theta_{\tau}\right|^{2}\right]^{\frac{1}{2}}\mathbb{E}_{\epsilon}\left[\left|\left(\theta_{s}^{*}-\theta_{s}\right)^{\intercal}u_{s}\right|^{2}\right]^{\frac{1}{2}}\\
 & \le\frac{2}{\sigma^{2}}\sqrt{\epsilon\left(1-\epsilon\right)}\left|u^{\intercal}\theta\right|\left|\left(\theta^{*}-\theta\right)^{\intercal}u\right|\le2\omega\left(1+\omega\right)\eta^{2}\sqrt{\epsilon\left(1-\epsilon\right)}\left\Vert u\right\Vert ^{2},
\end{align*}
\[
\left|u^{\intercal}\Sigma_{3}u\right|\le\frac{1}{\sigma^{2}}\mathbb{E}_{\epsilon}\left[\left|u_{\tau}^{\intercal}\theta_{\tau}\right|^{2}\right]\le\frac{\epsilon}{\sigma^{2}}\left|u^{\intercal}\theta\right|^{2}\le\left(1+\omega\right)^{2}\eta^{2}\epsilon\left\Vert u\right\Vert ^{2},
\]
and $u^{\intercal}\Sigma_{2}u\ge0$, since $\Sigma_{2}$ is positive
semi-definite. Then
\begin{align*}
u^{\intercal}\mathbb{E}\Sigma_{\theta}(Y,X_{s})u & =u^{\intercal}\left(I_{p}+\Sigma_{1}+\Sigma_{2}-\Sigma_{3}\right)u\\
 & \ge\left(1-2\omega\left(1+\omega\right)\eta^{2}\sqrt{\epsilon\left(1-\epsilon\right)}-\left(1+\omega\right)^{2}\eta^{2}\epsilon\right)\left\Vert u\right\Vert ^{2},
\end{align*}
and it follows that
\[
\mathbb{E}V(\theta'|\theta;(Y,X_{s}))\le-\nu\left(\omega,\eta\right)\left\Vert \theta'-\theta^{*}\right\Vert ^{2}\text{ for }\theta\in B_{r}\left(\theta^{*}\right)\text{ and }\theta'\in\mathbb{R}^{p},
\]
where
\[
\nu\left(\omega,\eta\right)=\frac{1}{2\sigma^{2}}\left(1-2\omega\left(1+\omega\right)\eta^{2}\sqrt{\epsilon\left(1-\epsilon\right)}-\left(1+\omega\right)^{2}\eta^{2}\epsilon\right),
\]
and the lemma is proved.
\end{proof}

\subsection{Proofs for Empirical Convergence}

We prove lemmas for empirical convergence of the model, we begin with
some basic facts.

\subsubsection{Preliminaries}

To ease notations, we omit the dependence of $(Y,X_{s})$ in $\mu_{\theta}$,
$b_{\theta}$ and $\tau$ in $A_{\theta}$ etc. in this section. For
$\tau\in\{\circ,1\}^{p}$, denote $[\tau]\coloneqq\left\{ j\in\mathbb{N}\mid\tau^{j}=1\right\} $.
We first prove some technical results for related random variables.
Recall in this model, $Y=\left\langle \theta^{*},X\right\rangle +W$
with $X\sim\mathcal{N}\left(0,I_{p}\right)$ and $W\sim\mathcal{N}\left(0,\sigma^{2}\right)$.
\begin{lem}
\label{lem:emp-pre-1-rmc}For $u\in\mathbb{R}^{p}$, missing pattern
$\tau\in\{\circ,1\}^{p}$ and $s=\ones-\tau$, the random variable
$u_{s}^{\intercal}X_{s}$ is Gaussian with Orlicz norm $\left\Vert u_{s}^{\intercal}X_{s}\right\Vert _{\psi_{2}}=\left\Vert u_{s}\right\Vert $
and $Y$ is Gaussian with Orlicz norm $\left\Vert Y\right\Vert _{\psi_{2}}\le\sqrt{\left\Vert \theta^{*}\right\Vert ^{2}+\sigma^{2}}$,
while $Y-u_{s}^{\intercal}X_{s}$ is sub-gaussian with Orlicz norm
$\left\Vert Y-u_{s}^{\intercal}X_{s}\right\Vert _{\psi_{2}}\le\left\Vert u_{\tau}\right\Vert +\sigma.$
\end{lem}
\begin{proof}
By rotation invariance of Gaussian variables, we have
\[
\text{Var}\left(u_{s}^{\intercal}X_{s}\right)=\text{Var}\left(\sum_{j\in[s]}u^{j}X^{j}\right)=\sum_{j\in[s]}\text{Var}\left(u^{j}X^{j}\right)=\sum_{j\in[s]}\left|u^{j}\right|^{2}=\left\Vert u_{s}\right\Vert ^{2},
\]
hence $\left\Vert u_{s}^{\intercal}X_{s}\right\Vert _{\psi_{2}}=\left\Vert u_{s}\right\Vert $.
Since $Y=\theta^{*\intercal}X+W$, it is clearly Gaussian and
\[
\text{Var}\left(Y\right)=\text{Var}\left(\theta^{*\intercal}X\right)+\text{Var}\left(W\right)=\left\Vert \theta^{*}\right\Vert ^{2}+\sigma^{2},
\]
hence $\left\Vert Y\right\Vert _{\psi_{2}}=\sqrt{\left\Vert \theta^{*}\right\Vert ^{2}+\sigma^{2}}$.
Now since $Y-\theta_{s}^{*\intercal}X_{s}=\theta_{\tau}^{*\intercal}X_{\tau}+W$,
then
\[
\left\Vert Y-u_{s}^{\intercal}X_{s}\right\Vert _{\psi_{2}}\le\left\Vert u_{\tau}^{\intercal}X_{\tau}\right\Vert _{\psi_{2}}+\left\Vert W\right\Vert _{\psi_{2}}=\left\Vert u_{\tau}\right\Vert +\sigma.
\]
\end{proof}
\begin{lem}
\label{lem:emp-pre-2-rmc}For $u\in\mathbb{R}^{p}$, missing pattern
$\tau\in\{\circ,1\}^{p}$ and $s=\ones-\tau$, the random variable
$u^{\intercal}\mu_{\theta}$ is sub-gaussian with Orlicz norm $\left\Vert u^{\intercal}\mu_{\theta}\right\Vert _{\psi_{2}}\le\sqrt{3}\left\Vert u\right\Vert $
and $u^{\intercal}\mu_{\theta}Y$ is sub-exponential with Orlicz norm
$\left\Vert u^{\intercal}\mu_{\theta}Y\right\Vert _{\psi_{1}}\le CK\left\Vert u\right\Vert $
where $K\coloneqq\left\Vert \theta^{*}\right\Vert +\sigma$.
\end{lem}
\begin{proof}
It follows from \prettyref{lem:emp-pre-1-rmc} that $u^{\intercal}b_{\theta}=\frac{u_{\tau}^{\intercal}\theta_{\tau}}{\sigma^{2}+\left\Vert \theta_{\tau}\right\Vert ^{2}}\left(Y-\theta_{s}^{\intercal}X_{s}\right)$
is sub-gaussian with Orlicz norm $\left\Vert u^{\intercal}b_{\theta}\right\Vert _{\psi_{2}}\le\frac{\left\Vert \theta_{\tau}\right\Vert +\sigma}{\sigma^{2}+\left\Vert \theta_{\tau}\right\Vert ^{2}}\left|u_{\tau}^{\intercal}\theta_{\tau}\right|$
and hence $u^{\intercal}\mu_{\theta}=u^{\intercal}\left(X_{s}+b_{\theta}\right)$
is sub-gaussian with Orlicz norm
\begin{align}
\left\Vert u^{\intercal}\mu_{\theta}\right\Vert _{\psi_{2}} & \le\left\Vert u_{s}\right\Vert +\frac{\left\Vert \theta_{\tau}\right\Vert +\sigma}{\sigma^{2}+\left\Vert \theta_{\tau}\right\Vert ^{2}}\left|u_{\tau}^{\intercal}\theta_{\tau}\right|\nonumber \\
 & \le\frac{1}{\sigma^{2}+\left\Vert \theta_{\tau}\right\Vert ^{2}}\left[\left(\sigma^{2}+\left\Vert \theta_{\tau}\right\Vert ^{2}\right)\left\Vert u_{s}\right\Vert +\left(\left\Vert \theta_{\tau}\right\Vert ^{2}+\sigma\left\Vert \theta_{\tau}\right\Vert \right)\left\Vert u_{\tau}\right\Vert \right]\nonumber \\
 & \le\frac{1}{\sigma^{2}+\left\Vert \theta_{\tau}\right\Vert ^{2}}\left[\left(\sigma^{2}+\left\Vert \theta_{\tau}\right\Vert ^{2}\right)^{2}+\left\Vert \theta_{\tau}\right\Vert ^{2}\left(\left\Vert \theta_{\tau}\right\Vert +\sigma\right)^{2}\right]^{\frac{1}{2}}\left[\left\Vert u_{s}\right\Vert ^{2}+\left\Vert u_{\tau}\right\Vert ^{2}\right]^{\frac{1}{2}}\nonumber \\
 & \le\frac{1}{\sigma^{2}+\left\Vert \theta_{\tau}\right\Vert ^{2}}\left[\left(\sigma^{2}+\left\Vert \theta_{\tau}\right\Vert ^{2}\right)\left(\sigma^{2}+3\left\Vert \theta_{\tau}\right\Vert ^{2}\right)\right]^{\frac{1}{2}}\left\Vert u\right\Vert \le\sqrt{3}\left\Vert u\right\Vert .\label{eq:mu-sub-gau-rmc}
\end{align}
Hence $u^{\intercal}\mu_{\theta}Y$ is sub-exponential with Orlicz
norm
\begin{equation}
\left\Vert u^{\intercal}\mu_{\theta}Y\right\Vert _{\psi_{1}}\le C_{1}\left\Vert u^{\intercal}\mu_{\theta}\right\Vert _{\psi_{2}}\left\Vert Y\right\Vert _{\psi_{2}}\le\sqrt{3}C_{1}\left\Vert u\right\Vert \sqrt{\left\Vert \theta^{*}\right\Vert ^{2}+\sigma^{2}}\le CK\left\Vert u\right\Vert ,\label{eq:mu-y-sub-exp-rmc}
\end{equation}
where $K\coloneqq\left\Vert \theta^{*}\right\Vert +\sigma$, since
$\left(\frac{K}{\sqrt{2}}\le\right)\sqrt{\left\Vert \theta^{*}\right\Vert ^{2}+\sigma^{2}}\le K$.
\end{proof}
\begin{lem}
\label{lem:emp-pre-3-rmc}For $u,v\in\mathbb{R}^{p}$, missing pattern
$\tau\in\{\circ,1\}^{p}$ and $s=\ones-\tau$, the random variable
$u^{\intercal}A_{\theta}v$ is bounded hence sub-gaussian with Orlicz
norm $\left\Vert u^{\intercal}A_{\theta}v\right\Vert _{\psi_{2}}\le2\left\Vert u\right\Vert \left\Vert v\right\Vert $
and $u^{\intercal}\Sigma_{\theta}v$ is sub-exponential with Orlicz
norm $\left\Vert u^{\intercal}\Sigma_{\theta}v\right\Vert _{\psi_{1}}\le C\left\Vert u\right\Vert \left\Vert v\right\Vert $.
\end{lem}
\begin{proof}
In view of definition \prettyref{eq:A-theta-rmc}, we have
\[
\left|u^{\intercal}A_{\theta}v\right|\le\left|u_{\tau}^{\intercal}v_{\tau}\right|+\frac{\left|u_{\tau}^{\intercal}\theta_{\tau}\right|\left|v_{\tau}^{\intercal}\theta_{\tau}\right|}{\sigma^{2}+\left\Vert \theta_{\tau}\right\Vert ^{2}}\le\left(1+\frac{\left\Vert \theta_{\tau}\right\Vert ^{2}}{\sigma^{2}+\left\Vert \theta_{\tau}\right\Vert ^{2}}\right)\left\Vert u\right\Vert \left\Vert v\right\Vert <2\left\Vert u\right\Vert \left\Vert v\right\Vert ,
\]
and since $u^{\intercal}\Sigma_{\theta}v=\left(u^{\intercal}\mu_{\theta}\right)\left(v^{\intercal}\mu_{\theta}\right)+u^{\intercal}A_{\theta}v$,
the result follows from \prettyref{lem:emp-pre-2-rmc}.
\end{proof}
\begin{lem}
\label{lem:emp-pre-4-rmc}For $u\in\mathbb{R}^{p}$, missing pattern
$\tau\in\{\circ,1\}^{p}$ and $s=\ones-\tau$, the random variable
$u^{\intercal}\left(\mu_{\theta}-\mu_{\theta^{*}}\right)$ is sub-gaussian
with Orlicz norm
\[
\left\Vert u^{\intercal}\left(\mu_{\theta}-\mu_{\theta^{*}}\right)\right\Vert _{\psi_{2}}\le\frac{\left(1+\omega\right)\eta}{\sigma}\left(\eta\left(1+\eta\right)\left(2+\omega\right)+1\right)\left\Vert \theta-\theta^{*}\right\Vert \left\Vert u\right\Vert ,
\]
and $u^{\intercal}\left(\mu_{\theta}-\mu_{\theta^{*}}\right)Y$ is
sub-exponential with Orlicz norm
\[
\left\Vert u^{\intercal}\left(\mu_{\theta}-\mu_{\theta^{*}}\right)Y\right\Vert _{\psi_{1}}\le C\left(\eta\left(1+\eta\right)\left(2+\omega\right)+1\right)\eta\left(1+\eta\right)\left(1+\omega\right)\left\Vert \theta-\theta^{*}\right\Vert \left\Vert u\right\Vert .
\]
\end{lem}
\begin{proof}
In view of definition \prettyref{eq:mu-theta-rmc}, we have
\[
\mu_{\theta}-\mu_{\theta^{*}}=\left(\frac{Y-\theta_{s}^{\intercal}X_{s}}{\sigma^{2}+\left\Vert \theta_{\tau}\right\Vert ^{2}}-\frac{Y-\theta_{s}^{*\intercal}X_{s}}{\sigma^{2}+\left\Vert \theta_{\tau}^{*}\right\Vert ^{2}}\right)\theta_{\tau}+\frac{Y-\theta_{s}^{*\intercal}X_{s}}{\sigma^{2}+\left\Vert \theta_{\tau}^{*}\right\Vert ^{2}}\left(\theta_{\tau}-\theta_{\tau}^{*}\right).
\]
Note the first summand can be rewritten as
\[
\frac{1}{\sigma^{2}+\left\Vert \theta_{\tau}\right\Vert ^{2}}\left(\frac{Y-\theta_{s}^{*\intercal}X_{s}}{\sigma^{2}+\left\Vert \theta_{\tau}^{*}\right\Vert ^{2}}\left(\left\Vert \theta_{\tau}^{*}\right\Vert ^{2}-\left\Vert \theta_{\tau}\right\Vert ^{2}\right)+\left(\theta_{s}-\theta_{s}^{*}\right)^{\intercal}X_{s}\right)\theta_{\tau},
\]
and it follows that,
\begin{align*}
\left|u^{\intercal}\left(\mu_{\theta}-\mu_{\theta^{*}}\right)\right| & \le\frac{\left\Vert \theta_{\tau}\right\Vert }{\sigma^{2}}\left(\frac{1}{\sigma^{2}}\left(\left\Vert \theta_{\tau}^{*}\right\Vert +\left\Vert \theta_{\tau}\right\Vert \right)\left|Y-\theta_{s}^{*\intercal}X_{s}\right|\left\Vert \theta_{\tau}^{*}-\theta_{\tau}\right\Vert +\left|\left(\theta_{s}-\theta_{s}^{*}\right)^{\intercal}X_{s}\right|\right)\left\Vert u\right\Vert \\
 & \le\frac{\left(1+\omega\right)\eta}{\sigma^{2}}\left(\left(2+\omega\right)\eta\left|Y-\theta_{s}^{*\intercal}X_{s}\right|\left\Vert \theta_{\tau}^{*}-\theta_{\tau}\right\Vert +\sigma\left|\left(\theta_{s}-\theta_{s}^{*}\right)^{\intercal}X_{s}\right|\right)\left\Vert u\right\Vert .
\end{align*}
Hence in view of \prettyref{lem:emp-pre-1-rmc}, $u^{\intercal}\left(\mu_{\theta}-\mu_{\theta^{*}}\right)$
is sub-gaussian with Orlicz norm
\begin{align*}
\left\Vert u^{\intercal}\left(\mu_{\theta}-\mu_{\theta^{*}}\right)\right\Vert _{\psi_{2}} & \le\frac{\left(1+\omega\right)\eta}{\sigma^{2}}\left(\left(2+\omega\right)\eta\left(\left\Vert \theta_{\tau}^{*}\right\Vert +\sigma\right)\left\Vert \theta_{\tau}^{*}-\theta_{\tau}\right\Vert +\sigma\left\Vert \theta_{s}-\theta_{s}^{*}\right\Vert \right)\left\Vert u\right\Vert \\
 & \le\frac{\left(1+\omega\right)\eta}{\sigma}\left(\eta\left(1+\eta\right)\left(2+\omega\right)+1\right)\left\Vert \theta-\theta^{*}\right\Vert \left\Vert u\right\Vert .
\end{align*}
Therefore $u^{\intercal}\left(\mu_{\theta}-\mu_{\theta^{*}}\right)Y$
is sub-exponential with Orlicz norm
\begin{align*}
\left\Vert u^{\intercal}\left(\mu_{\theta}-\mu_{\theta^{*}}\right)Y\right\Vert _{\psi_{1}} & \le C\left\Vert u^{\intercal}\left(\mu_{\theta}-\mu_{\theta^{*}}\right)\right\Vert _{\psi_{2}}\left\Vert Y\right\Vert _{\psi_{2}}\\
 & \le C\left(\eta\left(1+\eta\right)\left(2+\omega\right)+1\right)\eta\left(1+\eta\right)\left(1+\omega\right)\left\Vert \theta-\theta^{*}\right\Vert \left\Vert u\right\Vert .
\end{align*}
\end{proof}
\begin{lem}
\label{lem:emp-pre-5-rmc}For $u,v\in\mathbb{R}^{p}$, missing pattern
$\tau\in\{\circ,1\}^{p}$ and $s=\ones-\tau$, the random variable
$u^{\intercal}\left(A_{\theta}-A_{\theta^{*}}\right)v$ is bounded
hence sub-gaussian with Orlicz norm
\[
\left\Vert u^{\intercal}\left(A_{\theta}-A_{\theta^{*}}\right)v\right\Vert _{\psi_{2}}\le\frac{1}{\sigma}\left(\left(1+\omega\right)\eta^{2}+1\right)\left(2+\omega\right)\eta\left\Vert \theta^{*}-\theta\right\Vert \left\Vert u\right\Vert \left\Vert v\right\Vert ,
\]
and $u^{\intercal}\left(\mu_{\theta}\mu_{\theta}^{\intercal}-\mu_{\theta^{*}}\mu_{\theta^{*}}^{\intercal}\right)v$
is sub-exponential with Orlicz norm
\[
\left\Vert u^{\intercal}\left(\mu_{\theta}\mu_{\theta}^{\intercal}-\mu_{\theta^{*}}\mu_{\theta^{*}}^{\intercal}\right)v\right\Vert _{\psi_{1}}\le\sqrt{3}C\frac{\left(1+\omega\right)\eta}{\sigma}\left(\eta\left(1+\eta\right)\left(2+\omega\right)+1\right)\left\Vert \theta-\theta^{*}\right\Vert \left\Vert u\right\Vert \left\Vert v\right\Vert .
\]
\end{lem}
\begin{proof}
Since we can write
\begin{align*}
A_{\theta}-A_{\theta^{*}} & =\frac{1}{\sigma^{2}+\left\Vert \theta_{\tau}^{*}\right\Vert ^{2}}\theta_{\tau}^{*}\theta_{\tau}^{*\intercal}-\frac{1}{\sigma^{2}+\left\Vert \theta_{\tau}\right\Vert ^{2}}\theta_{\tau}\theta_{\tau}^{\intercal}\\
 & =\frac{\left\Vert \theta_{\tau}\right\Vert ^{2}-\left\Vert \theta_{\tau}^{*}\right\Vert ^{2}}{\left(\sigma^{2}+\left\Vert \theta_{\tau}^{*}\right\Vert ^{2}\right)\left(\sigma^{2}+\left\Vert \theta_{\tau}\right\Vert ^{2}\right)}\theta_{\tau}^{*}\theta_{\tau}^{*\intercal}+\frac{1}{\sigma^{2}+\left\Vert \theta_{\tau}\right\Vert ^{2}}\left(\theta_{\tau}^{*}\theta_{\tau}^{*\intercal}-\theta_{\tau}\theta_{\tau}^{\intercal}\right)
\end{align*}
and
\[
\theta_{\tau}^{*}\theta_{\tau}^{*\intercal}-\theta_{\tau}\theta_{\tau}^{\intercal}=\left(\theta_{\tau}^{*}-\theta_{\tau}\right)\theta_{\tau}^{*\intercal}+\theta_{\tau}\left(\theta_{\tau}^{*}-\theta_{\tau}\right)^{\intercal},
\]
it follows that
\[
\left|u^{\intercal}\left(\theta_{\tau}^{*}\theta_{\tau}^{*\intercal}-\theta_{\tau}\theta_{\tau}^{\intercal}\right)v\right|\le\left(\left\Vert \theta_{\tau}^{*}\right\Vert +\left\Vert \theta_{\tau}\right\Vert \right)\left\Vert \theta_{\tau}^{*}-\theta_{\tau}\right\Vert \left\Vert u\right\Vert \left\Vert v\right\Vert ,
\]
and hence
\begin{align*}
\left|u^{\intercal}\left(A_{\theta}-A_{\theta^{*}}\right)v\right| & \le\left(\frac{\left\Vert \theta_{\tau}^{*}\right\Vert \left\Vert \theta_{\tau}\right\Vert }{\sigma^{2}+\left\Vert \theta_{\tau}^{*}\right\Vert ^{2}}+1\right)\frac{\left\Vert \theta_{\tau}^{*}\right\Vert +\left\Vert \theta_{\tau}\right\Vert }{\sigma^{2}+\left\Vert \theta_{\tau}\right\Vert ^{2}}\left\Vert \theta_{\tau}^{*}-\theta_{\tau}\right\Vert \left\Vert u\right\Vert \left\Vert v\right\Vert \\
 & \le\frac{1}{\sigma}\left(\left(1+\omega\right)\eta^{2}+1\right)\left(2+\omega\right)\eta\left\Vert \theta_{\tau}^{*}-\theta_{\tau}\right\Vert \left\Vert u\right\Vert \left\Vert v\right\Vert .
\end{align*}
Similarly, since
\[
\mu_{\theta}\mu_{\theta}^{\intercal}-\mu_{\theta^{*}}\mu_{\theta^{*}}^{\intercal}=\left(\mu_{\theta}-\mu_{\theta^{*}}\right)\mu_{\theta}^{\intercal}+\mu_{\theta^{*}}\left(\mu_{\theta}-\mu_{\theta^{*}}\right)^{\intercal}
\]
and in view of \prettyref{lem:emp-pre-2-rmc} and \prettyref{lem:emp-pre-4-rmc},
$u^{\intercal}\left(\mu_{\theta}\mu_{\theta}^{\intercal}-\mu_{\theta^{*}}\mu_{\theta^{*}}^{\intercal}\right)v$
is sub-exponential with Orlicz norm
\begin{align*}
\left\Vert u^{\intercal}\left(\mu_{\theta}\mu_{\theta}^{\intercal}-\mu_{\theta^{*}}\mu_{\theta^{*}}^{\intercal}\right)v\right\Vert _{\psi_{1}} & \le\left\Vert u^{\intercal}\left(\mu_{\theta}-\mu_{\theta^{*}}\right)\mu_{\theta}^{\intercal}v\right\Vert _{\psi_{1}}+\left\Vert u^{\intercal}\mu_{\theta^{*}}\left(\mu_{\theta}-\mu_{\theta^{*}}\right)^{\intercal}v\right\Vert _{\psi_{1}}\\
 & \le C\left(\left\Vert u^{\intercal}\left(\mu_{\theta}-\mu_{\theta^{*}}\right)\right\Vert _{\psi_{2}}\left\Vert v^{\intercal}\mu_{\theta}\right\Vert _{\psi_{2}}+\left\Vert u^{\intercal}\mu_{\theta}\right\Vert _{\psi_{2}}\left\Vert v^{\intercal}\left(\mu_{\theta}-\mu_{\theta^{*}}\right)\right\Vert _{\psi_{2}}\right)\\
 & \le\sqrt{3}C\frac{\left(1+\omega\right)\eta}{\sigma}\left(\eta\left(1+\eta\right)\left(2+\omega\right)+1\right)\left\Vert \theta-\theta^{*}\right\Vert \left\Vert u\right\Vert \left\Vert v\right\Vert .
\end{align*}
\end{proof}

\subsubsection{\label{subsec:pf-eps-1-rmc}Proof of \prettyref{lem:epsilon-1-rmc}}
\begin{proof}
By \prettyref{eq:Gam-rmc}, for $u\in\mathbb{S}^{p-1}$, we have
\[
u^{\intercal}\Gamma(\theta;(Y,X_{s}))=\frac{1}{\sigma^{2}}\left[u^{\intercal}\left(\mu_{\theta}-\mu_{\theta^{*}}\right)Y-u^{\intercal}\left(\Sigma_{\theta}-\Sigma_{\theta^{*}}\right)\theta^{*}\right]=\frac{1}{\sigma^{2}}\left[D_{1}-D_{2}-D_{3}\right],
\]
where $D_{1}\coloneqq u^{\intercal}\left(\mu_{\theta}-\mu_{\theta^{*}}\right)Y$,
$D_{2}\coloneqq u^{\intercal}\left(\mu_{\theta}\mu_{\theta}^{\intercal}-\mu_{\theta^{*}}\mu_{\theta^{*}}^{\intercal}\right)\theta^{*}$
and $D_{3}\coloneqq u^{\intercal}\left(A_{\theta}-A_{\theta^{*}}\right)\theta^{*}$,
and it follows from \prettyref{lem:emp-pre-4-rmc} and \prettyref{lem:emp-pre-5-rmc}
that $D_{1}$ and $D_{2}$ are sub-exponential with Orlicz norms
\[
\left\Vert D_{1}\right\Vert _{\psi_{1}}\le C_{1}\left(\eta\left(1+\eta\right)\left(2+\omega\right)+1\right)\eta\left(1+\eta\right)\left(1+\omega\right)\left\Vert \theta-\theta^{*}\right\Vert \text{ and}
\]
\[
\left\Vert D_{2}\right\Vert _{\psi_{1}}\le C_{2}\left(\eta\left(1+\eta\right)\left(2+\omega\right)+1\right)\left(1+\omega\right)\eta^{2}\left\Vert \theta-\theta^{*}\right\Vert ,
\]
while $D_{3}$, which is \emph{independent} of $D_{1}$ and $D_{2}$,
is sub-gaussian with Orlicz norm
\[
\left\Vert D_{3}\right\Vert _{\psi_{2}}\le\left(\left(1+\omega\right)\eta^{2}+1\right)\left(2+\omega\right)\eta^{2}\left\Vert \theta^{*}-\theta\right\Vert .
\]
It follows that $u^{\intercal}\Gamma(\theta;(Y,X_{s}))$ is sub-exponential
with Orlicz norm
\[
\left\Vert u^{\intercal}\Gamma(\theta;(Y,X_{s}))\right\Vert _{\psi_{1}}\le\frac{C(\eta,\omega)}{\sigma^{2}}\left\Vert \theta^{*}-\theta\right\Vert ,
\]
where $C(\eta,\omega)=O\left(\left(1+\omega\right)^{2}\left(1+\eta\right)^{4}\right)$.
Now the result follows from \prettyref{lem:constr-sub-exp} on the
concentration of sub-exponential random vectors.
\end{proof}

\subsubsection{\label{subsec:pf-esp-2-rmc}Proof of \prettyref{lem:epsilon-2-rmc}}
\begin{proof}
In view of \prettyref{eq:V-rmc} and by \prettyref{lem:emp-pre-3-rmc},
$V(\theta'|\theta;(Y,X_{s}))$ is sub-exponential with Orlicz norm
\[
\left\Vert V(\theta'|\theta;(Y,X_{s}))\right\Vert _{\psi_{1}}\le\frac{C}{\sigma^{2}}\left\Vert \theta'-\theta^{*}\right\Vert ^{2},
\]
and the result follows from \prettyref{lem:constr-sub-exp} on the
concentration of sub-exponential random variables.
\end{proof}

\subsubsection{\label{subsec:pf-esp-s-rmc}Proof of \prettyref{lem:epsilon-s-rmc}}
\begin{proof}
In view of \prettyref{eq:E-rmc}, for $u\in\mathbb{S}^{p-1}$, we
have
\[
u^{\intercal}\mathcal{E}(Y,X_{s})=\frac{1}{\sigma^{2}}\left[u^{\intercal}\mu_{\theta^{*}}Y-u^{\intercal}\Sigma_{\theta^{*}}\theta^{*}\right].
\]
Then $u^{\intercal}\mu_{\theta^{*}}Y$ is sub-exponential with Orlicz
norm $\left\Vert u^{\intercal}\mu_{\theta^{*}}Y\right\Vert _{\psi_{1}}\le C_{1}K$
by \prettyref{lem:emp-pre-2-rmc}, and $u^{\intercal}\Sigma_{\theta^{*}}\theta^{*}$
is sub-exponential with Orlicz norm $\left\Vert u^{\intercal}\Sigma_{\theta^{*}}\theta^{*}\right\Vert _{\psi_{1}}\le C_{2}\left\Vert \theta^{*}\right\Vert $
by \prettyref{lem:emp-pre-3-rmc}. Hence $u^{\intercal}\mathcal{E}(Y,X_{s})$
is sub-exponential with Orlicz norm
\[
\left\Vert u^{\intercal}\mathcal{E}(Y,X_{s})\right\Vert _{\psi_{1}}\le\frac{1}{\sigma^{2}}\left(C_{1}K+C_{2}\left\Vert \theta^{*}\right\Vert \right)\le\frac{C}{\sigma}\left(1+\eta\right),
\]
and the result follows from \prettyref{lem:constr-sub-exp} on the
concentration of sub-exponential random vectors.
\end{proof}

\section{Miscellaneous Results and Proofs}

We collect various results and proofs in this section.

\subsection{\label{subsec:pf-sel-con-loglik}Proof of \prettyref{prop:sel-con-loglik}}
\begin{proof}
We show that $L_{*}\left(\theta\right)\le L_{*}\left(\theta^{*}\right)$
for $\theta\in\Omega$. By definition
\[
L_{*}(\theta)=\int_{\mathcal{Y}}\left(\log p_{\theta}(y)\right)p_{\theta^{*}}(y)dy\le\int_{\mathcal{Y}}\left(\log p_{\theta^{*}}(y)\right)p_{\theta^{*}}(y)dy=L_{*}(\theta^{*}),
\]
where the inequality follows from a version of the Jensen's inequality
in \prettyref{lem:jen-ineq}.
\end{proof}

\subsection{\label{subsec:pf-self-const}Proof of \prettyref{prop:self-const}}
\begin{proof}
We show that $Q_{*}\left(\theta'|\theta^{*}\right)\le Q_{*}\left(\theta^{*}|\theta^{*}\right)$
for $\theta'\in\Omega$. By definition
\begin{align*}
Q_{*}(\theta'|\theta^{*}) & =\int_{\mathcal{Y}}\left(\int_{\mathcal{Z}(y)}\log\left(f_{\theta'}(y,z)\right)k_{\theta^{*}}(z|y)dz\right)p_{\theta^{*}}(y)dy\\
 & =\int_{\mathcal{\mathcal{Y}\times\mathcal{Z}}}\log\left(f_{\theta'}(y,z)\right)k_{\theta^{*}}(z|y)p_{\theta^{*}}(y)dzdy\\
 & =\int_{\mathcal{Y}\times\mathcal{Z}}\log\left(f_{\theta'}(y,z)\right)f_{\theta^{*}}(y,z)dzdy,
\end{align*}
then the result follows from a version of the Jensen's inequality
in \prettyref{lem:jen-ineq}.
\end{proof}

\subsection{\label{subsec:inter-con-ineq}An Interpretation of the Convergence
Inequality}
\begin{lem}
Suppose $\theta^{*}\in\mathbb{R}^{p}$, $\kappa<1$, $\varepsilon>0$
and a sequence $\{\theta^{t}\}_{t=0}^{T}$ such that $\left\Vert \theta^{t}-\theta^{*}\right\Vert >\varepsilon$
for $0\le t\le T$. Then it satisfies the inequality
\[
\left\Vert \theta^{t}-\theta^{*}\right\Vert \le\kappa^{t}\left\Vert \theta^{0}-\theta^{*}\right\Vert +\varepsilon\text{ for }0\le t\le T,
\]
if and only if there exists a sequence $\{\zeta^{t}\}_{t=0}^{T}$
such that $\zeta^{t}\in\mathbb{S}_{\varepsilon}^{p-1}\left(\theta^{*}\right)$
and
\[
\left\Vert \theta^{t}-\zeta^{t}\right\Vert \le\kappa^{t}\left\Vert \theta^{0}-\theta^{*}\right\Vert \text{ for }0\le t\le T,
\]
where $\mathbb{S}_{\varepsilon}^{p-1}\left(\theta^{*}\right)\coloneqq\{u\in\mathbb{R}^{p}\mid\left\Vert u-\theta^{*}\right\Vert =\varepsilon\}$
is the sphere centered at $\theta^{*}$ with radius $\varepsilon$.
\end{lem}
\begin{proof}
Sufficiency. By triangle inequality
\[
\left\Vert \theta^{t}-\theta^{*}\right\Vert \le\left\Vert \theta^{t}-\zeta^{t}\right\Vert +\left\Vert \zeta^{t}-\theta^{*}\right\Vert \le\kappa^{t}\left\Vert \theta^{0}-\theta^{*}\right\Vert +\varepsilon.
\]

Necessity. Let $\zeta^{t}\coloneqq(1-\lambda^{t})\theta^{*}+\lambda^{t}\theta^{t}$
where $\lambda^{t}\coloneqq\frac{\varepsilon}{\left\Vert \theta^{t}-\theta^{*}\right\Vert }<1$
for $0\le t\le T$. Then since $\left\Vert \zeta^{t}-\theta^{*}\right\Vert =\varepsilon$
and
\[
\left\Vert \theta^{t}-\zeta^{t}\right\Vert +\left\Vert \zeta^{t}-\theta^{*}\right\Vert =\left\Vert \theta^{t}-\theta^{*}\right\Vert \le\kappa^{t}\left\Vert \theta^{0}-\theta^{*}\right\Vert +\varepsilon,
\]
the result follows.
\end{proof}
\begin{rem*}
It is clear from the proof that for each $t$, the point $\zeta^{t}$
is simply the intersection of the line segment joining $\theta^{*}$
and $\theta^{t}$ with the sphere $\mathbb{S}_{\varepsilon}^{p-1}\left(\theta^{*}\right)$,
and hence $\left\Vert \theta^{t}-\zeta^{t}\right\Vert $ is simply
the distance between $\theta^{t}$ and the ball $B_{\varepsilon}(\theta^{*})$.
In the context of an empirical EM sequence, when $\theta^{t}$ lies
outside the ball $B_{\varepsilon}(\theta^{*})$ of statistical error,
it ``converges'' geometrically onto it at the rate $\kappa$.
\end{rem*}

\subsection{\label{subsec:infor-mat}A Digression to the Theory of Information
Matrices}

The Fisher information matrix, the complete and missing information
matrices as well as the convergence rate matrix are classical objects
for analyzing the asymptotic convergence of the EM algorithm \cite{Dempster,Redner,Meng-1,Meng-2}.
In this section, we briefly formulate and extend the classical information
matrix theory to make connections with our analysis of oracle convergence
of the EM algorithm.

On both sides of \prettyref{eq:int-1-em}, differentiating twice with
respect to $\theta'$ and taking conditional expectation of $Z$ given
$y$ at parameter $\theta$, we have
\[
\mathcal{I}(\theta';y)=\mathcal{I}_{c}(\theta'|\theta;y)-\mathcal{I}_{m}(\theta'|\theta;y),
\]
where we define $\mathcal{I}(\theta';y)\coloneqq-\nabla_{1}\nabla_{1}^{\intercal}L(\theta';y)$
as the \emph{negative} of the Hessian matrix of the stochastic log-likelihood
function and define\footnote{In this section the differential operators $\nabla$ and $\nabla\nabla^{\intercal}$
are with respect to the parameter $\theta'$.}
\[
\mathcal{I}_{c}(\theta'|\theta;y)\coloneqq-\int_{\mathcal{Z}(y)}\nabla\nabla^{\intercal}\left(\log f_{\theta'}(y,z)\right)k_{\theta}(z|y)dz=-\nabla_{1}\nabla_{1}^{\intercal}Q(\theta'|\theta;y),
\]
\[
\mathcal{I}_{m}(\theta'|\theta;y)\coloneqq-\int_{\mathcal{Z}(y)}\nabla\nabla^{\intercal}\left(\log k_{\theta'}(z|y)\right)k_{\theta}(z|y)dz
\]
for $\theta\in\Omega$, whenever these matrices are well-defined.
Note $\mathcal{I}_{m}(\theta|\theta;y)$ is positive semi-definite
for $y\in\mathcal{Y}$ by \prettyref{lem:fisher-ident}.

By taking expectations, we define the \emph{observed information matrix}
\begin{equation}
\mathcal{I}(\theta')\coloneqq\int_{\mathcal{Y}}\mathcal{I}(\theta';y)p_{\theta^{*}}(y)dy=-\nabla_{1}\nabla_{1}^{\intercal}L_{*}(\theta'),\label{eq:obs-inf-mat}
\end{equation}
the \emph{complete information matrix}
\begin{equation}
\mathcal{I}_{c}(\theta'|\theta)\coloneqq\int_{\mathcal{Y}}\mathcal{I}_{c}(\theta'|\theta;y)p_{\theta^{*}}(y)dy=-\nabla_{1}\nabla_{1}^{\intercal}Q_{*}(\theta'|\theta),\label{eq:cmp-inf-mat}
\end{equation}
and the \emph{missing information matrix}
\[
\mathcal{I}_{m}(\theta'|\theta)\coloneqq\int_{\mathcal{Y}}\mathcal{I}_{m}(\theta'|\theta;y)p_{\theta^{*}}(y)dy.
\]
We obtain the \emph{oracle information equation }or the \emph{missing
information principle} (\cite{Orchard,McLachlan}) at the \emph{population}
level
\begin{equation}
\mathcal{I}(\theta')=\mathcal{I}_{c}(\theta'|\theta)-\mathcal{I}_{m}(\theta'|\theta).\label{eq:ora-infor-eq}
\end{equation}

At the true population parameter $\theta^{*}$ and by \prettyref{lem:fisher-ident},
we have
\begin{equation}
\mathcal{I}(\theta^{*})=\int_{\mathcal{Y}}\left[\nabla\log p_{\theta^{*}}(y)\right]\left[\nabla^{\intercal}\log p_{\theta^{*}}(y)\right]p_{\theta^{*}}(y)dy=I(\theta^{*}),\label{eq:fis-inf-mat}
\end{equation}
which is just the Fisher information matrix and is positive semi-definite.
Similarly, we have
\[
\mathcal{I}_{c}(\theta^{*}|\theta)=-\int_{\mathcal{X}}\nabla\nabla^{\intercal}\left(\log f_{\theta^{*}}(y,z)\right)k_{\theta}(z|y)p_{\theta^{*}}(y)dydz,
\]
\[
\mathcal{I}_{m}(\theta^{*}|\theta)=-\int_{\mathcal{X}}\nabla\nabla^{\intercal}\left(\log k_{\theta^{*}}(z|y)\right)k_{\theta}(z|y)p_{\theta^{*}}(y)dzdy
\]
and by \prettyref{lem:fisher-ident},
\[
\mathcal{I}_{c}(\theta^{*}|\theta^{*})=\int_{\mathcal{X}}\left[\nabla\log f_{\theta^{*}}(y,z)\right]\left[\nabla^{\intercal}\log f_{\theta^{*}}(y,z)\right]dydz,
\]
\[
\mathcal{I}_{m}(\theta^{*}|\theta^{*})=\int_{\mathcal{X}}\left[\nabla\log k_{\theta^{*}}(z|y)\right]\left[\nabla^{\intercal}\log k_{\theta^{*}}(z|y)\right]p_{\theta^{*}}(y)dzdy,
\]
which are positive semi-definite.

In classical analysis of parameter estimation by MLE, we usually require
that the Fisher information matrix of the parametric density be positive
definite at $\theta^{*}$. A connection of this condition and our
strong concavity condition of the oracle $Q$-function is made in
\prettyref{prop:fisher-non-empt-V} for which we give the following
proof.
\begin{proof}[Proof of \prettyref{prop:fisher-non-empt-V}]
\label{proof:fisher-non-empt-V}In view of \prettyref{eq:ora-infor-eq},
we have
\[
I(\theta^{*})=\mathcal{I}(\theta^{*})=\mathcal{I}_{c}(\theta^{*}|\theta)-\mathcal{I}_{m}(\theta^{*}|\theta)\text{ for }\theta\in B_{r_{1}}(\theta^{*}),
\]
for some $r_{1}>0$. Since $\mathcal{I}_{m}(\theta^{*}|\theta^{*})$
is positive semi-definite and $\mathcal{I}(\theta^{*})$ is positive
definite by our assumption, $\mathcal{I}_{c}(\theta^{*}|\theta)$
is positive definite at $\theta=\theta^{*}$ and hence its minimal
eigenvalue $\lambda_{\min}(\theta^{*})>0$. Then by continuity of
$\lambda_{\min}$, there exists $0<r_{2}<r_{1}$ such that
\[
\nu:=\frac{1}{3}\inf\left\{ \lambda_{\min}(\theta)\mid\theta\in\overline{B}_{r_{2}}(\theta^{*})\right\} >0.
\]
Now since $\nabla_{1}\nabla_{1}^{\intercal}Q_{*}(\theta^{*}|\theta)=-\mathcal{I}_{c}(\theta^{*}|\theta)$
by \prettyref{eq:cmp-inf-mat}, which implies that there exists $0<r<r_{2}$
such that
\[
Q_{*}(\theta'|\theta)-Q_{*}(\theta^{*}|\theta)-\left\langle \nabla_{1}Q_{*}(\theta^{*}|\theta),\theta'-\theta^{*}\right\rangle \le-\nu\left\Vert \theta'-\theta^{*}\right\Vert ^{2}
\]
whenever $\theta',\theta\in B_{r}(\theta^{*})$ and it follows that
$\nu\in\mathcal{V}(r,r)\neq\varnothing$.
\end{proof}

\subsection{\label{subsec:mea}A Note on the Measurability Issue}

In the statement of some definitions and assumptions in this paper,
we implicitly used the fact that certain \emph{uncountable} operations
of a family of measurable functions preserve the measurability of
the resulting function. To be specific, let $T\subset\mathbb{R}^{q}$
be a (possibly uncountable) index set and $(\mathscr{S},\mathscr{E},\mathbb{P})$
be a probability measure space.
\begin{lem}
If $g(y,\theta):\mathbb{R}\times T\to\overline{\mathbb{R}}$ is a
Borel measurable function, then for any random variable $Y$ on $(\mathscr{S},\mathscr{E},\mathbb{P})$,
the supremum $Z\coloneqq\sup_{\theta\in T}g(Y,\theta)$ is an $\mathscr{E}$-measurable
function hence a random variable.
\end{lem}
\begin{proof}
See Appendix C of \cite{Pollard} for a proof.
\end{proof}
\begin{rem*}
This result can be readily generalized to random vectors $Y=(Y_{1},\cdots,Y_{n})$
on $(\mathscr{S},\mathscr{E},\mathbb{P})$. In all our cases, the
index set $T=B_{r}(\theta^{*})$ or $T=B_{r}(\theta^{*})\times B_{R}(\theta^{*})$,
and as a simple consequence, the sets like
\[
\{\varpi\mid g(Y,\theta)\le a\text{ for }\theta\in T\}=\bigcap_{\theta\in T}\{\varpi\mid g(Y,\theta)\le a\}=\{\varpi\mid Z\le a\}
\]
are indeed measurable. See \cite{Dellacherie,Pollard} for more detailed
discussions on this topic.
\end{rem*}

\section{Auxiliaries}

We give auxiliary results used throughout the paper in this section.

\subsection{Supporting Lemmas}
\begin{lem}
\label{lem:fun-bounds}For real valued functions $f$ and $g$ on
a non-empty set $X$.
\begin{enumerate}[label=(\alph*),ref=(\alph*)]
\item \label{enu:fun-bnd-a} If $\sup f(x)<+\infty$ and $\sup g(x)<+\infty$, then
\[
\left|\sup f(x)-\sup g(x)\right|\le\sup\left|f(x)-g(x)\right|;
\]
\item \label{enu:fun-bnd-b} If $\inf f(x)>-\infty$ and $\inf g(x)>-\infty$, then
\[
\left|\inf f(x)-\inf g(x)\right|\le\sup\left|f(x)-g(x)\right|.
\]
\end{enumerate}
\end{lem}
\begin{proof}
$(a)$ Since $f(x)\le g(x)+\left|f(x)-g(x)\right|$, we have
\[
\sup f(x)\le\sup g(x)+\sup\left|f(x)-g(x)\right|.
\]
Then since $\sup g(x)<+\infty$, subtracting it from both sides yields
\[
\sup f(x)-\sup g(x)\le\sup\left|f(x)-g(x)\right|.
\]
By exchanging the roles of $f$ and $g$, we have $\sup g(x)-\sup f(x)\le\sup\left|g(x)-f(x)\right|$,
which then combines to give the desired result. $(b)$ Apply $(a)$
to $-f$ and $-g$, the result follows.
\end{proof}
\begin{lem}
\label{lem:quad-taylor}If $F(x)=x^{\intercal}Ax+b^{\intercal}x+c$
is a quadratic function of $x\in\mathbb{R}^{p}$, where $A\in\mathbb{R}^{p\times p}$
is symmetric, then $F(x)-F(x_{0})-\left\langle \nabla F(x_{0}),x-x_{0}\right\rangle =(x-x_{0})^{\intercal}A(x-x_{0})$.
\end{lem}
\begin{proof}
This result follows from simple calculation.
\end{proof}
\begin{lem}
\label{lem:jen-ineq}Suppose $f(x)$ and\textup{ $g(x)$} are positive
and Lebesgue integrable functions on $\mathbb{R}^{p}\ (p\ge1)$. If
$\int_{\mathbb{R}^{p}}f(x)dx=\int_{\mathbb{R}^{p}}g(x)dx=1$, then
\[
\int_{\mathbb{R}^{p}}g(x)\log f(x)dx\le\int_{\mathbb{R}^{p}}g(x)\log g(x)dx.
\]
\end{lem}
\begin{proof}
Let $d\mu(x)\coloneqq g(x)dx$, then $\left(\mathbb{R}^{p},\mathscr{B}^{p},\mu\right)$
is clearly a probability measure space. Since $-\log(x)$ is convex
on $\mathbb{R}^{+}$ and $\frac{f(x)}{g(x)}\in L^{1}\left(\mu\right)$,
by applying the Jensen's Inequality \cite{Folland,Kuczma}, one has
\[
-\log\left(\int_{\mathbb{R}^{p}}\frac{f(x)}{g(x)}d\mu(x)\right)\le\int_{\mathbb{R}^{p}}-\log\left(\frac{f(x)}{g(x)}\right)d\mu(x).
\]
Since $\int_{\mathbb{R}^{p}}f(x)dx=1$, the left-hand side of the
above inequality is zero, hence

\[
0\ge\int_{\mathbb{R}^{p}}\log\left(\frac{f(x)}{g(x)}\right)d\mu(x)=\int_{\mathbb{R}^{p}}g(x)\log f(x)dx-\int_{\mathbb{R}^{p}}g(x)\log g(x)dx,
\]
and the lemma follows.
\end{proof}
\begin{lem}
\label{lem:fisher-ident}For a family of parametric densities $\{p_{\theta}(x)\}_{\theta\in\Omega}$
where $\Omega\subseteq\mathbb{R}^{d}$, there holds
\[
-\int_{\mathcal{X}}\nabla\nabla^{\intercal}\left(\log p_{\theta}(x)\right)p_{\theta}(x)dx=\int_{\mathcal{X}}\left[\nabla\log p_{\theta}(x)\right]\left[\nabla^{\intercal}\log p_{\theta}(x)\right]p_{\theta}(x)dx
\]
for $\theta\in\Omega$ and this $d\times d$ matrix is positive semi-definite.
\end{lem}
\begin{proof}
Direct calculation yields
\[
\nabla\nabla^{\intercal}\left(\log p_{\theta}(x)\right)p_{\theta}(x)=\nabla\nabla^{\intercal}p_{\theta}(x)-\left[\nabla\log p_{\theta}(x)\right]\left[\nabla^{\intercal}\log p_{\theta}(x)\right]p_{\theta}(x).
\]
Then the result follows by integration on both sides and noting $\int_{\mathcal{X}}p_{\theta}(x)dx=1$
for $\theta\in\Omega$.
\end{proof}

\subsection{\label{subsec:cov-num}The $\texorpdfstring{\epsilon}{epsilon}$-Net
and Discretization of Norm}

Suppose $(X,d)$ is a compact metric space, a \emph{finite} subset
$N\subseteq X$ is called an \emph{$\epsilon$-net} if for any $x\in X$
there exists $y\in N$ such that $d(x,y)<\epsilon$.\footnote{It follows from the compactness of $X$ that there exists an $\epsilon$-net
for any $\epsilon>0$.} Then $\mathcal{N}_{\epsilon}(X):=\min\{\text{card}(N)\mid N\text{ is an }\epsilon\text{-net of }X\}$
is called the \emph{$\epsilon$-covering number of} $X$. For the
unit sphere $\mathbb{S}^{p-1}$ with induced Euclidean norm, we have
$\mathcal{N}_{\epsilon}(\mathbb{S}^{p-1})<\left(1+\frac{2}{\epsilon}\right)^{p}$.
See \cite{Vershynin} for a proof.
\begin{lem}[Discretization of Norm]
\label{lem:discret-norm}There exists $\{u_{i}\in\mathbb{S}^{p-1}\mid1\le i\le L\}$
with $L<5^{p}$ such that for any $Z\in\mathbb{R}^{p}$, there holds
the inequality
\[
\left\Vert Z\right\Vert \le2\max_{1\le i\le L}u_{i}^{\intercal}Z.
\]
\end{lem}
\begin{proof}
Let $\{u_{i}\in\mathbb{S}^{p-1}\mid1\le i\le L\}$ be a $\frac{1}{2}$-net
of the $\mathbb{S}^{p-1}\subset\mathbb{R}^{p}$ such that $L=\mathcal{N}_{\frac{1}{2}}\left(\mathbb{S}^{p-1}\right)<5^{p}$,
then for any $u\in\mathbb{S}^{p-1}$, there exists $u_{i}$ such that
$\left\Vert u-u_{i}\right\Vert \le\frac{1}{2}$. For a vector $Z\in\mathbb{R}^{p}$,
we have
\[
u^{\intercal}Z\le\left|u^{\intercal}Z-u_{i}^{\intercal}Z\right|+u_{i}^{\intercal}Z\le\left\Vert u-u_{i}\right\Vert \left\Vert Z\right\Vert +\max_{1\le i\le L}u_{i}^{\intercal}Z\le\frac{1}{2}\left\Vert Z\right\Vert +\max_{1\le i\le L}u_{i}^{\intercal}Z.
\]
Then we have
\[
\left\Vert Z\right\Vert =\sup_{u\in\mathbb{S}^{p-1}}u^{\intercal}Z\le\frac{1}{2}\left\Vert Z\right\Vert +\max_{1\le i\le L}u_{i}^{\intercal}Z,
\]
and the lemma follows.
\end{proof}

\subsection{Concentration of Random Vectors}

In this section we prove some Concentration Inequalities for sub-gaussian
and sub-exponential random vectors. We exploit the Orlicz norm in
the proofs. An exposition on Orlicz norm and concentration of random
variables can be found in \cite{Vershynin}. Here, we mention the
following facts.
\begin{lem}
\label{lem:orl-norm}Let $X$ and $Y$ be random variables.
\begin{enumerate}[label=(\alph*),ref=(\alph*)]
\item \label{enu:cent}(Centering) If $X$ has mean $\mathbb{E}X$, then
$\left\Vert X-\mathbb{E}X\right\Vert _{\psi_{i}}\le2\left\Vert X\right\Vert _{\psi_{i}}$
for $i=1,2$;
\item \label{enu:prod-sub-gau}(Product of Sub-gaussians) If $X$ and $Y$
are sub-gaussian, then $XY$ is sub-exponential with Orlicz norm $\left\Vert XY\right\Vert _{\psi_{1}}\le C\left\Vert X\right\Vert _{\psi_{2}}\left\Vert Y\right\Vert _{\psi_{2}}.$
\end{enumerate}
\end{lem}
\begin{proof}
See \cite{Vershynin}.
\end{proof}
A concentration inequality for sub-gaussian random vectors.
\begin{lem}
\label{lem:constr-sub-gau}Suppose $Y$ is a centered random vector
in $\mathbb{R}^{p}$ such that $u^{\intercal}Y$ is sub-gaussian with
Orlicz norm $\left\Vert u^{\intercal}Y\right\Vert _{\psi_{2}}\le K$
for any $u\in\mathbb{S}^{p-1}$. If $Y_{k}$ is an i.i.d. copy of
$Y$ for $k=1,\cdots,n$, then for $\delta>0$ there holds
\[
\left\Vert \frac{1}{n}\sum_{k=1}^{n}Y_{k}\right\Vert \le CK\sqrt{\frac{\log(L/\delta)}{n}}
\]
with probability at least $1-\delta$.
\end{lem}
\begin{proof}
Let $\{u_{i}\}_{i=1}^{L}$ be a $\frac{1}{2}$-net of the unit sphere
$\mathbb{S}^{p-1}\subset\mathbb{R}^{p}$ and let $Z:=\frac{1}{n}\sum_{k=1}^{n}Y_{k}$,
then by \prettyref{lem:discret-norm}, $\left\Vert Z\right\Vert \le2\max_{1\le i\le L}u_{i}^{\intercal}Z$.
By rotation invariance of sub-gaussian variables, we have
\[
\left\Vert u^{\intercal}Z\right\Vert _{\psi_{2}}^{2}=\frac{1}{n^{2}}\left\Vert \sum_{i=1}^{n}u^{\intercal}Y_{k}\right\Vert _{\psi_{2}}^{2}\le\frac{C_{1}}{n^{2}}\sum_{i=1}^{n}\left\Vert u^{\intercal}Y_{k}\right\Vert _{\psi_{2}}^{2}\le\frac{C_{2}K^{2}}{n}.
\]
Then the moment generating function of $\left\Vert Z\right\Vert $
is bounded by
\begin{align*}
\mathbb{E}\exp\left(\lambda\left\Vert Z\right\Vert \right) & \le\mathbb{E}\exp\left(2\lambda\max_{1\le i\le L}u_{i}^{\intercal}Z\right)=\mathbb{E}\left[\max_{1\le i\le L}\exp\left(2\lambda u_{i}^{\intercal}Z\right)\right]\\
 & \le\sum_{i=1}^{L}\mathbb{E}\exp\left(2\lambda u_{i}^{\intercal}Z\right)\stackrel{(a)}{\le}\sum_{i=1}^{L}\exp\left(C_{3}\frac{4\lambda^{2}K^{2}}{n}\right)\\
 & =L\exp\left(C_{4}\frac{\lambda^{2}K^{2}}{n}\right)
\end{align*}
for $\lambda>0$, where $(a)$ follows from the sub-gaussianity of
$u_{i}^{\intercal}Z$. Hence by Chernoff bound, for any $t>0$, we
have
\begin{align*}
\Pr\left\{ \left\Vert Z\right\Vert \ge t\right\}  & \le\inf_{\lambda>0}\left\{ \exp\left(-\lambda t\right)\mathbb{E}\exp\left(\lambda\left\Vert Z\right\Vert \right)\right\} \\
 & \le\inf_{\lambda>0}\left\{ L\exp\left(C_{4}\frac{\lambda^{2}K^{2}}{n}-\lambda t\right)\right\} \\
 & =L\exp\left(-\frac{nt^{2}}{4C_{4}K^{2}}\right),
\end{align*}
and the lemma follows by setting $L\exp\left(-\frac{nt^{2}}{4C_{4}K^{2}}\right)=\delta$
and solving for $t$.
\end{proof}
\begin{rem*}
In view of \prettyref{lem:orl-norm}$\prettyref{enu:cent}$, the result
above can be extended to non-centered random vectors by simply replacing
$Y$ with $Y-\mathbb{E}Y$.
\end{rem*}
A concentration inequality for sub-exponential random vectors.
\begin{lem}
\label{lem:constr-sub-exp}Suppose $Y$ is a centered random vector
in $\mathbb{R}^{p}$ such that $u^{\intercal}Y$ is sub-exponential
with Orlicz norm $\left\Vert u^{\intercal}Y\right\Vert _{\psi_{1}}\le K$
for any $u\in\mathbb{S}^{p-1}$. If $Y_{k}$ is an i.i.d. copy of
$Y$ for $k=1,\cdots,n$, then for $\delta>0$ and $n>c\log\left(L/\delta\right)$
there holds
\[
\left\Vert \frac{1}{n}\sum_{k=1}^{n}Y_{k}\right\Vert \le CK\sqrt{\frac{\log(L/\delta)}{n}}
\]
with probability at least $1-\delta$.
\end{lem}
\begin{proof}
Let $\{u_{i}\}_{i=1}^{L}$ be a $\frac{1}{2}$-net of the unit sphere
$\mathbb{S}^{p-1}\subset\mathbb{R}^{p}$ and let $Z:=\frac{1}{n}\sum_{k=1}^{n}Y_{k}$,
then by \prettyref{lem:discret-norm}, $\left\Vert Z\right\Vert \le2\max_{1\le i\le L}u_{i}^{\intercal}Z$.

Then for $0<\lambda<C_{1}n/K$, the moment generating function of
$\left\Vert Z\right\Vert $ exists and is bounded by
\begin{align*}
\mathbb{E}\exp\left(\lambda\left\Vert Z\right\Vert \right) & \le\mathbb{E}\exp\left(2\lambda\max_{1\le i\le L}u_{i}^{\intercal}Z\right)=\mathbb{E}\left[\max_{1\le i\le L}\exp\left(2\lambda u_{i}^{\intercal}Z\right)\right]\\
 & \le\sum_{i=1}^{L}\mathbb{E}\exp\left(2\lambda u_{i}^{\intercal}Z\right)=\sum_{i=1}^{L}\mathbb{E}\exp\left(\sum_{k=1}^{n}\frac{2\lambda}{n}u_{i}^{\intercal}Y_{k}\right)\\
 & \stackrel{(a)}{=}\sum_{i=1}^{L}\prod_{k=1}^{n}\mathbb{E}\exp\left(\frac{2\lambda}{n}u_{i}^{\intercal}Y_{k}\right)\stackrel{(b)}{\le}\sum_{i=1}^{L}\prod_{k=1}^{n}\exp\left(C_{2}\left(\frac{2\lambda K}{n}\right)^{2}\right)\\
 & =\sum_{i=1}^{L}\exp\left(C_{2}\frac{4\lambda^{2}K^{2}}{n}\right)=L\exp\left(C_{3}\frac{\lambda^{2}K^{2}}{n}\right),
\end{align*}
where $(a)$ follows from the independence of $Y_{k}$; $(b)$ follows
from the fact that $u_{i}^{\intercal}Y_{k}$ is sub-exponential. Then
by Chernoff bound, for any $t>0$, we have
\begin{align*}
\Pr\left\{ \left\Vert Z\right\Vert \ge t\right\}  & \le\inf\left\{ \exp\left(-\lambda t\right)\mathbb{E}\exp\left(\lambda\left\Vert Z\right\Vert \right)\mid0<\lambda<C_{1}n/K\right\} \\
 & \le\inf\left\{ L\exp\left(C_{3}\frac{\lambda^{2}K^{2}}{n}-\lambda t\right)\mid0<\lambda<C_{1}n/K\right\} \\
 & =L\exp\left(-\frac{nt^{2}}{4C_{3}K^{2}}\right),
\end{align*}
if $\frac{nt}{2C_{3}K^{2}}<C_{1}n/K$ or $t<C_{4}K$. By setting $L\exp\left(-\frac{nt^{2}}{4C_{3}K^{2}}\right)=\delta$,
we have $t=CK\sqrt{\frac{\log(L/\delta)}{n}}$, and the lemma follows
whenever $n>c\log\left(L/\delta\right)$.
\end{proof}
\begin{rem*}
In view of \prettyref{lem:orl-norm}$\prettyref{enu:cent}$, the result
above can be extended to non-centered random vectors by simply replacing
$Y$ with $Y-\mathbb{E}Y$.
\end{rem*}
\newpage

\pdfbookmark[1]{References}{sec:ref}

\end{document}